\documentclass[aihp]{imsart}

\RequirePackage{amsthm,amsmath,amsfonts,amssymb,mathtools,mathrsfs,dsfont,enumerate,stackrel,cite}
\RequirePackage[numbers]{natbib}
\RequirePackage[colorlinks,citecolor=blue,urlcolor=blue]{hyperref}
\RequirePackage{graphicx}

\startlocaldefs
\theoremstyle{plain}
\newtheorem{theorem}{Theorem}
\newtheorem{corollary}{Corollary}
\newtheorem{lemma}{Lemma}
\newtheorem{proposition}{Proposition}

\theoremstyle{remark}
\newtheorem{remark}{Remark}
\newcommand{\prob}{\mathbb{P}}
\newcommand{\Ex}{\mathbb{E}}
\newcommand{\pae}{\prob\mbox{-a.e.\ }\omega}

\newcommand{\ppas}{\mathbb{P}\times P\mbox{-a.s.}}
\newcommand{\Rl}{\mathbb{R}}
\newcommand{\PP}{\mathbb{P}\times P}
\newcommand{\EE}{\mathbb{E}\times E}
\newcommand{\X}{\{X_t\}_{t\in\Rl_{\ge 0}}}

\endlocaldefs

\begin{document}

\begin{frontmatter}

\title{Large fluctuations and transport properties of the L\'evy-Lorentz gas}
\runtitle{Large fluctuations and transport properties of the L\'evy-Lorentz gas}

\begin{aug}
\author{\inits{M.}\fnms{Marco} \snm{Zamparo}\ead[label=e1]{marco.zamparo@uniba.it}}
\address{Dipartimento di Fisica, Universit\`a degli Studi di Bari and INFN, Sezione di Bari,
    via Amendola 173, 70126 Bari, Italy \printead{e1}}
\end{aug}

\begin{abstract}
The L\'evy-Lorentz gas describes the motion of a particle on the real
line in the presence of a random array of scattering points, whose
distances between neighboring points are heavy-tailed i.i.d.\ random
variables with finite mean. The motion is a continuous-time,
constant-speed interpolation of the simple symmetric random walk on
the marked points. In this paper we study the large fluctuations of
the continuous-time process and the resulting transport properties of
the model, both annealed and quenched, confirming and extending
previous work by physicists that pertain to the annealed
framework. Specifically, focusing on the particle displacement, and
under the assumption that the tail distribution of the interdistances
between scatterers is regularly varying at infinity, we prove a
precise large deviation principle for the annealed fluctuations and
present the asymptotics of annealed moments, demonstrating annealed
superdiffusion.  Then, we provide an upper large deviation estimate
for the quenched fluctuations and the asymptotics of quenched moments,
showing that the asymptotic diffusive regime conditional on a typical
arrangement of the scatterers is normal diffusion, and not
superdiffusion. Although the L\'evy-Lorentz gas seems to be accepted
as a model for anomalous diffusion, our findings suggest that
superdiffusion is a transient behavior which develops into normal
diffusion on long timescales, and raise a new question about how the
transition from the quenched normal diffusion to the annealed
superdiffusion occurs.
\end{abstract}

\begin{abstract}[language=french]
Le gaz de L\'evy-Lorentz mod\'elise le d\'eplacement d'une particule
sur l'axe des r\'eels en pr\'esence d'obstacles distribu\'es de telle
fa\c{c}on que les distances les uns avec les autres sont des variables
aléatoires i.i.d. \`a queue lourde et moyenne finie.  La dynamique est
donn\'ee par l'interpolation, en temps continu et vitesse fixe, de la
marche al\'eatoire sym\'etrique sur les obstacles.  Cet article
\'etudie les grandes fluctuations du processus en temps continu et les
propri\'et\'es de transport du mod\`ele sous-jacent. Les r\'esultats
obtenus dans les cas de d\'esordre ``annealed'' et ``quenched''
confirment et g\'en\'eralisent des r\'esultats precedents issus de la
physique dans le cas ``annealed''. En particulier, sous l'hypoth\`ese
que la queue de la loi des distances inter-obstacles est \`a variation
r\'eguli\`ere \`a l'infini, nous prouvons dans le cas ``annealed'' un
principe de grandes d\'eviations et obtenons l'expression asymptotique
des moments, qui montrent l'existence d'un regime de sur-diffusion.
Dans le cas quenched, nous obtenons l'expression asymptotique des
moments et une borne sup\'erieure sur de grandes deviations des
fluctuations. Cela nous permet de montrer, pour des configurations
d'obstacles typiques, que le r\'egime asymptotique de diffusion est la
diffusion normale, et non la sur-diffusion. Bien que le gaz de
L\'evy-Lorentz soit en g\'en\'eral utilis\'e pour mod\'eliser la
diffusion anormale, notre r\'esultat sugg\`ere que le regime
sur-diffusif est seulement transitoire. Cela soul\`eve \'egalement la
question de la nature de la transition, entre diffusion normale
``quenched'' et sur-diffusion ``annealed''.
\end{abstract}



\end{frontmatter}

\section{Introduction}

The Lorentz gas model, introduced by Lorentz in 1905 to describe the
diffusion of conduction electrons in metals, has illuminated kinetic
theory and transport theory over more than half a century
\cite{Book_LG,Beijeren1}.  In this model a number of fixed scatterers
are distributed at random over a portion of space and one point
particle that has velocity of constant magnitude is reflected or
possibly transmitted when hitting a scatterer.  Deterministic Lorentz
gas models \cite{Book_LG} assume that the velocity of the particle
after a collision with a scatterer is uniquely defined, whereas this
velocity follows some stochastic rule in stochastic Lorentz gases
\cite{Beijeren1,Beijeren2,Grassberger}. Building up on the growing
interest in anomalous diffusion that recent years have witnessed
\cite{Klages1}, in 2000 Barkai, Fleurov, and Klafter \cite{Barkai2}
proposed a one-dimensional stochastic Lorentz gas with heavy-tailed
random interdistances between scatterers for the study of
superdiffusion in porous media.  Superdiffusion is a form of diffusion
at a faster speed than square root of time due to occasional very long
ballistic flights of the particle.  They named \textit{L\'evy-Lorentz
  gas} such a model.  Although some heavy-tailed models, represented
by homogeneous L\'evy walks \cite{Zaburdaev1}, had already been
introduced at that time to describe superdiffusion, the interest in
the L\'evy-Lorentz gas stemmed from the fact that the long ballistic
flights performed by the particle are ascribed to the nature of the
medium, and not to a special law governing the walker's decision as in
L\'evy walks. In fact, in 2008 Barthelemy, Bertolotti, and Wiersma
\cite{Barthelemy1} showed that it is possible to engineer optical
materials, called \textit{L\'evy glasses}, in which light waves
perform a motion between targets that are spaced according to
polynomial-tailed interdistances.

The rigorous mathematical study of the L\'evy-Lorentz gas was
initiated in 2016 by Bianchi, Cristadoro, Lenci, and Ligab\`o
\cite{Bianchi1}, who proved a central limit theorem for both the
quenched and the annealed framework. The present paper aims to get to
the heart of this study by investigating the large fluctuations and
the transport properties of the model, the latter being of main
importance for physics. In this section we define the L\'evy-Lorentz
gas model and review previous findings, both rigorous and
heuristic. The main contributions of the paper are presented and
discussed in the next section.

\subsection{The L\'evy-Lorentz gas model}

A one-dimensional stochastic Lorentz gas can be sketched, in a
nutshell, as follows.  Let $\omega:=\{\omega_r\}_{r\in\mathbb{Z}}$ be
the positions in ascending order of certain scatterers placed on the
real line. A particle moving on the line with unit speed from the
initial position $\omega_0$ collides with these scatterers, $V_n$
being its velocity direction before the $n$th collision, and is
reflected or transmitted with prescribed probability.

Following the formulation of the L\'evy-Lorentz gas given in
\cite{Barkai2}, we assume that $V_1,V_2,\ldots$ are i.i.d.\ random
variables, on a probability space $(\mathcal{S},\mathfrak{S},P)$, that
take values in $\mathbb{Z}_2:=\{-1,+1\}$ with equal probability.  The
$n$th collision occurs at the scattering point labeled by
\begin{equation*}
 S_n:=\sum_{k=1}^nV_k ~~~~~ (S_0:=0)
\end{equation*}
and at the time
\begin{equation*}
T_n^\omega:=\sum_{k=1}^n\big|\omega_{S_k}-\omega_{S_{k-1}}\big| ~~~~~ (T_0^\omega:=0).
\end{equation*}
The process $\{S_n\}_{n\in\mathbb{Z}_{\ge 0}}$ is the simple symmetric
random walk. The motion of the particle is defined by piecewise linear
interpolation of the points
$\{(T_n^\omega,\omega_{S_n})\}_{n\in\mathbb{Z}_{\ge 0}}$, and its
position at time $t\in[T_{n-1}^\omega,T_n^\omega)$ with $n\ge 1$ reads
\begin{equation}
X_t^\omega:=\omega_{S_{n-1}}+V_n(t-T_{n-1}^\omega).
\label{def_position_X}
\end{equation}
Unit speed implies $|X_t^\omega|\le t$, which can be formally seen by
observing that
$|\omega_{S_{n-1}}|=|\sum_{k=1}^{n-1}(\omega_{S_k}-\omega_{S_{k-1}})|\le\sum_{k=1}^{n-1}|\omega_{S_k}-\omega_{S_{k-1}}|=T_{n-1}^\omega$.
Regarding the scatterers arrangements, we suppose that
$\omega:=\{\omega_r\}_{r\in\mathbb{Z}}$ constitutes a random
environment sampled from a probability space
$(\Omega,\mathscr{F},\prob)$, with $\omega_0:=0$ for
definiteness. According to \cite{Barkai2}, we make the hypothesis that
the distances $\Delta_r$ between consecutive scatterers, which
associate each $\omega$ with $\Delta_r^\omega:=\omega_r-\omega_{r-1}$,
form a bilateral sequence of i.i.d.\ positive random variables.  In
\cite{Barkai2} the tail probability $\prob[\Delta_1>\cdot\,]$ of the
variable $\Delta_1$ was assumed to decay in a polynomial way. In this
paper we shall focus on the slightly more general class of regularly
varying tail distributions, which will be introduced in Section
\ref{main_results}.

Let $(\Omega\times\mathcal{S},\mathscr{F}\times\mathfrak{S},\PP)$ be
the product probability space associated with the random environment
and the particle's velocities and let $\Ex$, $E$, and $\Ex\times E$ be
the expectations with respect to $\prob$, $P$, and $\PP$. The
sequences $\{V_n\}_{n\in\mathbb{Z}_{\ge 1}}$ and
$\{\Delta_r\}_{r\in\mathbb{Z}}$ naturally extend to two independent
sequences over the product probability space by composing with the
projection onto $\mathcal{S}$ and $\Omega$, respectively.  With an
abuse of notation in favor of simplicity, which however should not
cast doubt, in the present paper we use the same symbol for both a
variable defined on $(\mathcal{S},\mathfrak{S},P)$, or
$(\Omega,\mathscr{F},\prob)$, and its natural extension over the
product probability space.  Besides, we think of the collision time
$T_n^\omega$ and the particle position $X_t^\omega$ as the
$\omega$-section, for a given $\omega\in\Omega$, of random variables
$T_n$ and $X_t$ on
$(\Omega\times\mathcal{S},\mathscr{F}\times\mathfrak{S},\PP)$.

\begin{remark}
  The particle displacement $X_t$ is $\mathbb{P}\times
  P$-a.s.\ well-defined for all $t>0$, that is
  $\lim_{n\uparrow+\infty}T_n=+\infty$ $\ppas$.  In fact, the
  hypothesis that $\Delta_1$ is positive is tantamount to say that
  $\prob[\Delta_1>\delta]>0$ for some real number $\delta>0$.  The
  second Borel-Cantelli lemma ensures that there are infinitely many
  $\Delta_r^\omega$'s above $\delta$ for $\pae$, so that
  $\lim_{r\uparrow+\infty}|\omega_{-r}|\vee|\omega_r|=+\infty$ for
  $\pae$.  Then, the law of the iterated logarithm for
  $\{S_n\}_{n\in\mathbb{Z}_{\ge 0}}$ gives
  $\lim_{n\uparrow+\infty}T_n^\omega=+\infty$ $P$-a.s.\ for $\pae$
  since $\{T_n^\omega\}_{n\in\mathbb{Z}_{\ge 0}}$ is an increasing
  sequence such that $T_n^\omega\ge |\omega_{S_n}|$ for all
  $n$. Fubini's theorem allows us to conclude that
  $\lim_{n\uparrow+\infty}T_n=+\infty$ $\ppas$.
\end{remark}

\subsection{Previous results}
\label{model_old}

The study of random walks in random environments began in the early
70's \cite{Zeitouni1}, and since that time a great amount of work has
been done for random walks in heavy-tailed random environments with
different purposes. In the \textit{random walk in a random scenery}
with a heavy-tailed scenery
%\cite{Asselah_RWRS1,Bolthausen_RWRS1,Castell_RWRS1,Castell_RWRS2,Kesten,Pene_RWRS1}
\cite{Kesten,Pene_RWRS1}
a walker explores a given graph and cumulates a heavy-tailed reward
associated with its vertices. This model is a famous example of
probabilistic model with long time dependence.  In the \textit{random
  conductance model} with heavy-tailed random conductances
\cite{Biskup_RC2,Kawazu1}
%\cite{Barlow_RC1,Barlow_RC2,Berger_RC1,Biskup_RC2,Kawazu1}
a walker moves on a given graph with heavy-tailed jumping rates
associated with undirected edges. This model explains subdiffusion,
i.e.\ diffusion at a slower speed than square root.  In the
\textit{random trap model} with heavy-tailed trapping landscape
\cite{BenArous1,Fontes1}
%\cite{BenArous1,BenArous_RT1,Cerny_RT1,Fontes1,Mourrat_RT1,Zindy_RT1}
a walker moves on a given graph with heavy-tailed jumping rates
determined by the depth of the traps associated with vertices. This
model accounts for subdiffusion and aging via correlation functions.
Parallel to these, there are \textit{random walks on point processes}
with heavy-tailed spacing
\cite{Berger1,Caputo1,Caputo2,Magdziarz1,Rousselle1,Stivanello} where
a walker jumps between points with heavy-tailed
interdistances. Criteria for recurrence and transience have been
established for these models, as well as some limit theorems for the
suitably rescaled process, including the law of large numbers and the
central limit theorem. The L\'evy-Lorentz gas is naturally connected
with the latter models. In fact, if the focus is on the sequence
$\{\omega_{S_n}\}_{n\in\mathbb{Z}_{\ge 0}}$ of the scatterers that the
particle reaches, then the problem falls under the umbrella of
\textit{random walks on point processes}. However, if the focus is on
the continuous motion $X_t$ of the particle in the physical time $t$,
as ours is, then the problem is completely different because events
must be contextualized into a time frame by involving collision times.

The study of the L\'evy-Lorentz gas in the physical time has been
recently initiated by Bianchi, Cristadoro, Lenci, and Ligab\`o
\cite{Bianchi1}, who addressed the typical fluctuations of $X_t$ as
$t$ goes to infinity under the hypothesis
$\mu:=\Ex[\Delta_1]<+\infty$. The case $\mu=+\infty$ has been later
considered by Bianchi, Lenci, and P\`ene \cite{Bianchi2}. A first
important result of \cite{Bianchi1} is the following strong law of
large numbers for the collision times.
\begin{theorem}[Bianchi \textit{et al.}\ \cite{Bianchi1}]
\label{SLLN_T}
The following limit holds $\ppas$:
\begin{equation*}
\lim_{n\uparrow+\infty}\frac{T_n}{n}=\mu.
\end{equation*}
\end{theorem}
Theorem \ref{SLLN_T} was used by the authors of \cite{Bianchi1} to
prove the quenched central limit theorem reported below, which
characterizes the typical fluctuations of the process $\X$ when
$\mu<+\infty$.  We point out that two different perspectives can be
adopted when studying a motion in a random environment: that of the
\textit{quenched process}, where the dynamics conditional on a typical
realization of the environment is analyzed, and that of the
\textit{annealed process}, where the interest is in the effect of
averaging over the environments.
\begin{theorem}[Bianchi \textit{et al.}\ \cite{Bianchi1}]
\label{CLTq}
The following conclusion holds for $\pae$ if $\mu<+\infty$: for all
$x\in\Rl$
  \begin{equation*}
  \lim_{t\uparrow+\infty}P\bigg[\frac{X_t^\omega}{\sqrt{\mu t}}\le x\bigg]=
  \frac{1}{\sqrt{2\pi}}\int_{-\infty}^x e^{-\frac{1}{2}\xi^2}d\xi.
  \end{equation*}
\end{theorem}
Theorem \ref{CLTq} immediately gives an annealed central limit theorem
thanks to Fubini's theorem and the dominated convergence theorem.
\begin{corollary}[Bianchi \textit{et al.}\ \cite{Bianchi1}]
  \label{CLTa}
  The following conclusion holds for all $x\in\Rl$ provided that
  $\mu<+\infty$:
  \begin{equation*}
   \lim_{t\uparrow+\infty}\,\PP\bigg[\frac{X_t}{\sqrt{\mu t}}\le x\bigg]=
   \lim_{t\uparrow+\infty}\int_{\Omega}P\bigg[\frac{X_t^\omega}{\sqrt{\mu t}}\le x\bigg]\,\prob[d\omega]
   =\frac{1}{\sqrt{2\pi}}\int_{-\infty}^xe^{-\frac{1}{2}\xi^2}d\xi.
  \end{equation*}
\end{corollary}
No hypothesis about the tail probability $\prob[\Delta_1>\cdot\,]$ of
$\Delta_1$ other than $\mu<+\infty$ is needed in order to establish
Theorems \ref{SLLN_T} and \ref{CLTq}. On the contrary, the case
$\mu=+\infty$ requires to specify $\prob[\Delta_1>\cdot\,]$ somehow.
In \cite{Bianchi2}, the authors exploited techniques from
\textit{random walks in random sceneries} to investigate the typical
fluctuations of $\X$ when $\Delta_1$ belongs to the basin of normal
attraction of an $\alpha$-stable random variable $Z$ with index
$\alpha\in(0,1)$. We recall that normal attraction means that
$n^{-\frac{1}{\alpha}}\sum_{r=1}^n\Delta_r$ converges in distribution
to $Z$ as $n$ is sent to infinity \cite{Feller_normal}.  One has
$\mu=+\infty$ in this case. The main result of \cite{Bianchi2} is that
the finite dimensional distributions of the process
$\{\lambda^{-\frac{1}{\alpha+1}}X_{t\lambda}\}_{t\in\Rl_{\ge 0}}$
converge, as the scaling parameter $\lambda$ goes to infinity, to
those of a stochastic process related to the Kesten-Spitzer process
\cite{Kesten}.  The present paper deals with the case $\mu<+\infty$
and makes use of Theorems \ref{SLLN_T} and \ref{CLTq} and Corollary
\ref{CLTa} for some essential steps.

\subsection{Transport properties and heuristics}
\label{transport}

The identification of anomalous diffusion in physics generally passes
through inspection of the time scaling of the mean-square displacement
\cite{Klages1}.  Regarding the possibility of computing the
mean-square displacement of the L\'evy-Lorentz gas, it must be said
that the findings of \cite{Bianchi1} and \cite{Bianchi2} about the
typical fluctuations do not suffice to determine the moments of the
particle position $\X$, since they could be, and are, affected by the
large fluctuations.  The annealed mean square displacement of the
L\'evy-Lorentz gas defined by the law
$\prob\big[\Delta_1>x\big]:=(x\vee 1)^{-\alpha}$ for every $x>0$ was
first investigated numerically for $\alpha=3/2$ by Barkai, Fleurov,
and Klafter \cite{Barkai2}. Subsequently, by resorting to heuristic
arguments based on a similar electric problem \cite{Beenakker1},
Burioni, Caniparoli, and Vezzani \cite{Burioni1} suggested that the
annealed moments of $X_t$ behave for large times $t$ as
\begin{equation}
  \EE\big[|X_t|^q\big]\sim C
  \begin{cases}
    t^{\frac{q}{\alpha+1}} & \mbox{if $\alpha<1$ and $q<\alpha$},\\
    t^{q-\frac{\alpha^2}{\alpha+1}} & \mbox{if $\alpha<1$ and $q>\alpha$},\\
    t^{\frac{q}{2}} & \mbox{if $\alpha>1$ and $q<2\alpha-1$}, \\
    t^{q-\alpha+\frac{1}{2}} & \mbox{if $\alpha>1$ and $q>2\alpha-1$},
  \end{cases}
  \label{scaling}
\end{equation}
where $C$ is some non-specified constant that possibly depends on $q$
and $\alpha$. The symbol $\sim$ denotes asymptotic equivalence: two
functions $f$ and $g$ over the positive real axis are asymptotic
equivalent, written symbolically as $f(t)\sim g(t)$, if
$\lim_{t\uparrow+\infty}f(t)/g(t)=1$. According to (\ref{scaling})
with $q=2$, the mean-square displacement of the L\'evy-Lorentz gas
should exhibit normal scaling with exponent $1$ if $\alpha>3/2$ and
superdiffusive scaling with exponent large than $1$ if $\alpha<3/2$.
Very recently, the asymptotic behavior (\ref{scaling}) has been
corroborated through simplified models, both deterministic dynamical
systems \cite{Rondoni1,Rondoni2,Rondoni3} and random walks
\cite{Artuso1,Artuso2}, which were argued to approximate the
L\'evy-Lorentz gas to some extent. Vezzani, Barkai, and Burioni
\cite{Vezzani1,Vezzani2} have also provided a heuristic estimate of
the far tail of the annealed displacement distribution by appealing to
the \textit{principle of a single big jump} \cite{Foss} from the
heavy-tailed world.  The present work was stimulated by the quest for
providing rigorous insight into the large fluctuations of the
L\'evy-Lorentz gas and for establishing (\ref{scaling}) on a solid
mathematical ground.  Superdiffusion, if any, is connected with the
way a large fluctuation of the process $\X$ occurs.

\section{Main results of the paper and discussion}
\label{main_results}

In this paper we investigate the large fluctuations of the process
$\X$ and its moments under the assumption
$\mu:=\Ex[\Delta_1]<+\infty$. Our basic hypothesis to deal with the
annealed problem is that the tail probability
$\prob[\Delta_1>\cdot\,]$ of $\Delta_1$ is regularly varying at
infinity with index $-\alpha$, meaning that for every $x>0$
\begin{equation*}
\prob\big[\Delta_1>x\big]=\frac{\ell(x)}{x^\alpha}
 \end{equation*}
with a \textit{slowly varying function} at infinity $\ell$. A slowly
varying function at infinity $\ell$ is a positive measurable function
on some neighborhood of infinity that satisfies the scale-invariance
property $\lim_{x\uparrow+\infty}\ell(\beta x)/\ell(x)=1$ for any
positive real number $\beta$.  Calculations that involve slowly
varying functions are made possible by Karamata theory and we refer to
\cite{RV} for details.  Two of their elementary properties are
$\lim_{x\uparrow+\infty}x^\gamma\ell(x)=+\infty$ and
$\lim_{x\uparrow+\infty}x^{-\gamma}\ell(x)=0$ for all $\gamma>0$ (see
\cite{RV}, Proposition 1.3.6). It follows that a necessary condition
for $\mu:=\Ex[\Delta_1]=\int_0^{+\infty}\prob[\Delta_1>x]\,dx<+\infty$
is $\alpha\ge 1$, whereas a sufficient condition is $\alpha>1$. Thus,
the assumption $\mu<+\infty$ requires $\alpha\ge 1$ and a further
hypothesis for integrability on $\ell$ when $\alpha=1$.  We present
annealed results first, and the corresponding quenched results later.

\subsection{Annealed fluctuations and annealed moments}
\label{annealed_results}

Our study begins with the large fluctuations of the particle
displacement averaged over the environments. The L\'evy-Lorentz gas
model fulfills a pretty obvious left-right symmetry, which is
described by the following lemma whose formal proof is omitted because
straightforward .  This symmetry allows us to restrict to positive
fluctuations.
\begin{lemma}
  \label{simmetria}
$\{-X_t\}_{t\in\Rl_{\ge 0}}$ and $\X$ share the same finite-dimensional distributions under the law $\PP$.
\end{lemma}

In order to characterize the positive large fluctuations of
$\X$, for each $\alpha\ge 1$ we introduce the
measurable function $f_\alpha$ that maps any $x\in(0,1)$ in
\begin{equation}
\label{deffa}
f_\alpha(x):=\sum_{l=0}^{+\infty}\mathds{1}_{\big\{l<\frac{1-x}{2x}\big\}}\Bigg[\bigg(\frac{2l+2}{1+x}\bigg)^\alpha
    -\bigg(\frac{2l}{1-x}\bigg)^\alpha\Bigg].
\end{equation}
Elementary properties of $f_\alpha$ that allow to appreciate the main
results of the section are stated by the following lemma, which is
proved in Appendix \ref{proof:prop_funalpha}.
\begin{lemma}
  \label{prop_funalpha}
  $0<f_\alpha(x)\le x^{-\alpha}$ for all $x\in(0,1)$ and $\lim_{x\downarrow 0}x^\alpha f_\alpha(x)=1/(\alpha+1)$.
\end{lemma}

The function $f_\alpha$ is put into context by the following theorem,
which provides a sharp large deviation estimate for
$\X$ within the annealed framework and
represents our first main result. To understand its content, recall
that the positive values of $X_t$ does not exceed $t$, with the
consequence that $\PP[X_t>xt]$ is non-trivial only for $x\in(0,1]$.
\begin{theorem}
\label{atail} 
Assume that $\mu<+\infty$ and that $\prob[\Delta_1>\cdot\,]$ is
regularly varying at infinity with index $-\alpha\le-1$. For every
$x\in(0,1]$ set
\begin{equation*}
  F(x):=\frac{1}{\sqrt{2\pi\mu}}\int_x^1 f_\alpha\bigg(\frac{x}{\xi}\bigg)\frac{d\xi}{\xi^\alpha\sqrt{1-\xi}}.
\end{equation*}
Then, the following conclusion holds for any number $\delta\in(0,1]$:
\begin{equation*}
 \adjustlimits\lim_{t\uparrow+\infty}\sup_{~x\in[\delta,1]~}\!\Bigg\{\bigg|\frac{\PP[X_t>xt]}{\sqrt{t}\,\prob[\Delta_1>t]}-F(x)\bigg|\Bigg\}=0.
\end{equation*}
\end{theorem}

While Corollary \ref{CLTa} describes the typical fluctuations of $X_t$
at large $t$, Theorem \ref{atail} characterizes the large fluctuations
with order of magnitude $t$.  Since, basically, $\X$ is a process of
accumulation of interdistances between scattering points, the
explanation of Theorem \ref{atail} comes from the theory of large
deviations for sums of i.i.d.\ random variables. Generally speaking,
the only significant mechanism to produce a large fluctuation of a sum
of i.i.d.\ random variables is that either many small deviations all
in the same direction occur or a single summand takes a very large
value \cite{SUMRV}. The latter is known with the folklore name of
\textit{principle of a single big jump} \cite{Foss}.  One event or the
other depends on the satisfiability of the Cram\'er's condition, which
is the property of the moment generating function of the summands to
be finite in an open neighborhood of the origin. When the Cram\'er
condition is satisfied, the \textit{Gibbs conditioning principle}
\cite{Dembo_Gibbs} states that, subject to a large deviation of the
sum, the summands become i.i.d. in the limit of infinite summands, but
their marginal distribution is modified in such a way that the
behavior imposed on the sum becomes typical. In this case, the
probability of a large fluctuation of the sum is suppressed
exponentially and its sharp asymptotics is described by the Cram\'er
theorem and its refinements \cite{Hoglund}. For the stochastic Lorentz
gas under consideration, these arguments lead to the conclusion that
if the interdistances between scatterers satisfied the Cram\'er's
condition, then there would be no long gaps to make long ballistic
flights, and a large fluctuation of the displacement would be realized
by many jumps all in the same direction. This is the trait of normal
diffusion.

The situation is the opposite for subexponential summands that violate
the Cram\'er's condition, as it is the case with the L\'evy-Lorentz
gas.  It has been established for a large class of subexponential
random variables that the conditional distribution of the summands
subject to a large deviation of their sum converges to a product of
independent copies of the original distribution, except for one
variable that realizes the large deviation event by taking a very
large value \cite{Bigjump2}. This result provides a detailed picture
of the \textit{principle of a single big jump}.  In this case, the
probability of a large fluctuation of the sum displays a
subexponential decay and its sharp asymptotics has been investigated
for several types of summands
\cite{Bigjump1,Mikosch_LDP,Nagaev,Ng1,Rozovski}. For regularly varying
summands at infinity with finite mean we refer to Theorem 3.1 of
\cite{Ng1}, which gives the following \textit{precise large deviation
  principle} for the interdistances between scatterers
$\Delta_1,\Delta_2,\ldots$ under the assumption $\mu<+\infty$: for
each $\gamma>0$
\begin{equation}
  \label{LDPheavy}
  \adjustlimits\lim_{s\uparrow+\infty}\sup_{~z\ge \gamma s~}\!\Bigg\{\bigg|\frac{\prob\big[\sum_{r=1}^s\Delta_r-\mu s>z\big]}{s\,\prob[\Delta_1>z]}-1\bigg|\Bigg\}=0.
\end{equation}
Formula (\ref{LDPheavy}) creates a logical bridge between the
\textit{principle of a single big jump} and Theorem \ref{atail}.  In
fact, the displacement $X_t$ cumulates the lengths of the different
edges that the particle travels by time $t$, the number of which is of
the order of magnitude of $\sqrt{t}$ since the process of scatterer
exploration is the simple symmetric random walk.  Accordingly, the
\textit{principle of a single big jump} suggests that if the
L\'evy-Lorentz gas is averaged over the environments, then the
particle undergoes a large displacement at time $t$ through long
inertial flights over a single large gap between scatterers, which is
chosen from a number of edges proportional to $\sqrt{t}$. This
picture, which is the hallmark of annealed superdiffusion, perfectly
matches with the scaling factor $\sqrt{t}\,\prob[\Delta_1>t]$ in
Theorem \ref{atail} through formula (\ref{LDPheavy}). What exactly
happens is that once the particle has reached the longest edge it
starts bouncing back and forth in this gap, being rapidly scattered
back by many close collisions as it leaves the longest gap.
Consistently, it is very likely that at any current time the particle
is in this gap. In other words, the longest edge is expected to be the
current edge.  The possibility to travel several times over the
largest gap is responsible for the complex structure of the function
$F$ that enters Theorem \ref{atail}.

Theorem \ref{atail} is proved in Section \ref{proof:atail}, after that
some preliminary results about the simple symmetric random walk and
the collision times are presented in Section \ref{prel}.  The tail
estimate stated by Theorem \ref{atail} has been previously sketched
out by Vezzani, Barkai, and Burioni \cite{Vezzani1,Vezzani2} by a
heuristic use of the \textit{principle of a single big jump}.  In
Section \ref{proof:atail} we shall provide solid mathematical
justifications to the use of this principle, through formula
(\ref{LDPheavy}), and to the fact that the longest edge is the current
edge. Such justifications are completely missed in papers
\cite{Vezzani1} and \cite{Vezzani2} by physicists

The next theorem is our second main result, which identifies the
asymptotic behavior of the annealed moments of $X_t$ of generic order
$q$ as $t$ goes to infinity. In order to appreciate the content of the
theorem, let us observe that $\int_0^1 x^{q-1}f_\alpha(x)\,dx$ exists
finite if and only if $q>\alpha$ by Lemma \ref{prop_funalpha}, and
this is the case when $q=2\alpha-1$ and $\alpha>1$ or $q>2\alpha-1$
and $\alpha\ge 1$. Here the symbol $\sim$ denotes asymptotic
equivalence as defined in Section \ref{transport} and $\Gamma$ is the
Euler gamma function. The Euler gamma function allows one to express
the $q$-order moment of a Gaussian random variable with mean $0$ and
variance $\mu$ as $\sqrt{(2\mu)^q/\pi}\,\Gamma(\frac{q+1}{2})$.
\begin{theorem}
  \label{annealed_moment}
  Assume that $\mu<+\infty$ and that $\prob[\Delta_1>\cdot\,]$ is
  regularly varying at infinity with index $-\alpha\le-1$. For every
  $q>0$ set
\begin{equation*}
 g_q:=\sqrt{\frac{(2\mu)^q}{\pi}}\,\Gamma\bigg(\frac{q+1}{2}\bigg)
\end{equation*}
  and
\begin{equation*}
  d_q:=\sqrt{\frac{2}{\mu}}\,\frac{\Gamma(q-\alpha+1)}{\Gamma(q-\alpha+3/2)}\int_0^1 qx^{q-1}f_\alpha(x)\,dx.
\end{equation*}
Then, the following asymptotic equivalence holds for all $q>0$:
  \begin{equation*}
  \EE\big[|X_t|^q\big]\sim
  \begin{cases}
    g_qt^{\frac{q}{2}} & \mbox{if $q<2\alpha-1$ or $q=2\alpha-1$ and $\alpha=1$},\\
    g_qt^{\frac{q}{2}}+d_qt^{q+\frac{1}{2}}\prob[\Delta_1>t] & \mbox{if $q=2\alpha-1$ and $\alpha>1$},\\
    d_qt^{q+\frac{1}{2}}\prob[\Delta_1>t] & \mbox{if $q>2\alpha-1$}.
  \end{cases}
  \end{equation*}
\end{theorem}
Theorem \ref{annealed_moment} is proved in Section \ref{proof_amoment}
and confirms the findings (\ref{scaling}) of Burioni, Caniparoli, and
Vezzani \cite{Burioni1} for $\alpha>1$. The proof combine the annealed
central limit theorem stated by Corollary \ref{CLTa}, to describe the
typical fluctuations of $X_t$ at the spatial scale $\sqrt{t}$, with
the precise large deviation principle of Theorem \ref{atail}, to deal
with the large fluctuations of $X_t$ at the spatial scale
$t$. However, Theorem \ref{annealed_moment} is not a simple
consequence of Corollary \ref{CLTa} and Theorem \ref{atail} because
additional work is needed to establish that the contribution of the
fluctuations that are in between these two regimes is negligible. This
additional work makes use of Rosenthal's inequalities
\cite{Rosenthal1} to control the moments of a sum of i.i.d.\ random
variables.

Theorem \ref{annealed_moment} provides in particular the mean-square
displacement, i.e.\ the variance of the particle displacement. It is
clear that $\EE[X_t]=0$ for all $t>0$ due to the symmetry of the model
stated by Lemma \ref{simmetria}, so that the variance of $X_t$ is the
second moment $\EE[X_t^2]$ described by the following corollary of
Theorem \ref{annealed_moment}. Let us point out that $g_2=\mu$. A
simple calculation shows that $d_2=\sigma$ with
\begin{equation*}
  \sigma:=\frac{2^{3/2}}{\sqrt{\mu}\,\Gamma(7/2-\alpha)}\begin{cases}
  \sum_{l=0}^{+\infty}\Big[2-(2l+2)\ln(2l+2)+2\ln(2l+1)+(2l)\ln(2l)\Big] & \mbox{if }\alpha=1,\\
  \sum_{l=0}^{+\infty}\Gamma(1-\alpha)\Big[(2l+2)^\alpha-2\alpha(2l+1)^{\alpha-1}-(2l)^\alpha\Big] & \mbox{if }1<\alpha\le 3/2.
  \end{cases}
\end{equation*}
\begin{corollary}
  \label{coll_diff}
  The following asymptotic equivalence holds within the setting of
  Theorem \ref{annealed_moment}:
  \begin{equation*}
  \EE\big[X_t^2\big]\sim
  \begin{cases}
    \sigma t^{\frac{5}{2}}\prob[\Delta_1>t] & \mbox{if $1\le\alpha<3/2$},\\
    \mu t+\sigma t^{\frac{5}{2}}\prob[\Delta_1>t] & \mbox{if $\alpha=3/2$},\\
    \mu t & \mbox{if $\alpha>3/2$}.
  \end{cases}
  \end{equation*}
\end{corollary}

Corollary \ref{coll_diff} shows that, within the annealed framework,
the L\'evy-Lorentz gas with $\mu<+\infty$ and
$\prob[\Delta_1>\cdot\,]$ regularly varying at infinity with index
$-\alpha\le-1$ displays superdiffusive scaling of the mean-square
displacement with exponent $5/2-\alpha$ for $\alpha\in[1,3/2)$, and
  normal scaling for $\alpha\ge 3/2$. The diffusion coefficient in the
  two regimes is $\sigma$ and $\mu$, respectively.

\subsection{Quenched fluctuations and quenched moments}
\label{quenched_results}

The quenched framework is of main importance because the properties of
the system conditional on a typical realization of the environment are
directly observable in real experiments. However, to the best of our
knowledge, the quenched transport properties of the L\'evy-Lorentz gas
have never been investigated before, neither by physicists nor by
mathematicians. In this section we discuss the large fluctuations of
the process $\{X_t^\omega\}_{t\in\Rl_{\ge 0}}$ for a typical
environment $\omega$, and we determine the asymptotic behavior of its
moments.  No other hypothesis besides $\mu<+\infty$ is made here about
the probability distribution of $\Delta_1$.

An upper bound to the quenched probability of a large fluctuation of
the particle displacement is established by the following theorem,
which constitutes the third main result of the paper. The theorem
reveals that such probability decays at least as fast as a stretched
exponential with stretching exponent one-half, as a result of the
compromise between a large number of collisions and a large
fluctuation of the simple symmetric random walk. The proof is reported
in Section \ref{qtail_proof}.
\begin{theorem}
\label{qtail}
  Assume that $\mu<+\infty$. There exists a real number
  $\kappa>0$ such that the following property holds for $\pae$: for
  all $x\in(0,1]$
\begin{equation*}
  \limsup_{t\uparrow+\infty}\frac{1}{\sqrt{xt}}\ln P\big[|X_t^\omega|>xt\big]\le -\kappa.
\end{equation*}
\end{theorem}

\begin{remark}
  \label{remark_Q}
  Physicists \cite{Barkai2,Burioni1,Vezzani1} generally assume that
  $\prob[\Delta_1\ge\delta_o]=1$ for some real number $\delta_o>0$
  when dealing with the L\'evy-Lorentz gas.  This assumption allows us
  to improve the stretching exponent in Theorem \ref{qtail}, leading
  to a pure exponential decay.  In fact, it will be pointed out in
  Section \ref{qtail_proof} that under the assumption
  $\prob[\Delta_1\ge\delta_o]=1$ with $\delta_o>0$ there exists
  $\kappa>0$ such that the following property holds for $\pae$: for
  all $x\in(0,1]$
\begin{equation*}
  \limsup_{t\uparrow+\infty}\frac{1}{x^2t}\ln P\big[|X_t^\omega|>xt\big]\le -\kappa.
\end{equation*}
\end{remark}

Theorem \ref{qtail} gives only a rough bound, which needs some
hypothesis on the interdistances between scattering points at the
small spatial scales to be improved, as in Remark
\ref{remark_Q}. However, it is already sufficient to state that the
annealed approach and the quenched approach to the L\'evy-Lorentz gas
give completely different answers, and, in particular, the motion
conditional on a typical realization of the environment is no longer
superdiffusive. In fact, Theorem \ref{qtail} and Remark \ref{remark_Q}
are not compatible with a picture where a large fluctuation of the
displacement is due to a large fluctuation of one edge, to which a
polynomial decay of probabilities would correspond. They are instead
explained by conjecturing many jumps all in the same direction, as if, to
make a comparison with an annealed situation described in the
previous section, the lengths of the edges satisfied the Cram\'er's
condition. As a matter of fact, in a typical realization of the
environment the interdistances between scatterers all come out of the
same size and there is no room for long ballistic flights, so that
diffusion becomes normal.

The annealed approach and the quenched approach often provide
different answers. To name some instances of this problem we mention
the survival probability of a random walk among random traps
\cite{BenArous_transition,Gantert_survival}, the large fluctuations of
a random walk in a random environment under the slowdown regime
\cite{Pisztora_quenched,Pisztora_annealed}, the return probability of
a random walk among random conductances with a polynomial tail near
zero \cite{Berger_HK,FontesMathieu}, and the phenomenon of
intermittency for the parabolic Anderson model \cite{PAM_1,PAM_2}.
Regarding the different results of the two approaches to the
L\'evy-Lorentz gas, the point is that a typical realization of the
environment does not account for edges as long as necessary to make a
large fluctuation more probable in the domain of the \textit{principle
  of a single big jump} than in the domain of the \textit{Gibbs
  conditioning principle}. The argument that reconciles the quenched
scenario with Theorem \ref{atail} for the annealed large fluctuations
is that there are realizations of the environment that take an
arbitrarily large time to enter the exponential regime described by
Theorem \ref{qtail}. In other words, at any time there is a fraction
of sample environments, smaller and smaller as the time goes on, that
sustains long ballistic flights of the particle, whereas the shorter
inertial segments corresponding to the other samples make up a
diffusive motion.

The fact that the quenched large fluctuations are suppressed
exponentially entails that the quenched moments are determined only by
the typical fluctuations described by Theorem \ref{CLTq}, which give
normal diffusive scaling of moments. This property is established by
the following theorem, which constitutes the fourth and last main
result of the paper and is proved in Section \ref{proof_qmoments}.
\begin{theorem}
  \label{quenched_moment}
  Assume that $\mu<+\infty$.  The following asymptotic
  equivalence holds for $\pae$: for all $q>0$
  \begin{equation*}
    E\big[|X_t^\omega|^q\big]\sim g_qt^{\frac{q}{2}}=:\sqrt{\frac{(2\mu)^q}{\pi}}\,\Gamma\bigg(\frac{q+1}{2}\bigg)\,t^{\frac{q}{2}}.
  \end{equation*}
\end{theorem}
We supplement Theorem \ref{quenched_moment} with Figure
\ref{fig:moment}, which depicts three quenched moments
$E[|X_t^\omega|^q]$ versus the time $t$ for three sample environments
$\omega$ each under the Pareto's law $\prob[\Delta_1>x]:=(x\vee
1)^{-\alpha}$ for $x>0$ with $\alpha=2$. This law, which is the law
used by physicists \cite{Barkai2,Burioni1,Vezzani1}, gives
$\mu:=\Ex[\Delta_1]=\alpha/(\alpha-1)=2$. The moment orders we
consider are $q=3=2\alpha-1$, $q=4$, and $q=5$, which entail a
superdiffusive behavior within the annealed framework according to
Theorem \ref{annealed_moment}. The figure shows that the quenched
moments exhibit an initial superdiffusive growth, which develops into
normal diffusive scaling on longer time scales as stated by Theorem
\ref{quenched_moment}. Superdiffusive scaling at the level of typical
realizations of the environment thus appears as a transient regime.

\begin{figure}
\centering
\includegraphics[width=14.5cm,height=6cm]{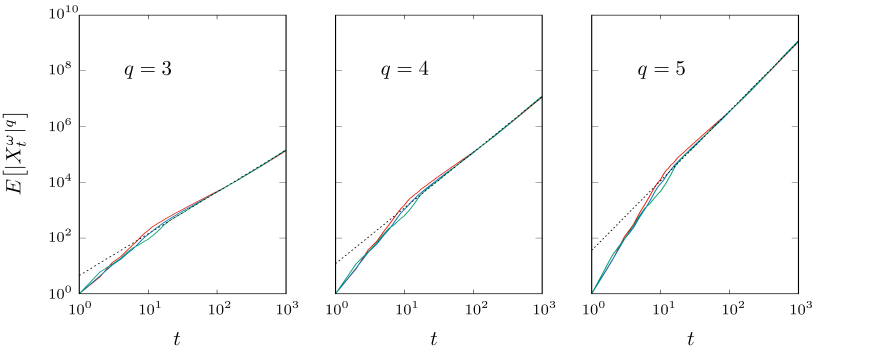}
\caption{Quenched moments $E[|X_t^\omega|^q]$ versus time $t$ in
  log-log scale for $q=3$ (left), $q=4$ (center), and $q=5$
  (right). For each $q$, three different samples $\omega$
  corresponding to different colors are considered and the black
  dotted line is the moment $g_qt^{\frac{q}{2}}$ of the Brownian
  motion rescaled by $\mu$.  Here the model is defined by the Pareto's
  law $\prob[\Delta_1>x]:=(x\vee 1)^{-\alpha}$ for $x>0$ with $\alpha=2$,
  which gives $\mu=2$, and the quenched moments are estimated by
  averaging over $10^{10}$ simulated particle trajectories.}
\label{fig:moment}
\end{figure}

In the light of our findings, asymptotic superdiffusion in individual
samples, if any, necessarily requires that the interdistances between
scattering points have infinite mean $\mu$. Whether or not the
hypothesis $\mu=+\infty$ is sufficient to observe quenched
superdiffusion is left as a completely open problem to be investigated
in future research. We conjecture that the quenched $q$-order moments
may exhibit a scaling law in time with a deterministic exponent as in
Theorem \ref{quenched_moment} even in the case $\mu=+\infty$,
similarly to what was previously found by the author \cite{Zamparo} in
the context of self-similar L\'evy processes, but this time the
coefficient $g_q$ may depend on the sample environment.  Returning to
the case $\mu<+\infty$, while identifying the nature of diffusion in
different frameworks, this work raises a new question about how the
transition from the quenched normal diffusion to the annealed
superdiffusion occurs. The transition from the annealed to the
quenched asymptotics has been addressed in other contexts, such as
sums of random exponentials \cite{BenArous_exponentials}, random walks
among random traps \cite{BenArous_transition}, and the parabolic
Anderson model \cite{PAM_transition_1,PAM_transition_2}.  In these
cases, the approach used by the authors was to average the random
field of interest over a box of increasing size with time, thus
introducing a new parameter that rules the size growth and governs the
quenched to annealed transition. As far as the L\'evy-Lorentz gas is
concerned, in order to study the quenched to annealed transition one
may consider to average the process over an increasing pool of
environment realizations. Work is in progress along this line. As a
final curious remark, we point out that large differences between
individual samples were observed in the experiments by Barthelemy,
Bertolotti, and Wiersma \cite{Barthelemy1} on the transmitted light in
L\'evy glasses, and for this reason they were prompted to average over
a set of 900 disorder realizations. This probably means that their
experimental results lie somewhere between the quenched and the
annealed regime.

\section{Preliminaries to the proofs of the main theorems}
\label{prel}

This section is devoted to some preparatory results for the proof of
Theorems \ref{atail}, \ref{annealed_moment}, \ref{qtail}, and
\ref{quenched_moment}.  For $r\in\mathbb{Z}$, let us say that the edge
$r$ is between points $r-1$ and $r$ and, for $n\ge 1$, let us denote
by $R_n$ the edge that the simple symmetric random walk covers in the
$n$th jump:
\begin{equation*}
  R_n:=\frac{1+S_{n-1}+S_n}{2}.
\end{equation*}
For each $n\ge 1$ and $r\in\mathbb{Z}$, let us consider the local time
$U_{n,r}$ on the edge $r$ given by
\begin{equation*}
  U_{n,r}:=\sum_{k=1}^n\mathds{1}_{\{R_k=r\}} ~~~~~ (U_{0,r}:=0).
\end{equation*}
The variables $R_n$ and $U_{n,r}$ are originally defined on the
probability space $(\mathcal{S},\mathfrak{S},P)$ of velocities.  We
make use of them to recast the $n$th collision time as
\begin{equation*}
  T_n^\omega:=\sum_{k=1}^n|\omega_{S_k}-\omega_{S_{k-1}}|=\sum_{k=1}^n|\omega_{R_k}-\omega_{R_k-1}|=
  \sum_{k=1}^n\Delta^\omega_{R_k}=\sum_{r=-\infty}^{+\infty}\Delta_r^\omega U_{n,r}.
\end{equation*}
Our habit of denoting with same symbols both the variables defined on
$(\mathcal{S},\mathfrak{S},P)$, or $(\Omega,\mathscr{F},\prob)$, and
their natural extensions over the product probability space
$(\Omega\times\mathcal{S},\mathscr{F}\times\mathfrak{S},\prob\times
P)$ allows us to write
$T_n=\sum_{k=1}^n\Delta_{R_k}=\sum_{r=-\infty}^{+\infty}\Delta_rU_{n,r}$. Let
us point out that this habit also justifies the writing of certain
formulas involving probability measures and projections, such as
$\PP[R_n>0]=P[R_n>0]$ or $\EE[\Delta_r]=\Ex[\Delta_r]$.

In Section \ref{jump_edge} we propose some standard bounds for the
variables $R_n$ and $U_{n,r}$, moving their proofs to sections of the
appendix.  These bounds are employed in Section \ref{coll_times} to
achieve an annealed and a quenched large deviation estimate for
collision times.

\subsection{Jumps and local time on edges}
\label{jump_edge}

Set $M_n^+:=\max\{R_1,\ldots,R_n\}$ and
$M_n^-:=\min\{R_1,\ldots,R_n\}$ for each $n\ge 1$.  The first result
we list supplies the distribution of these random variables and is
proved in Appendix \ref{proof:distMn}. Let us observe that $0\le
M_n^+\le n$, as $R_1\ge 0$ and $R_k\le k$ for all $k$, and that
$-n<M_n^-\le 1$, as $R_1\le 1$ and $R_k\ge-k$ for all $k$.  Hereafter
we make the convention that the binomial coefficient ${n\choose k}$ is
$0$ if $k>n$.
\begin{proposition}
  \label{dist:Mn}
  The following conclusions hold:
  \begin{enumerate}[$(i)$]
  \item
    $\{M_n^-\}_{n\in\mathbb{Z}_{\ge 1}}$ and
    $\{1-M_n^+\}_{n\in\mathbb{Z}_{\ge 1}}$ share the same finite-dimensional distributions under the law $P$;
  \item for every integers $n\ge 1$ and $m\ge 0$
  \begin{equation*}
    P\big[M_n^+=m\big]=\frac{1}{2^n}{n\choose\lfloor\frac{n+m+1}{2}\rfloor}.
  \end{equation*}
  \end{enumerate}
\end{proposition}

The tail distribution of $M_n^+$ is described by the following lemma,
whose proof is reported in Appendix \ref{proof:tailMn}.
\begin{lemma}
  \label{tail:Mn}
  For every integer $n\ge 1$ and real number $L>0$
  \begin{equation*}
    P\big[M_n^+>L\big]\le 2e^{-\frac{L^2}{2n}}.
  \end{equation*}
\end{lemma}

The next lemma provides an upper bound for the moments of $M_n^+$ and
is demonstrated in Appendix \ref{proof:momentMn}.
\begin{lemma}
  \label{moment:Mn}
  For every integer $n\ge 1$ and real number $q>0$
  \begin{equation*}
    E\big[(M_n^+)^q\big]\le 2(qn)^{\frac{q}{2}}.
  \end{equation*}
\end{lemma}

The local time $U_{n,r}$ on the edge $r$ is a random variable taking
non-negative integer values.  We point out that $U_{n,r}=0$ unless
$-n<r\le n$ and that $\sum_{r=-\infty}^{+\infty} U_{n,r}=n$. The
following proposition provides the distribution of $U_{n,r}$. The
proof is presented in Appendix \ref{proof:Udist}.
\begin{proposition}
\label{Udist}
  The following conclusions hold for every integer $n\ge 1$:
  \begin{enumerate}[(i)]
  \item $\{U_{n,1-r}\}_{r\in\mathbb{Z}}$ and $\{U_{n,r}\}_{r\in\mathbb{Z}}$ share the same finite-dimensional distributions under the law $P$;
\item for all $r\ge 1$ and $u\ge 0$
\begin{equation*}
  P\big[U_{n,r}=u\big]=
  \begin{cases}
    1-\frac{1}{2^n}\sum_{k=1}^{+\infty}{n \choose \lfloor\frac{n+r+k}{2}\rfloor} & \mbox{if }u=0,\\
    \frac{1}{2^n}{n \choose \lfloor\frac{n+r+u}{2}\rfloor} & \mbox{if }u\ge 1.
    \end{cases}
\end{equation*}
\end{enumerate}
\end{proposition}

Based on the exact results of Proposition \ref{Udist}, the next lemma
provides an estimate of the probability that at least one among the
$U_{n,r}$'s with different $r$ exceeds a given threshold. It is proved
in Appendix \ref{proof:boundU}.
\begin{lemma}
\label{boundU}
For every integer $n\ge 1$ and real number $L>0$
\begin{equation*}
P\Big[\max_{r\in\mathbb{Z}}\{U_{n,r}\}>L\Big]\le 4n e^{-\frac{L^2}{2n}}.
\end{equation*}
\end{lemma}

\subsection{Collision times}
\label{coll_times}

The law of large numbers for collision times stated by Theorem
\ref{SLLN_T} implies that $\PP[T_n\le \gamma n]$ goes to zero as $n$
goes to infinity for each $\gamma<\mu$. We now use Lemma \ref{boundU}
to investigate how zero is approached when $\mu<+\infty$.
\begin{lemma}
\label{Ttailann}
Assume that $\mu<+\infty$. For each $\gamma<\mu$ there exists
$\kappa>0$ such that for all sufficiently large $n$
\begin{equation*}
\PP\big[T_n\le \gamma n\big]\le e^{-\kappa n^{1/3}}.
\end{equation*}
\end{lemma}

\begin{proof}[Proof of Lemma \ref{Ttailann}]
Fix $n\ge 1$. Pick $\gamma<\mu$ and two real numbers $L>0$ and
$\lambda>0$. By making use of the bound $\mathds{1}_{\{T_n\le\gamma
  n\}}\le e^{\lambda\gamma n-\lambda T_n}$ and Lemma \ref{boundU} we
find
\begin{align}
  \nonumber
  \PP\big[T_n\le \gamma n\big]&\le \EE\bigg[\mathds{1}_{\{T_n\le\gamma n\}}
  \prod_{r\in\mathbb{Z}}\mathds{1}_{\{U_{n,r}\le L\}}\bigg]+P\Big[\max_{r\in\mathbb{Z}}\{U_{n,r}>L\}\Big]\\
\nonumber
&\le e^{\lambda\gamma n}\,\EE\bigg[e^{-\lambda T_n}\prod_{r\in\mathbb{Z}}\mathds{1}_{\{U_{n,r}\le L\}}\bigg]+
4n e^{-\frac{L^2}{2n}}.
\end{align}
Then, the identity $T_n=\sum_{r=-\infty}^{+\infty}\Delta_r U_{n,r}$
and Fubini's theorem yield
\begin{align}
\nonumber
\PP\big[T_n\le \gamma n\big]&\le
e^{\lambda\gamma n}\,\EE\Bigg[\prod_{r=-\infty}^{+\infty}\sum_{u=0}^{\lfloor L\rfloor}e^{-\lambda \Delta_ru}\,\mathds{1}_{\{U_{n,r}=u\}}\Bigg]
  +4n e^{-\frac{L^2}{2n}}\\
  &=e^{\lambda\gamma n}E\Bigg[\prod_{r=-\infty}^{+\infty}\sum_{u=0}^{\lfloor L\rfloor}\Ex\big[e^{-\lambda\Delta_1u}\big]\mathds{1}_{\{U_{n,r}=u\}}\Bigg]
  +4n e^{-\frac{L^2}{2n}}.
\label{boundT1}
\end{align}
The monotone convergence theorem tells us that
$\lim_{\eta\uparrow+\infty}\Ex[\min\{\Delta_1,\eta\}]=\mu$, so that
there exists $\eta>0$ with the property that
$\Ex[\min\{\Delta_1,\eta\}]\ge (\mu+\gamma)/2$ as $\gamma<\mu$. Since
$1-z\le e^{-z}\le 1-z+z^2/2$ for every $z\ge 0$, we have for each
$u\ge 0$
\begin{equation}
  \Ex\big[e^{-\lambda\Delta_1u}\big]\le\Ex\big[e^{-\lambda\min\{\Delta_1,\eta\}u}\big]
  \le 1-\frac{\lambda(\mu+\gamma)}{2}u+\frac{\lambda^2\eta^2}{2}u^2\le e^{-\lambda\frac{\mu+\gamma}{2}u+\frac{\lambda^2\eta^2}{2}u^2}.
\label{boundT2}
\end{equation}
By combining (\ref{boundT1}) with (\ref{boundT2}) and by recalling
that $\sum_{r=-\infty}^{+\infty}U_{n,r}=n$ we get
\begin{align}
\nonumber
\PP\big[T_n\le \gamma n\big]
&\le e^{\lambda\gamma n}\,E\Bigg[\prod_{r=-\infty}^{+\infty}\sum_{u=0}^{\lfloor L\rfloor}e^{-\frac{\lambda(\mu+\gamma)}{2}u+\frac{\lambda^2\eta^2L}{2}u}
  \,\mathds{1}_{\{U_{n,r}=u\}}\Bigg]
+4n e^{-\frac{L^2}{2n}}\\
\nonumber
&\le e^{\lambda\gamma n}\,E\bigg[\prod_{r=-\infty}^{+\infty}e^{-\frac{\lambda(\mu+\gamma)}{2}U_{n,r}+\frac{\lambda^2\eta^2L}{2}U_{n,r}}\bigg]
+4n e^{-\frac{L^2}{2n}}= e^{-\frac{\lambda(\mu-\gamma)}{2}n+\frac{\lambda^2\eta^2L}{2}n}+4n e^{-\frac{L^2}{2n}}.
\end{align}
At this point, by optimizing over $\lambda$, that is by taking
$\lambda=(\mu-\gamma)/(2\eta^2L)$, we find
\begin{equation*}
\PP\big[T_n\le \gamma n\big]\le e^{-\frac{(\mu-\gamma)^2}{8\eta^2L}n}+4n e^{-\frac{L^2}{2n}}.
\end{equation*}
The lemma is proved by choosing $L$ proportional to $n^{2/3}$.
\end{proof}

The same estimate of Lemma \ref{Ttailann} applies to typical
realizations of the environment, as stated by the following result.
\begin{lemma}
\label{Ttailq}
Assume that $\mu<+\infty$. For each $\gamma<\mu$ there exists
$\kappa>0$ such that the following property holds for $\pae$: for all
sufficiently large $n$
\begin{equation*}
P\big[T_n^\omega\le \gamma n\big]\le e^{-\kappa n^{1/3}}.
\end{equation*}
\end{lemma}

\begin{proof}[Proof of Lemma \ref{Ttailq}]
Pick $\gamma<\mu$. Lemma \ref{Ttailann} tells us that there exists a
number $\kappa>0$ such that $\PP[T_n\le\gamma n]\le e^{-2\kappa
  n^{1/3}}$ for $n$ large enough. Let $\mathsf{Z}_n$ be the positive
random variable on $(\Omega,\mathscr{F},\prob)$ that associates each
$\omega$ with $e^{\kappa n^{1/3}}P[T_n^\omega\le \gamma n]$. By
invoking Fubini's theorem, we find for all sufficiently large $n$
\begin{equation*}
\Ex[\mathsf{Z}_n]=e^{\kappa n^{1/3}}\int_{\Omega} P\big[T_n^\omega\le\gamma n\big]\,\prob[d\omega]=e^{\kappa n^{1/3}}\,\PP\big[T_n\le\gamma n\big]\le e^{-\kappa n^{1/3}}.
\end{equation*}
This bound gives
$\sum_{n=1}^{+\infty}\prob[\mathsf{Z}_n>\epsilon]\le(1/\epsilon)\sum_{n=1}^{+\infty}\Ex[\mathsf{Z}_n]<+\infty$
for every $\epsilon>0$ thanks to Markov's inequality, which implies
$\lim_{n\uparrow+\infty}\mathsf{Z}_n=0$ $\prob$-a.s.\ by the
Borel-Cantelli lemma. The fact that $\lim_{n\uparrow+\infty}e^{\kappa
  n^{1/3}}P[T_n^\omega\le \gamma n]=0$ for $\pae$ suffices to prove
the lemma.
\end{proof}

\section{Annealed fluctuations: proof of Theorem \ref{atail}}
\label{proof:atail}

The proof of Theorem \ref{atail} is driven by the idea that a large
fluctuation of the particle displacement $X_t$ at a certain time $t$
is determined by a large fluctuation of the current edge along which
the particle is traveling. The current edge and the current velocity
are $R_{N_t}$ and $V_{N_t}$, respectively, where $N_t\ge 1$ is the
label of the first collision after time $t$:
\begin{equation*}
  N_t:=\inf\big\{n\ge 1\,:\,T_n>t\big\}.
\end{equation*}
By definition, $N_t$ is the unique positive integer that satisfies
$T_{N_t-1}\le t<T_{N_t}$, and $N_t-1$ is the number of collisions up
to time $t$.  The edge corresponding to the current edge can be
traveled back and forth several times before $t$, but ultimately its
net contribution to the particle displacement $X_t$ is
\begin{equation}
  \tau_t:=\begin{cases}
    \mathds{1}_{\{V_{N_t}=-1\}}(T_{N_t}-t)+\mathds{1}_{\{V_{N_t}=1\}}(t-T_{N_t-1}) & \mbox{if }R_{N_t}\ge 1,\\
    \mathds{1}_{\{V_{N_t}=-1\}}(t-T_{N_t-1})+\mathds{1}_{\{V_{N_t}=1\}}(T_{N_t}-t) & \mbox{if }R_{N_t}\le 0.
  \end{cases}
  \label{deftau}
\end{equation}
The edges that the particle has completely covered by time $t$ and
that together with $\tau_t$ determine the displacement $X_t$ are those
from $1$ to $R_{N_t}-1$ if $R_{N_t}>1$ and those from $R_{N_t}+1$ to
$0$ if $R_{N_t}<0$. Starting from (\ref{def_position_X}), simple
algebra actually allows one to verify that
\begin{equation}
  X_t=\begin{cases}
  +\sum_{k=1}^{S_{N_t-1}}\Delta_k+V_{N_t}(t-T_{N_t-1}) & \mbox{if } S_{N_t-1}\ge 1,\\
  -\sum_{k=S_{N_t-1}+1}^0\Delta_k+V_{N_t}(t-T_{N_t-1}) & \mbox{if }S_{N_t-1}\le 0
  \end{cases}
  =\begin{cases}
  +\sum_{k=1}^{R_{N_t}-1}\Delta_k+\tau_t & \mbox{if } R_{N_t}\ge 1,\\
  -\sum_{k=R_{N_t}+1}^0\Delta_k-\tau_t & \mbox{if }R_{N_t}\le 0.
  \end{cases}
  \label{Xtau}
\end{equation}
Hereafter we make the usual convention that an empty sum is 0 and an
empty product is 1.  We point out that the rigid constraint
$T_{N_t-1}\le t<T_{N_t}$ yields the bounds $0\le \tau_t\le
T_{N_t}-T_{N_t-1}=\Delta_{R_{N_t}}$. It follows in particular from
(\ref{Xtau}) that $0\le X_t\le \sum_{k=1}^{R_{N_t}}\Delta_k$ or
$-\sum_{k=R_{N_t}}^0\Delta_k\le X_t\le 0$ depending on whether
$R_{N_t}\ge 1$ or $R_{N_t}\le 0$.

Now that the current edge has been characterized, let us outline our
strategy to prove Theorem \ref{atail}.  Fix arbitrary real numbers
$\delta\in(0,1]$ and $h\in(1/2,1)$. Pick $x\in[\delta,1]$ and observe
  that the condition $X_t>xt>0$ implies $R_{N_t}\ge 1$, since $X_t\le
  0$ when $R_{N_t}\le 0$. Based on the bounds
  \begin{equation}
  \PP\big[\tau_t>xt,\,R_{N_t}\ge 1\big]\le\PP\big[X_t>xt\big]\le \PP\big[X_t>xt,\,\tau_t\le hxt\big]+\PP\big[\tau_t>hxt,\,R_{N_t}\ge 1\big],
\label{eq:atail_0}
  \end{equation}
we shall show that a large fluctuation of $X_t$ is entirely determined
by a large fluctuation of $\tau_t$.  Specifically, we will verify by
exact computation that the probability that $\tau_t$ takes a value of
the order of magnitude of $t$ is of the order of magnitude of
$\sqrt{t}\,\prob[\Delta_1>t]$, whereas the probability
$\PP[X_t>xt,\,\tau_t\le hxt]$ of a large fluctuation of $X_t$ that
cannot be ascribed to $\tau_t$ is negligible with respect to
$\sqrt{t}\,\prob[\Delta_1>t]$. Regarding the latter, we stress that
the conditions $X_t>xt$ and $\tau_t\le hxt$ entail $R_{N_t}\ge 2$, as
$X_t=\tau_t$ when $R_{N_t}=1$ according to (\ref{Xtau}). They also
imply $\sum_{k=1}^{R_{N_t}-1}\Delta_k=X_t-\tau_t>(1-h)xt\ge(1-h)\delta
t$. Thus, setting $\eta:=(1-h)\delta>0$ for brevity, we have
\begin{align}
\nonumber
  \PP\big[X_t>xt,\,\tau_t\le hxt\big]&\le\PP\Bigg[\sum_{k=1}^{R_{N_t}-1}\Delta_k>\eta t,\,R_{N_t}\ge 2\Bigg]\\
    \nonumber
    &\le \PP\Bigg[\sum_{k=1}^{R_{N_t}-1}\Delta_k>\eta t,\,\max\big\{\Delta_1,\ldots,\Delta_{R_{N_t}-1}\big\}\le \eta t/2,\,R_{N_t}\ge 2\Bigg]\\
    &+\PP\Big[\max\big\{\Delta_1,\ldots,\Delta_{R_{N_t}-1}\big\}>\eta t/2,\,R_{N_t}\ge 2\Big].
    \label{eq:atail_1}
\end{align}
Inequality (\ref{eq:atail_1}) constitutes our starting point to prove
that the probability $\PP[X_t>xt,\,\tau_t\le hxt]$ is negligible with
respect to $\sqrt{t}\,\prob[\Delta_1>t]$ by drawing attention to two
distinct imperatives: there must be a long gap, which cannot be
different from the current edge.  In fact, the \textit{principle of a
  single big jump} for i.i.d.\ heavy-tailed random variables suggests
that a fluctuation of $\sum_{k=1}^{R_{N_t}-1}\Delta_k$ larger than
$\eta t$ requires a comparable fluctuation of one of the summands. We
will verify this fact by showing that the first term of the r.h.s.\ of
(\ref{eq:atail_1}) is suppressed by the constraint
$\max\{\Delta_1,\ldots,\Delta_{R_{N_t}-1}\}\le \eta t/2$.  On the
other hand, as the summands are the lengths of the edges that the
particle has traveled in the past, when one of them is allowed for a
large fluctuation, i.e.\ when
$\max\{\Delta_1,\ldots,\Delta_{R_{N_t}-1}\}> \eta t/2$, the particle
experiences the circumstance of not currently being on the edge with
the largest length. This circumstance turns out to be very unlikely,
meaning that the particle spends most of the time on the current edge
after reaching it for the first time. We shall demonstrate this fact
by proving that also the second term of the r.h.s.\ of
(\ref{eq:atail_1}) is negligible. The mechanism by which the particle
remains confined to the current edge, traveling back and forth on it,
is that the many close collisions involving the particle as it leaves
the longest edge send it back very quickly.

When evaluating all terms in (\ref{eq:atail_0}) and (\ref{eq:atail_1})
we can restrict to events for which the number of collisions by time
$t$ has at most order of magnitude of $t/\mu$, as it is very unlikely
that $N_t$ takes larger values. In fact, we can cut off $N_t$ at
$b_t:=\lfloor 2t/\mu\rfloor+1$ since Lemma \ref{Ttailann} tells us
that there exists a real number $\kappa>0$ such that for all
sufficiently large $t$
  \begin{equation}
    \PP\big[N_t>b_t\big]=\PP\big[T_{b_t}\le t\big]\le\PP\big[T_{b_t}\le \mu b_t/2\big]\le e^{-\kappa b_t^{1/3}}.
\label{upper_prob_1}
  \end{equation}
By combining (\ref{eq:atail_0}) and (\ref{eq:atail_1}) we get for any
$x\in[\delta,1]$
\begin{equation}
\Psi_t(x)\le \PP\big[X_t>xt\big]\le \Theta_t+\Phi_t+\Psi_t(hx)+3\,\PP\big[N_t>b_t\big]
\label{atail_prob_LU}
\end{equation}
with
\begin{equation*}
\Theta_t:=\PP\Bigg[\sum_{k=1}^{R_{N_t}-1}\Delta_k>\eta t,\,\max\big\{\Delta_1,\ldots,\Delta_{R_{N_t}-1}\big\}\le \eta t/2,\,R_{N_t}\ge 2,\,N_t\le b_t\Bigg],
\end{equation*}
\begin{equation*}
\Phi_t:=\PP\Big[\max\big\{\Delta_1,\ldots,\Delta_{R_{N_t}-1}\big\}>\eta t/2,\,R_{N_t}\ge 2,\,N_t\le b_t\Big],
\end{equation*}
and
\begin{equation*}
\Psi_t(y):=\PP\Big[\tau_t>yt,\,R_{N_t}\ge 1,\,N_t\le b_t\Big].
\end{equation*}
In Section \ref{sec:Theta1} we shall make use of the precise large
deviation principle (\ref{LDPheavy}) to show that
\begin{equation}
\lim_{t\uparrow+\infty}\frac{\Theta_t}{\sqrt{t}\,\prob[\Delta_1>t]}=0.
\label{Theta1}
\end{equation}
In Section \ref{sec:Phi1} we will prove that
\begin{equation}
\lim_{t\uparrow+\infty}\frac{\Phi_t}{\sqrt{t}\,\prob[\Delta_1>t]}=0.
\label{Phi1}
\end{equation}
Finally, in Section \ref{dominant_term} we shall demonstrate that
\begin{equation}
  \adjustlimits\lim_{t\uparrow+\infty}\sup_{~y\in[\delta/2,1]~}\!\Bigg\{\bigg|\frac{\Psi_t(y)}{\sqrt{t}\,\prob[\Delta_1>t]}-F(y)\bigg|\Bigg\}=0,
\label{Psi1}
\end{equation}
where for each $y\in(0,1]$
\begin{equation}
  F(y):=\frac{1}{\sqrt{2\pi\mu}}\int_y^1 f_\alpha\bigg(\frac{y}{\xi}\bigg)\frac{d\xi}{\xi^\alpha\sqrt{1-\xi}},
\label{defFy}
\end{equation}
$f_\alpha$ being the function defined by (\ref{deffa}).  In Section
\ref{sec:F(hx)} we shall verify that
\begin{equation}
\sup_{x\in[\delta,1]}\Big\{F(hx)-F(x)\Big\}\le\frac{\sqrt{1-h}+h^{1-\alpha}-1}{\sqrt{\mu}\,(\delta/2)^\alpha}.
\label{F(hx)}
\end{equation}
Since the tail probability of $\Delta_1$ is dominated from below by
polynomials due to regular variation (see \cite{RV}, Proposition
1.3.6), bound (\ref{upper_prob_1}) gives
\begin{equation}
  \lim_{t\uparrow+\infty}\frac{\PP[N_t>b_t]}{\sqrt{t}\,\prob[\Delta_1>t]}=0.
\label{atail_prob_LU_1}
\end{equation}
In conclusion, (\ref{atail_prob_LU}) in combination with the limits
(\ref{Theta1}), (\ref{Phi1}), (\ref{Psi1}), and
(\ref{atail_prob_LU_1}) and the estimate (\ref{F(hx)}) shows that
\begin{equation*}
  \adjustlimits\limsup_{t\uparrow+\infty}\sup_{x\in[\delta,1]}\Bigg\{\bigg|\frac{\PP[X_t>xt]}{\sqrt{t}\,\prob[\Delta_1>t]}-F(x)\bigg|\Bigg\}
  \le\frac{\sqrt{1-h}+h^{1-\alpha}-1}{\sqrt{\mu}\,(\delta/2)^\alpha}.
\end{equation*}
Theorem \ref{atail} follows from here by sending $h$ to 1.  Limit
(\ref{Theta1}) provides a rigorous interpretation of the
\textit{principle of a single big jump} in the context of the
L\'evy-Lorentz gas by demonstrating that large fluctuations of the
particle displacement cannot be observed in the absence of large
fluctuations of a single edge. Limit (\ref{Phi1}) confirms that it is
very unlikely that there is an edge that exhibits a large fluctuation,
but that this is not the current edge.

\subsection{The limit (\ref{Theta1}) for $\Theta_t$}
\label{sec:Theta1}

The first step to prove the limit (\ref{Theta1}) is to eliminate the
random variable $N_t$ from $\Theta_t$.  Bearing in mind that
$\max\{R_1,\ldots,R_{b_t}\}=M_{b_t}^+$, for each $t>0$ we have
\begin{align}
\nonumber
\Theta_t&:=\PP\Bigg[\sum_{k=1}^{R_{N_t}-1}\Delta_k>\eta t,\,\max\big\{\Delta_1,\ldots,\Delta_{R_{N_t}-1}\big\}\le \eta t/2,\,R_{N_t}\ge 2,\,N_t\le b_t\Bigg]\\
\nonumber
&\le\PP\Bigg[\sum_{k=1}^{R_{N_t}-1}\min\big\{\Delta_k,\eta t/2\big\}>\eta t,\,R_{N_t}\ge 2,\,N_t\le b_t\Bigg]\\
&\le\PP\Bigg[\sum_{k=1}^{M_{b_t}^+}\min\big\{\Delta_k,\eta t/2\big\}>\eta t,\,M_{b_t}^+\ge 2\Bigg]\le\Theta_t'+\Theta_t''
\label{theta_bound_1}
\end{align}
with
\begin{equation*}
\Theta_t':=\PP\Bigg[\sum_{k=1}^{M_{b_t}^+}\min\big\{\Delta_k,\eta t/2\big\}>\eta t,\,\max\big\{\Delta_1,\ldots,\Delta_{M_{b_t}^+}\big\}>\eta t,\,M_{b_t}^+\ge 2\Bigg]
\end{equation*}
and
\begin{equation*}
  \Theta_t'':=\PP\Bigg[\sum_{k=1}^{M_{b_t}^+}\Delta_k>\eta t,\,\max\big\{\Delta_1,\ldots,\Delta_{M_{b_t}^+}\big\}\le \eta t,\,M_{b_t}^+\ge 1\Bigg].
\end{equation*}
The bound (\ref{theta_bound_1}) shows that (\ref{Theta1}) follows if
we demonstrate that
\begin{equation}
\lim_{t\uparrow+\infty}\frac{\Theta_t'}{\sqrt{t}\,\prob[\Delta_1>t]}=\lim_{t\uparrow+\infty}\frac{\Theta_t''}{\sqrt{t}\,\prob[\Delta_1>t]}=0.
\label{theta_bound_2}
\end{equation}

The first limit of (\ref{theta_bound_2}) is immediate.  Since the
condition $\max\{\Delta_1,\ldots,\Delta_m\}>\eta t$ with $m\ge 2$
entails that at least one variable among the i.i.d.\ random variables
$\Delta_1,\ldots,\Delta_m$ is larger that $\eta t$, we can write
\begin{align}
\nonumber
\Theta_t'&\le \sum_{m=2}^{+\infty}\sum_{l=1}^m\PP\Bigg[\sum_{k=1}^m\min\big\{\Delta_k, \eta t/2\big\}>\eta t,\,\Delta_l>\eta t,\,M_{b_t}^+=m\Bigg]\\
\nonumber
&\le\sum_{m=2}^{+\infty}m\,\PP\Bigg[\sum_{k=1}^{m-1}\Delta_k>\eta t/2,\,\Delta_m>\eta t,\,M_{b_t}^+=m\Bigg]\\
\nonumber
&=\prob\big[\Delta_1>\eta t\big] \sum_{m=2}^{+\infty}m\,\prob\Bigg[\sum_{k=1}^{m-1}\Delta_k>\eta t/2\Bigg] P\big[M_{b_t}^+=m\big].
\end{align}
At this point, an application of Markov's inequality and Lemma
\ref{moment:Mn} give for all $t$ 
\begin{equation*}
\Theta_t'\le\prob\big[\Delta_1>\eta t\big] \sum_{m=2}^{+\infty}\frac{2\mu m(m-1)}{\eta t}P\big[M_{b_t}^+=m\big]
\le \frac{2\mu E[(M_{b_t}^+)^2]}{\eta t}\,\prob\big[\Delta_1>\eta t\big]\le \frac{8\mu b_t}{\eta t}\,\prob\big[\Delta_1>\eta t\big].
\end{equation*}
The limit (\ref{theta_bound_2}) for $\Theta_t'$ follows from here
since $b_t$ is of the order of magnitude of $t$ and
$\lim_{t\uparrow+\infty}\prob[\Delta_1>\eta
  t]/\prob[\Delta_1>t]=\eta^{-\alpha}<+\infty$ by regular variation.

The second limit of (\ref{theta_bound_2}) involves the precise large
deviation principle (\ref{LDPheavy}) for the interdistances between
scatterers. In a nutshell, $\Theta_t''$ is negligible with respect to
$\sqrt{t}\,\prob[\Delta_1>t]$ because a fluctuation of
$\sum_{k=1}^{M_{b_t}^+}\Delta_k$ lager than $\eta t$ is realized by a
comparable fluctuation of one of the summands, but this is impossible
under the constraint $\max\{\Delta_1,\ldots,\Delta_{M_{b_t}^+}\}\le
\eta t$. In order to get at a rigorous proof, let us notice at first
that the condition $\max\{\Delta_1,\ldots,\Delta_{M_{b_t}^+}\}>\eta t$
entails $\sum_{k=1}^{M_{b_t}^+}\Delta_k>\eta t$, so that
\begin{equation}
  \Theta_t''=\PP\Bigg[\sum_{k=1}^{M_{b_t}^+}\Delta_k>\eta t,\,M_{b_t}^+\ge 1\Bigg]-\PP\Bigg[\max\big\{\Delta_1,\ldots,\Delta_{M_{b_t}^+}\big\}>\eta t,\,M_{b_t}^+\ge 1\Bigg].
\label{theta_primo_1}
\end{equation}
To address the first term of the r.h.s.\ of (\ref{theta_primo_1}),
pick an arbitrary real number $\epsilon\in(0,1/2)$. In the light of
(\ref{LDPheavy}) with $\gamma=\mu$ there exists an integer $m_o\ge 1$
such that for all $m\ge m_o$ and $z\ge \mu m$
\begin{equation}
\prob\Bigg[\sum_{k=1}^m\Delta_k>\mu m+z\Bigg]\le m(1+\epsilon)\,\prob\big[\Delta_1>z\big].
\label{LDPheavy_1}
\end{equation}
Set $t_o:=\mu m_o/(\epsilon\eta)$. If $t>t_o$, then $\epsilon\eta
t-\mu m_o>0$ and Lemma \ref{tail:Mn} shows that
\begin{equation*}
P\big[\mu (M_{b_t}^+\vee m_o)>\epsilon \eta t\big]\le P\big[\mu M_{b_t}^+>\epsilon \eta t-\mu m_o\big]\le 2 e^{-\frac{(\epsilon \eta t-\mu m_o)^2}{2\mu^2b_t}}.
\end{equation*}
Thus, if $m\ge 1$ satisfies $\mu (m\vee m_o)\le \epsilon \eta t$, then
$\eta t-\mu (m\vee m_o)\ge(1-\epsilon)\eta t>\epsilon \eta t\ge\mu
(m\vee m_o)$ as $\epsilon<1/2$ and (\ref{LDPheavy_1}) yields
\begin{align}
  \nonumber
  \prob\Bigg[\sum_{k=1}^m\Delta_k>\eta t\Bigg]&\le\prob\Bigg[\sum_{k=1}^{m\vee m_o}\Delta_k>\eta t\Bigg]
  \le\prob\Bigg[\sum_{k=1}^{m\vee m_o}\Delta_k>\mu(m\vee m_o)+(1-\epsilon)\eta t\Bigg]\\
\nonumber
  &\le (m\vee m_o)(1+\epsilon)\,\prob\big[\Delta_1>(1-\epsilon)\eta t\big]\le (m_o+m)(1+\epsilon)\,\prob\big[\Delta_1>(1-\epsilon)\eta t\big].
\end{align}
It follows that for every $t>t_o$
\begin{align}
  \nonumber
  \PP\Bigg[\sum_{k=1}^{M_{b_t}^+}\Delta_k>\eta t,\,M_{b_t}^+\ge1\Bigg]&
  \le\sum_{m=1}^{+\infty}\mathds{1}_{\big\{\mu (m\vee m_o)\le \epsilon \eta t\big\}}\,\prob\Bigg[\sum_{k=1}^m\Delta_k>\eta t\Bigg]P\big[M_{b_t}^+=m\big]
  +P\big[\mu (M_{b_t}^+\vee m_o)>\epsilon \eta t\big]\\
  &\le (1+\epsilon)\,\prob\big[\Delta_1>(1-\epsilon)\eta t\big]\Big(m_o+E\big[M_{b_t}^+\big]\Big)+2 e^{-\frac{(\epsilon \eta t-\mu m_o)^2}{2\mu^2b_t}}.
\label{theta_primo_2}
\end{align}
Furthermore, the independence of the interdistances and the inequality
$1-(1-z)^m\ge mz-m^2z^2$ valid for every $m\ge1$ and $z\in[0,1]$ give
for all $t>0$
\begin{align}
\nonumber
\PP\bigg[\max\big\{\Delta_1,\ldots,\Delta_{M_{b_t}^+}\big\}>\eta t,\,M_{b_t}^+\ge 1\bigg]
&=\sum_{m=1}^{+\infty}\bigg\{1-\Big(1-\prob\big[\Delta_1>\eta t\big]\Big)^m\bigg\}P\big[M_{b_t}^+=m\big]\\
&\ge E\big[M_{b_t}^+\big]\prob\big[\Delta_1>\eta t\big]-E\big[(M_{b_t}^+)^2\big]\prob\big[\Delta_1>\eta t\big]^2.
\label{theta_primo_3}
\end{align}
In conclusion, by combining (\ref{theta_primo_1}) with
(\ref{theta_primo_2}) and (\ref{theta_primo_3}) we realize that for
all $t>t_o$
\begin{align}
  \nonumber
  \Theta_t''&\le (1+\epsilon)\,\prob\big[\Delta_1>(1-\epsilon)\eta t\big]\Big(m_o+E\big[M_{b_t}^+\big]\Big)
  -E\big[M_{b_t}^+\big]\prob\big[\Delta_1>\eta t\big]\\
  \nonumber
  &+E\big[(M_{b_t}^+)^2\big]\prob\big[\Delta_1>\eta t\big]^2+2 e^{-\frac{(\epsilon \eta t-\mu m_o)^2}{2\mu^2b_t}}\\
  \nonumber
  &=m_0(1+\epsilon)\,\prob\big[\Delta_1>(1-\epsilon)\eta t\big]+\epsilon E\big[M_{b_t}^+\big]\prob\big[\Delta_1>(1-\epsilon)\eta t\big]
  +E\big[M_{b_t}^+\big]\prob\big[(1-\epsilon)\eta t<\Delta_1\le\eta t\big]\\
  \nonumber
  &+E\big[(M_{b_t}^+)^2\big]\prob\big[\Delta_1>\eta t\big]^2+2 e^{-\frac{(\epsilon \eta t-\mu m_o)^2}{2\mu^2b_t}}\\
  \nonumber
  &\le m_0(1+\epsilon)\,\prob\big[\Delta_1>(1-\epsilon)\eta t\big]+2\epsilon\sqrt{b_t}\,\prob\big[\Delta_1>(1-\epsilon)\eta t\big]
  +2\sqrt{b_t}\,\prob\big[(1-\epsilon)\eta t<\Delta_1\le\eta t\big]\\
  \nonumber
  &+\frac{4\mu b_t}{\eta t}\,\prob\big[\Delta_1>\eta t\big]+2 e^{-\frac{(\epsilon \eta t-\mu m_o)^2}{2\mu^2b_t}},
\end{align}
where we have used Lemma \ref{moment:Mn} to bound the moments of
$M_{b_t}^+$ and the Markov's inequality to deduce $\prob[\Delta_1>\eta
  t]\le\mu/(\eta t)$. Regular variation implies
$\lim_{t\uparrow+\infty}\prob[\Delta_1>\beta
  t]/\prob[\Delta_1>t]=\beta^{-\alpha}$ for any $\beta>0$, as well as
the fact that $\prob[\Delta_1>\cdot\,]$ is dominated from below by
polynomials (see \cite{RV}, Proposition 1.3.6). Then, recalling that
$b_t:=\lfloor 2t/\mu\rfloor+1$, we find
\begin{equation*}
  \limsup_{t\uparrow+\infty}\frac{\Theta_t''}{\sqrt{t}\,\prob[\Delta_1>t]}\le \eta^{-\alpha}\sqrt{\frac{8}{\mu}}
  \bigg[\frac{1+\epsilon}{(1-\epsilon)^\alpha}-1\bigg].
\end{equation*}
The arbitrariness of $\epsilon\in(0,1/2)$ proves the limit
(\ref{theta_bound_2}) for $\Theta_t''$.

\subsection{The limit (\ref{Phi1}) for $\Phi_t$}
\label{sec:Phi1}

Let us prove the limit (\ref{Phi1}) for $\Phi_t$. Since the condition
$\max\{\Delta_1,\ldots,\Delta_r\}>\eta t/2$ implies that at least one
variable among $\Delta_1,\ldots,\Delta_r$ exceeds $\eta t/2$, we can
write down the bound
\begin{align}
  \nonumber
\Phi_t&:=\PP\Big[\max\big\{\Delta_1,\ldots,\Delta_{R_{N_t}-1}\big\}>\eta t/2,\,R_{N_t}\ge 2,\,N_t\le b_t\Big]\\
\nonumber
&=\sum_{r=1}^{+\infty}\,\PP\Big[\max\big\{\Delta_1,\ldots,\Delta_r\big\}>\eta t/2,\,R_{N_t}=r+1,\,N_t\le b_t\Big]\\
&\le\sum_{r=1}^{+\infty}\sum_{k=1}^r\,\PP\Big[\Delta_k>\eta t/2,\,R_{N_t}=r+1,\,N_t\le b_t\Big]=
\sum_{k=1}^{+\infty}\,\PP\Big[\Delta_k>\eta t/2,\,R_{N_t}>k,\,N_t\le b_t\Big].
\label{phi_bound_0}
\end{align}
Because of the constraint $1\le k<R_{N_t}$, the last expression
clearly states that an edge $k$ that determines the displacement of
the particle at time $t$ contributes with a large fluctuation above
$\eta t/2$, but this edge is not the current edge $R_{N_t}$. A second
large fluctuation involving another edge is very unlikely, so that the
particle undergoes many close collisions when it leaves the edge $k$.
Actually, since the sequence $\{R_n\}_{n\in\mathbb{Z}_{>0}}$ with
increments $-1$, $0$, or $1$ cannot reach a point larger than $k\ge 1$
without going through $k$ first, the condition $1\le k<R_{N_t}$
implies that there exists a positive integer $n<N_t$ such that $R_n=k$
and $\min\{R_{n+1},\ldots,R_{N_t}\}>k$. This means that the particle
passed on the longest edge $k$ one last time at the $n$th jump and
then has no longer passed over this edge. Then, the reason why
$\Phi_t$ is negligible with respect to $\sqrt{t}\,\prob[\Delta_1>t]$
is that it is very improbable that the particle no longer comes back
to the edge $k$ as a result of the many collisions between times $T_n$
and $T_{N_t-1}$.

Starting from (\ref{phi_bound_0}), the naive implementation of these
arguments leads us to the bound
\begin{equation}
\Phi_t\le \sum_{k=1}^{+\infty}\sum_{n=1}^{b_t-1}\,\PP\bigg[\Delta_k>\eta t/2,\,R_n=k,\,\min\big\{R_{n+1},\ldots,R_{N_t}\big\}>k,\,n<N_t\le b_t\bigg].
\label{phi_bound_1}
\end{equation}
The number $N_t$ is the complex mathematical object of
(\ref{phi_bound_1}) because of its dependence on the other random
elements. With the purpose of resolving such dependence and eliminate
$N_t$, we come back to (\ref{phi_bound_0}) and do some preliminary
manipulations to (\ref{phi_bound_1}). As $N_t$ is the unique positive
integer that satisfies $T_{N_t-1}\le t<T_{N_t}$, the idea is to invert
this relation in order to predict $N_t$. To do this, let us consider
for each $n\ge 0$ a collision time $\mathsf{T}_{n,k}$ deprived of the
time that the particle spends on a given edge $k\in\mathbb{Z}$:
\begin{equation}
\mathsf{T}_{n,k}:=T_n-\Delta_kU_{n,k}=\sum_{\substack{r=-\infty\\r\ne k}}^{+\infty}\Delta_rU_{n,r}.
\label{def_T_deprived}
\end{equation}
The sequence $\{\mathsf{T}_{n,k}\}_{n\in\mathbb{Z}_{\ge 0}}$ is
independent of $\Delta_k$ and non-decreasing, the latter being a
property of local times on edges. The law of large numbers for
collision times stated by Theorem \ref{SLLN_T} suggests that
$\mathsf{T}_{n,k}/n$ may converge in probability to $\mu$ when $n$ is
sent to infinity, even if the edge $k$ is affected by a large
fluctuation.  We have
$T_{N_t-1}=\mathsf{T}_{N_t-1,k}+\Delta_kU_{N_t-1,k}$, with
$U_{N_t-1,k}=U_{N_t,k}-\mathds{1}_{\{R_{N_t}=k\}}=U_{N_t,k}$ by
definition if $k\ne R_{N_t}$, and
$T_{N_t}=\mathsf{T}_{N_t,k}+\Delta_kU_{N_t,k}$, so that the constraint
$T_{N_t-1}\le t<T_{N_t}$ reads $\mathsf{T}_{N_t-1,k}\le
t-\Delta_kU_{N_t,k}<\mathsf{T}_{N_t,k}$ if $k\ne R_{N_t}$.  The trick
to invert this relation is to replace $\mathsf{T}_{N_t-1,k}/(N_t-1)$
and $\mathsf{T}_{N_t,k}/N_t$ with $\mu$, and to show that such a
replacement is justified.  So, we pick a real number
$\epsilon\in(0,1)$ and recast (\ref{phi_bound_0}) as
\begin{equation}
  \Phi_t\le \Phi_t'+\Phi_t''
\label{Phi_primo_secondo}
\end{equation}
with
\begin{equation}
\Phi_t':=\sum_{k=1}^{+\infty}\,\PP\Bigg[\Delta_k>\eta t/2,\,\bigg|\frac{\mathsf{T}_{N_t-1,k}}{N_t-1}-\mu\bigg|\vee
    \bigg|\frac{\mathsf{T}_{N_t,k}}{N_t}-\mu\bigg|\le\epsilon\mu,\,R_{N_t}>k\Bigg]
\label{def_Phi_primo}
\end{equation}
and
\begin{equation}
\Phi_t'':=\sum_{k=1}^{+\infty}\,\PP\Bigg[\Delta_k>\eta t/2,\,\bigg|\frac{\mathsf{T}_{N_t-1,k}}{N_t-1}-\mu\bigg|\vee
    \bigg|\frac{\mathsf{T}_{N_t,k}}{N_t}-\mu\bigg|>\epsilon\mu,\,R_{N_t}>k,\,N_t\le b_t\Bigg].
\label{def_Phi_secondo}
\end{equation}
Notice that we have dropped the condition $N_t\le b_t$ when defining
$\Phi_t'$ because irrelevant.

The term $\Phi_t'$ accounts for events where
$|\mathsf{T}_{N_t-1,k}/(N_t-1)-\mu|\le \epsilon\mu$ and
$|\mathsf{T}_{N_t,k}/N_t-\mu|\le\epsilon\mu$. For these events the
constraint $\mathsf{T}_{N_t-1,k}\le
t-\Delta_kU_{N_t,k}<\mathsf{T}_{N_t,k}$ implies $\Delta_kU_{N_t,k}\le
t$ and $(1-\epsilon)\mu(N_t-1)\le t-\Delta_kU_{N_t,k}<(1+\epsilon)\mu
N_t$. The latter is tantamount to $N_t\in
(p_{t-\Delta_kU_{N_t,k}},q_{t-\Delta_kU_{N_t,k}}+1]$ with integers
  $p_z$ and $q_z$ defined for every $z\in\Rl$ by
\begin{equation}
p_z:=\bigg\lfloor\frac{z\vee 0}{(1+\epsilon)\mu}\bigg\rfloor\ge 0
\label{def:pz}
\end{equation}
and
\begin{equation}
q_z:=\bigg\lfloor\frac{z\vee 0}{(1-\epsilon)\mu}\bigg\rfloor\ge p_z.
\label{def:qz}
\end{equation}
We have thus conveniently inverted the relation $T_{N_t-1}\le
t<T_{N_t}$ and found a narrow interval that encloses $N_t$. These
reasonings allow us to deduce from (\ref{def_Phi_primo}) that for
every $t>0$
\begin{align}
\nonumber
\Phi_t'&\le\sum_{k=1}^{+\infty}\,\PP\bigg[\Delta_k>\eta t/2,\,N_t\in \big(p_{t-\Delta_kU_{N_t,k}},q_{t-\Delta_kU_{N_t,k}}+1\big],
\,\Delta_kU_{N_t,k}\le t,\,R_{N_t}>k\bigg]\\
  &\le\sum_{u=0}^{\lfloor 2/\eta\rfloor}\sum_{k=1}^{+\infty}\,\PP\bigg[\Delta_k>\eta t/2,\,N_t\in\big(p_{t-\Delta_ku},q_{t-\Delta_ku}+1\big],\,U_{N_t,k}=u,\,R_{N_t}>k\bigg].
\label{Phi_primo_1}
\end{align}
At this point, we can recover the arguments that led to
(\ref{phi_bound_1}), namely that the condition $1\le k<R_{N_t}$
entails that there exists a positive integer $n<N_t$ such that $R_n=k$
and $\min\{R_{n+1},\ldots,R_{N_t}\}>k$. This gives in particular
$U_{N_t,k}=U_{n,k}\ge 1$. The novelty is that we can now use that
constraint $N_t\in(p_{t-\Delta_ku},q_{t-\Delta_ku}+1]$ to predict
  $N_t$ and to conclude that $n$ satisfies $n\le q_{t-\Delta_ku}$,
  $R_n=k$, and $\min\{R_{n+1},\ldots,R_{n\vee
    p_{t-\Delta_ku}+1}\}>k$. Thus, we can replace (\ref{Phi_primo_1})
  with
\begin{align}
  \nonumber
\Phi_t'&\le\sum_{u=1}^{\lfloor 2/\eta\rfloor}\sum_{k=1}^{+\infty}\sum_{n=1}^{+\infty}\,\PP\bigg[\Delta_k>\eta t/2,\,n\le q_{t-\Delta_ku},
  \,U_{n,k}=u,\,R_n=k,\,\min\big\{R_{n+1},\ldots,R_{n\vee p_{t-\Delta_ku}+1}\big\}>k\bigg].
\end{align}
This bound is a version of (\ref{phi_bound_1}) where $N_t$ has been
resolved. Bearing in mind that $\Delta_k$ is distributed as
$\Delta_1$, an application of Fubini's theorem finally yields for any
$t>0$
\begin{align}
  \nonumber
\Phi_t'&\le\sum_{u=1}^{\lfloor 2/\eta\rfloor}\Ex\Bigg[\mathds{1}_{\{\Delta_1>\eta t/2\}}\sum_{n=1}^{q_{t-\Delta_1u}}\sum_{k=1}^{+\infty}\,
  P\Big[U_{n,k}=u,\,R_n=k,\,\min\big\{R_{n+1},\ldots,R_{n\vee p_{t-\Delta_1u}+1}\big\}>k\Big]\Bigg]\\
&\le\sum_{u=1}^{\lfloor 2/\eta\rfloor}\Ex\Big[\mathds{1}_{\{\Delta_1>\eta t/2\}}\,
    \mathcal{E}_u\big\{p_{t-\Delta_1 u},q_{t-\Delta_1 u}\big\}\Big],
\label{Phi_primo_2}
\end{align}
where we have set
\begin{equation}
 \mathcal{E}_u\{p,q\}:=\sum_{n=1}^q\,P\Big[U_{n,R_n}=u,\,\min\big\{R_{n+1},\ldots,R_{n\vee p+1}\big\}>R_n\Big]
\label{defEcal}
\end{equation}
for all integers $0\le p\le q$ and $u>0$.  Through
(\ref{Phi_primo_2}), we have reduced the problem to an estimate of
$\mathcal{E}_u\{p,q\}$, which only involves the properties of the
simple symmetric random walk. We let the following lemma, whose proof
is provided in Appendix \ref{proof:Ecal_estimate}, to tackle
$\mathcal{E}_u\{p,q\}$.  Basically, $\mathcal{E}_u\{p,q\}$ measures
how likely it is that, at a given time, the particle is not on an edge
it has visited a few times in the past.

\begin{lemma}
  \label{Ecal_estimate}
Let $\mathcal{E}_u\{p,q\}$ by defined by (\ref{defEcal}) for every
integers $0\le p\le q$ and $u>0$. Then
\begin{equation*}
\mathcal{E}_u\{p,q\}\le 5+\sqrt{q}-\sqrt{p+1}.
\end{equation*}
\end{lemma}

Lemma \ref{Ecal_estimate} and definitions (\ref{def:pz}) and
(\ref{def:qz}) give for every positive $u$ and $t$
\begin{equation*}
  \mathcal{E}_u\big\{p_{t-\Delta_1 u},q_{t-\Delta_1 u}\big\}\le 5+\sqrt{\frac{(t-\Delta_1 u)\vee 0}{(1-\epsilon)\mu}}-\sqrt{\frac{(t-\Delta_1 u)\vee 0}{(1+\epsilon)\mu}}
  \le  5+\bigg(\frac{1}{\sqrt{1-\epsilon}}-\frac{1}{\sqrt{1+\epsilon}}\bigg)\sqrt{\frac{t}{\mu}}.
\end{equation*}
It follows from (\ref{Phi_primo_2}) that for all $t>0$
\begin{equation}
  \Phi_t'\le \frac{10}{\eta}\,\prob\big[\Delta_1>\eta t/2\big]+\frac{2}{\eta}\bigg(\frac{1}{\sqrt{1-\epsilon}}-\frac{1}{\sqrt{1+\epsilon}}\bigg)
  \sqrt{\frac{t}{\mu}}\,\prob\big[\Delta_1>\eta t/2\big].
\label{Phi_primo_3}
\end{equation}

Let us move to $\Phi_t''$.  This term involves the deviations of the
random variables $\mathsf{T}_{n,k}/n$ from $\mu$, which are not
significant as we now show. To deal with $\Phi_t''$ it is convenient
to distinguish events that account for few collisions from the
others. Set $a_t:=\lfloor \epsilon t/\mu\rfloor$, which is an integer
that satisfies $0\le a_t<b_t:=\lfloor 2t/\mu\rfloor+1$. Events in
(\ref{def_Phi_secondo}) that account for few collisions are
interpreted as events where $N_t\le a_t$, which can be simply tackled
by noticing that
\begin{align}
  \nonumber
  \sum_{k=1}^{+\infty}\,\PP\Big[\Delta_k>\eta t/2,\,R_{N_t}>k,\,N_t\le a_t\Big]&\le \sum_{k=1}^{+\infty}\,\PP\Big[\Delta_k>\eta t/2,\,M_{a_t}^+\ge k\Big]\\
  &=\prob\big[\Delta_1>\eta t/2\big]\, E\big[M_{a_t}^+\big]\le 2\sqrt{a_t}\,\prob\big[\Delta_1>\eta t/2\big],
  \label{few_coll}
\end{align}
where we have invoked Lemma \ref{moment:Mn} to get the last
bound. Then, for all $t>0$ definition (\ref{def_Phi_secondo})
translates into
\begin{align}
  \nonumber
\Phi_t''&\le 2\sqrt{a_t}\,\prob\big[\Delta_1>\eta t/2\big]+\sum_{k=1}^{+\infty}\,\PP\Bigg[\Delta_k>\eta t/2,\,\bigg|\frac{\mathsf{T}_{N_t-1,k}}{N_t-1}-\mu\bigg|\vee
    \bigg|\frac{\mathsf{T}_{N_t,k}}{N_t}-\mu\bigg|>\epsilon\mu,\,R_{N_t}>k,\,a_t<N_t\le b_t\Bigg]\\
\nonumber
&\le 2\sqrt{a_t}\,\prob\big[\Delta_1>\eta t/2\big]+\sum_{k=1}^{+\infty}\,\PP\Bigg[\Delta_k>\eta t/2,\,\max_{a_t\le n\le b_t}
  \bigg\{\bigg|\frac{\mathsf{T}_{n,k}}{n}-\mu\bigg|\bigg\}>\epsilon\mu,\,M_{b_t}^+\ge k\Bigg].
\end{align}
Since $\{\mathsf{T}_{n,k}\}_{n\in\mathbb{Z}_{\ge 0}}$ is independent
of $\Delta_k$ for any $k$, we can rewrite the latter as
\begin{equation}
\Phi_t'' \le 2\sqrt{a_t}\,\prob\big[\Delta_1>\eta t/2\big]+\mathcal{C}_t\{\epsilon\}\,\prob\big[\Delta_1>\eta t/2\big]
\label{Phi_secondo_1}
\end{equation}
with
\begin{equation}
  \mathcal{C}_t\{\epsilon\}:=
  \sum_{k=1}^{+\infty}\,\PP\Bigg[\max_{\lfloor\epsilon t/\mu\rfloor\le n\le b_t}\bigg\{\bigg|\frac{\mathsf{T}_{n,k}}{n}-\mu\bigg|\bigg\}>\epsilon\mu,\,M_{b_t}^+\ge k\Bigg]
\label{defc}
\end{equation}
for every real numbers $\epsilon\in(0,1)$ and $t>0$. The following
lemma based on Theorem \ref{SLLN_T} shows that
$\mathcal{C}_t\{\epsilon\}$ is negligible with respect to $\sqrt{t}$,
thus confirming that $\mathsf{T}_{n,k}/n$ is close to $\mu$ in
probability irrespective of the behavior of edge $k$. We stress that
it is important here that $n$ goes to infinity when $t$ is sent to
infinity. This is the reason why we have treated separately events
that involve few collisions.
\begin{lemma}
  \label{cboundimp}
   Let $\mathcal{C}_t\{\epsilon\}$ be given by (\ref{defc}) for any
   $\epsilon\in(0,1)$ and $t>0$. Then
  \begin{equation*}
    \lim_{t\uparrow+\infty}\frac{\mathcal{C}_t\{\epsilon\}}{\sqrt{t}}=0.
    \end{equation*}
\end{lemma}

By combining (\ref{Phi_primo_secondo}), (\ref{Phi_primo_3}), and
(\ref{Phi_secondo_1}) together, Lemma \ref{cboundimp} and the limit
$\lim_{t\uparrow+\infty}\prob[\Delta_1>\eta
  t/2]/\prob[\Delta_1>t]=2^\alpha\eta^{-\alpha}$ eventually show that
for each $\epsilon\in(0,1)$
\begin{equation*}
  \limsup_{t\uparrow+\infty}\frac{\Phi_t}{\sqrt{t}\,\prob[\Delta_1>t]}\le \frac{2^{\alpha+1}}{\eta^{\alpha+1}\sqrt{\mu}}
  \bigg(\frac{1}{\sqrt{1-\epsilon}}-\frac{1}{\sqrt{1+\epsilon}}\bigg)+\frac{2^{\alpha+1}}{\eta^\alpha}\sqrt{\frac{\epsilon}{\mu}}.
\end{equation*}
We get (\ref{Phi1}) by sending $\epsilon$ to zero.  It remains to
prove Lemma \ref{cboundimp}.
\begin{proof}[Proof of Lemma \ref{cboundimp}]
  Fix $\epsilon\in(0,1)$ and a positive real number $\theta$. We are
  going to show that
  \begin{equation}
    \limsup_{t\uparrow+\infty}\frac{\mathcal{C}_t\{\epsilon\}}{\sqrt{t}}\le \sqrt{\frac{8}{\mu}}\,\prob\big[\Delta_1>\theta\big].
\label{stima_Ct}
  \end{equation}
Then, the lemma will follow from the arbitrariness of $\theta$.

Set $a_t:=\lfloor\epsilon t/\mu\rfloor$ for brevity. Recalling that
$\mathsf{T}_{n,k}:=T_n-\Delta_kU_{n,k}$ and that $U_{n,k}$ is
non-decreasing with respect to $n$, starting from definition
(\ref{defc}) we have
\begin{align}
\nonumber
\mathcal{C}_t\{\epsilon\}&=\sum_{k=1}^{+\infty}\,\PP\Bigg[
  \max_{a_t\le n\le b_t}\bigg\{\bigg|\frac{T_n}{n}-\frac{\Delta_kU_{n,k}}{n}-\mu\bigg|\bigg\}>\epsilon\mu,\,M_{b_t}^+\ge k\Bigg]\\
\nonumber
&\le\sum_{k=1}^{+\infty}\,\PP\Bigg[\sup_{n\ge a_t}\bigg\{\bigg|\frac{T_n}{n}-\mu\bigg|\bigg\}>\epsilon\mu-\frac{\Delta_kU_{b_t,k}}{a_t},\,M_{b_t}^+\ge k\Bigg].
\end{align}
By distinguishing events where $\Delta_k>\theta$ from events where
$\Delta_k\le\theta$ we realize that
\begin{align}
\nonumber
\mathcal{C}_t\{\epsilon\}&\le\sum_{k=1}^{+\infty}\,\PP\Big[\Delta_k>\theta,\,M_{b_t}^+\ge k\Big]+
\sum_{k=1}^{+\infty}\,\PP\Bigg[\sup_{n\ge a_t}\bigg\{\bigg|\frac{T_n}{n}-\mu\bigg|\bigg\}>\epsilon\mu-\frac{\theta U_{b_t,k}}{a_t},\,M_{b_t}^+\ge k\Bigg]\\
\nonumber
&= E\big[M_{b_t}^+\big]\prob\big[\Delta_1>\theta\big]+
\sum_{k=1}^{+\infty}\,\PP\Bigg[\sup_{n\ge a_t}\bigg\{\bigg|\frac{T_n}{n}-\mu\bigg|\bigg\}>\epsilon\mu-\frac{\theta U_{b_t,k}}{a_t},\,M_{b_t}^+\ge k\Bigg].
\end{align}
By distinguishing events where
$\max_{r\in\mathbb{Z}}\{U_{b_t,r}\}>b_t^{2/3}$ from events where
$\max_{r\in\mathbb{Z}}\{U_{b_t,r}\}\le b_t^{2/3}$ and by invoking
the Cauchy-Schwarz inequality we find
\begin{align}
\nonumber
\mathcal{C}_t\{\epsilon\}&\le E\big[M_{b_t}^+\big]\prob\big[\Delta_1>\theta\big]+
\sum_{k=1}^{+\infty}\,\PP\Bigg[\max_{r\in\mathbb{Z}}\{U_{b_t,r}\}>b_t^{2/3},\,M_{b_t}^+\ge k\Bigg]\\
\nonumber
&+\sum_{k=1}^{+\infty}\,\PP\Bigg[\sup_{n\ge a_t}\bigg\{\bigg|\frac{T_n}{n}-\mu\bigg|\bigg\}>\epsilon\mu-\frac{\theta b_t^{2/3}}{a_t},\,M_{b_t}^+\ge k\Bigg]\\
\nonumber
&=E\big[M_{b_t}^+\big]\prob\big[\Delta_1>\theta\big]+
E\bigg[M_{b_t}^+\,\mathds{1}_{\big\{\max_{r\in\mathbb{Z}}\{U_{b_t,r}\}>b_t^{2/3}\big\}}\bigg]\\
\nonumber
&+\EE\bigg[M_{b_t}^+\,\mathds{1}_{\big\{\sup_{n\ge a_t}\big\{\big|\frac{T_n}{n}-\mu\big|\big\}>\epsilon\mu-\frac{\theta b_t^{2/3}}{a_t}\big\}}\bigg]\\
\nonumber
&\le E\big[M_{b_t}^+\big]\prob\big[\Delta_1>\theta\big]+\sqrt{E\big[(M_{b_t}^+)^2\big]}\sqrt{P\bigg[\max_{r\in\mathbb{Z}}\{U_{b_t,r}\}>b_t^{2/3}\bigg]}\\
\nonumber
&+\sqrt{E\big[(M_{b_t}^+)^2\big]}\sqrt{\PP\Bigg[\sup_{n\ge a_t}\bigg\{\bigg|\frac{T_n}{n}-\mu\bigg|\bigg\}>\epsilon\mu-\frac{\theta b_t^{2/3}}{a_t}\Bigg]}.
\end{align}
Let us observe that $\theta b_t^{2/3}/a_t\le\epsilon\mu/2$ for all
sufficiently large $t$ as both $a_t$ and $b_t$ are of the order of
magnitude of $t$.  At this point, by making use of Lemma
\ref{moment:Mn} for estimating the moments of $M_{b_t}^+$ and of Lemma
\ref{boundU} for estimating the tail probability of
$\max_{r\in\mathbb{Z}}\{U_{b_t,r}\}$, for all sufficiently large $t$
we obtain
\begin{equation*}
\mathcal{C}_t\{\epsilon\}\le 2\sqrt{b_t}\,\prob\big[\Delta_1>\theta\big]+4b_te^{-\frac{b_t^{1/3}}{4}}
+2\sqrt{b_t}\,\sqrt{\PP\Bigg[\sup_{n\ge a_t}\bigg\{\bigg|\frac{T_n}{n}-\mu\bigg|\bigg\}>\epsilon\mu/2\Bigg]}.
\end{equation*}
This bound proves (\ref{stima_Ct}) since $b_t:=\lfloor
2t/\mu\rfloor+1$ and $\lim_{t\uparrow+\infty}\PP[\sup_{n\ge
    a_t}\{|T_n/n-\mu|\}>\epsilon\mu/2]=0$ by Theorem \ref{SLLN_T}.
\end{proof}

\subsection{The limit (\ref{Psi1}) for $\Psi_t$}
\label{dominant_term}

In this section we demonstrate by exact computation the limit
(\ref{Psi1}) for the dominant contribution to the annealed probability
of a large fluctuation. We shall show that
\begin{equation}
  \adjustlimits\lim_{t\uparrow+\infty}\sup_{~y\in[\delta/2,1]~}\!
  \Bigg\{\bigg|\frac{\Psi_t(y)}{\sqrt{t}\,\prob[\Delta_1>t]}-\frac{\Psi_t'(y)}{\sqrt{t}\,\prob[\Delta_1>t]}\bigg|\Bigg\}=0
\label{Psi_Psi_primo}
\end{equation}
and
\begin{equation}
 \adjustlimits\lim_{t\uparrow+\infty}\sup_{~y\in[\delta/2,1]~}\!\Bigg\{\bigg|\frac{\Psi_t'(y)}{\sqrt{t}\,\prob[\Delta_1>t]}-F(y)\bigg|\Bigg\}=0,
\label{Psi_primo_0}
\end{equation}
where $F(y)$ is given by (\ref{defFy}) and $\Psi_t'(y)$ is defined for
each $y\in(0,1]$ by
\begin{equation}
  \Psi_t'(y):=\sqrt{\frac{t}{2\pi\mu}}\int_0^1\sum_{l=0}^{+\infty}\,\prob\Big[(2l)\Delta_1+yt\le\xi t<(2l+2)\Delta_1-yt\Big]\frac{d\xi}{\sqrt{1-\xi}}.
\label{def_Psi_primo}
\end{equation}

The substantial work of the section is the proof of limit
(\ref{Psi_Psi_primo}), whereas limit (\ref{Psi_primo_0}) is a
straightforward consequence of the uniform convergence theorem for
slowly varying functions. The uniform convergence theorem for slowly
varying functions (see \cite{RV}, Theorem 1.2.1) tells us that
\begin{equation}
\sigma_t:=\sup_{z\in[\delta/2,1]}\Bigg\{\bigg|\frac{\ell(zt)}{\ell(t)}-1\bigg|\Bigg\}
\label{UCT}
\end{equation}
goes to zero when $t$ goes to infinity. Let us outline the proof of
(\ref{Psi_primo_0}).  Fix $y\in[\delta/2,1]$. A necessary condition
for the constraint $(2l)\Delta_1+yt\le\xi t<(2l+2)\Delta_1-yt$ to be
satisfied is $\xi>(2l+1)y$, which is understood by observing that
$(2l)\Delta_1+yt\le\xi t<(2l+2)\Delta_1-yt$ requires
$\Delta_1>yt$. Then, starting from (\ref{def_Psi_primo}) we can write
\begin{align}
\nonumber
\Psi_t'(y)&=\sqrt{\frac{t}{2\pi\mu}}\int_0^1\sum_{l=0}^{+\infty}\mathds{1}_{\{\xi>(2l+1)y\}}\,
  \prob\Big[(2l)\Delta_1+yt\le\xi t<(2l+2)\Delta_1-yt\Big]\frac{d\xi}{\sqrt{1-\xi}}\\
  \nonumber
  &=\sqrt{\frac{t}{2\pi\mu}}\int_y^1\sum_{l=0}^{+\infty}\mathds{1}_{\big\{\xi>(2l+1)y\big\}}
  \prob\Big[(2l+2)\Delta_1-yt>\xi t\Big]\frac{d\xi}{\sqrt{1-\xi}}\\
  \nonumber
&-\sqrt{\frac{t}{2\pi\mu}}\int_y^1\sum_{l=1}^{+\infty}\mathds{1}_{\big\{\xi>(2l+1)y\big\}}
\prob\Big[(2l)\Delta_1+yt>\xi t\Big]\frac{d\xi}{\sqrt{1-\xi}}\\
\nonumber
  &=\sqrt{\frac{t}{2\pi\mu}}\int_y^1\sum_{l=0}^{+\infty}\mathds{1}_{\big\{l<\frac{\xi-y}{2y}\big\}}
\prob\bigg[\Delta_1>\frac{\xi+y}{2l+2}t\bigg]\frac{d\xi}{\sqrt{1-\xi}}\\
\nonumber
&-\sqrt{\frac{t}{2\pi\mu}}\int_y^1\sum_{l=1}^{+\infty}\mathds{1}_{\big\{l<\frac{\xi-y}{2y}\big\}}
\prob\bigg[\Delta_1>\frac{\xi-y}{2l}t\bigg]\frac{d\xi}{\sqrt{1-\xi}}.
\end{align}
Notice that the conditions $\xi>(2l+1)y$ and $(2l)\Delta_1+yt>\xi t$
are not compatible when $l=0$.  This identity and (\ref{defFy}) show
that for all $t>0$
\begin{align}
  \nonumber
  \bigg|\Psi_t'(y)-\sqrt{t}\,\prob\big[\Delta_1>t\big]F(y)\bigg|
  &\le\sqrt{\frac{t}{2\pi\mu}}\int_y^1\sum_{l=0}^{+\infty}\mathds{1}_{\big\{l<\frac{\xi-y}{2y}\big\}}
\Bigg|\prob\bigg[\Delta_1>\frac{\xi+y}{2l+2}t\bigg]-\prob\big[\Delta_1>t\big]\bigg(\frac{2l+2}{\xi+y}\bigg)^\alpha\Bigg|\frac{d\xi}{\sqrt{1-\xi}}\\
&+\sqrt{\frac{t}{2\pi\mu}}\int_y^1\sum_{l=1}^{+\infty}\mathds{1}_{\big\{l<\frac{\xi-y}{2y}\big\}}
\Bigg|\prob\bigg[\Delta_1>\frac{\xi-y}{2l}t\bigg]-\prob\big[\Delta_1>t\big]\bigg(\frac{2l}{\xi-y}\bigg)^\alpha\Bigg|\frac{d\xi}{\sqrt{1-\xi}}.
\label{Psi_primo_1}
\end{align}
We now observe that the constraints $0\le l<(\xi-y)/(2y)$ and $\xi\le
1$ entail $\delta/2\le y<(\xi+y)/(2l+2)\le 1$ as
$y\in[\delta/2,1]$. It follows from definition (\ref{UCT}) that $0\le
l<(\xi-y)/(2y)$ and $\xi\le 1$ give
\begin{equation}
  \Bigg|\prob\bigg[\Delta_1>\frac{\xi+y}{2l+2}t\bigg]-\prob\big[\Delta_1>t\big]\bigg(\frac{2l+2}{\xi+y}\bigg)^\alpha\Bigg|
  =\bigg(\frac{2l+2}{\xi+y}\bigg)^\alpha t^{-\alpha}\Bigg|\ell\bigg(\frac{\xi+y}{2l+2}t\bigg)-\ell(t)\Bigg|
  \le \frac{2^\alpha\sigma_t}{\delta^\alpha}\,\prob\big[\Delta_1>t\big].
\label{Psi_primo_2}
\end{equation}
Similarly, $1\le l<(\xi-y)/(2y)$ and $\xi\le 1$ imply
$\delta/2\le y<(\xi-y)/(2l)\le 1$, so that 
\begin{equation}
  \Bigg|\prob\bigg[\Delta_1>\frac{\xi-y}{2l}t\bigg]-\prob\big[\Delta_1>t\big]\bigg(\frac{2l}{\xi-y}\bigg)^\alpha\Bigg|
  =\bigg(\frac{2l}{\xi-y}\bigg)^\alpha t^{-\alpha}\Bigg|\ell\bigg(\frac{\xi-y}{2l}t\bigg)-\ell(t)\Bigg|
  \le\frac{2^\alpha\sigma_t}{\delta^\alpha}\,\prob\big[\Delta_1>t\big].
\label{Psi_primo_3}
\end{equation}
By combining (\ref{Psi_primo_1}) with (\ref{Psi_primo_2}) and
(\ref{Psi_primo_3}) we realize that for every $t>0$
\begin{align}
  \nonumber
  \bigg|\Psi_t'(y)-\sqrt{t}\,\prob\big[\Delta_1>t\big]F(y)\bigg|
&\le \frac{2^\alpha\sigma_t}{\delta^\alpha}\,\prob\big[\Delta_1>t\big]\sqrt{\frac{t}{2\pi\mu}}
\int_y^1\frac{\xi}{y}\frac{d\xi}{\sqrt{1-\xi}}\le \frac{2^{\alpha+1}\sigma_t}{\delta^{\alpha+1}\sqrt{\mu}}\,\sqrt{t}\,\prob\big[\Delta_1>t\big].
\end{align}
The arbitrariness of $y\in[\delta/2,1]$ proves (\ref{Psi_primo_0})
since $\lim_{t\uparrow+\infty}\sigma_t=0$.

Let us move to limit (\ref{Psi_Psi_primo}).  According to definition
(\ref{deftau}) for $R_{N_t}\ge 1$, for each $y\in[\delta/2,1]$ and
$t>0$ we have
\begin{align}
  \nonumber
  \Psi_t(y)&:=\PP\Big[\tau_t>yt,\,R_{N_t}\ge 1,\,N_t\le b_t\Big]\\
\nonumber
&=\PP\bigg[\Delta_{R_{N_t}}>yt,\,\mathds{1}_{\{V_{N_t}=-1\}}(T_{N_t}-t)+\mathds{1}_{\{V_{N_t}=1\}}(t-T_{N_t-1})>yt,\,R_{N_t}\ge 1,\,N_t\le b_t\bigg]\\
\nonumber
&=\sum_{n=1}^{b_t}\,\PP\bigg[\Delta_{R_n}>yt,\,\mathds{1}_{\{V_n=-1\}}(T_n-t)+\mathds{1}_{\{V_n=1\}}(t-T_{n-1})>yt,\,T_{n-1}\le t<T_n,\,R_n\ge 1\bigg],
\end{align}
where we have stressed the fact that $\tau_t>yt$ entails
$\Delta_{R_{N_t}}>yt$. This expression can be simplified by observing
that the conditions
$\mathds{1}_{\{V_n=-1\}}(T_n-t)+\mathds{1}_{\{V_n=1\}}(t-T_{n-1})>yt$
and $T_{n-1}\le t<T_n$ imply $T_{n-1}\le
t-\mathds{1}_{\{V_n=1\}}yt<T_n-yt$ and are implied by $T_{n-1}<
t-\mathds{1}_{\{V_n=1\}}yt<T_n-yt$. Then, for all $y\in[\delta/2,1]$
and $t>0$ we can state that
\begin{equation}
  \Psi_t(y)\le\sum_{n=1}^{b_t}\,\PP\bigg[\Delta_{R_n}>yt,\,T_{n-1}\le t-\mathds{1}_{\{V_n=1\}}yt<T_n-yt,\,R_n\ge 1\bigg]
\label{Psi_upper_0}
\end{equation}
and
\begin{equation}
  \Psi_t(y)\ge\sum_{n=1}^{b_t}\,\PP\bigg[\Delta_{R_n}>yt,\,T_{n-1}<t-\mathds{1}_{\{V_n=1\}}yt<T_n-yt,\,R_n\ge 1\bigg].
\label{Psi_lower_0}
\end{equation}
Bounds (\ref{Psi_upper_0}) and (\ref{Psi_lower_0}) are the starting
point to verify (\ref{Psi_Psi_primo}). Similar to what was done with
$\Phi_t$, the first step is to resolve the dependence on collision
times $\mathsf{T}_{n-1,k}:=T_{n-1}-\Delta_kU_{n-1,k}$ and
$\mathsf{T}_{n,k}:=T_n-\Delta_kU_{n,k}$ deprived of the time that the
particle spends on the longest edge $k$, which is now the current edge
$R_n$. We notice that for every $n\ge 1$
\begin{equation}
  \mathsf{T}_{n-1,R_n}=T_{n-1}-\Delta_{R_n}U_{n-1,R_n}=T_n-\Delta_{R_n}U_{n-1,R_n}=\mathsf{T}_{n,R_n}
  \label{T_coll_Psi}
\end{equation}
as $T_n=T_{n-1}+\Delta_{R_n}$ and $U_{n,R_n}=U_{n-1,R_n}+1$, and we
expect that $\mathsf{T}_{n,R_n}/n$ is close to $\mu$ in probability
when $n$ is large. Since, basically, once the particle has reached the
current edge it spends most of the time on this edge,
$\mathsf{T}_{n,R_n}$ is the time conditional on $N_t=n$ that the
particle needs to reach the current edge from its initial position for
the first time. We are going to establish an upper bound and a lower
bound free from the random variables $\mathsf{T}_{n,R_n}$. After that,
we will able to prove (\ref{Psi_Psi_primo}) by some calculations
involving the simple symmetric random walk.

Let us begin our program by devising an upper bound. Fix a real number
$\epsilon\in(0,1/2)$, which will be kept and sent to zero in the end.
As for $\Phi_t$, we need to isolate and treat separately events that
account for few collisions. Set $a_t:=\lfloor \epsilon t/\mu\rfloor$,
which is a non-negative integer smaller than $b_t:=\lfloor
2t/\mu\rfloor+1$. Starting from (\ref{Psi_upper_0}), for all
$y\in[\delta/2,1]$ and $t>0$ we write
\begin{align}
  \nonumber
  \Psi_t(y)&\le\sum_{n=1}^{a_t}\,\PP\bigg[\Delta_{R_n}>\delta t/2,\,T_{n-1}\le t<T_n,\,R_n\ge 1\bigg]\\
\nonumber
&+\sum_{n=a_t+1}^{b_t}\PP\Bigg[\Delta_{R_n}>\delta t/2,\,\bigg|\frac{\mathsf{T}_{n,R_n}}{n}-\mu\bigg|>\epsilon\mu,\,T_{n-1}\le t<T_n,\,R_n\ge 1\Bigg]\\
&+\sum_{n=a_t+1}^{+\infty}\PP\Bigg[\Delta_{R_n}>yt,\,\bigg|\frac{\mathsf{T}_{n,R_n}}{n}-\mu\bigg|\le\epsilon\mu,\,T_{n-1}\le t-\mathds{1}_{\{V_n=1\}}yt<T_n-yt,
  \,R_n\ge 1\Bigg].
\label{Psi_upper_1}
\end{align}
Notice that we have dropped the upper bound $b_t$ in the last term.
The same arguments we used for obtaining (\ref{few_coll}) show that
\begin{align}
  \nonumber
\sum_{n=1}^{a_t}\,\PP\bigg[\Delta_{R_n}>\delta t/2,\,T_{n-1}\le t<T_n,\,R_n\ge 1\bigg]&=\sum_{k=1}^{+\infty}\,\PP\Big[\Delta_k>\delta t/2,\,R_{N_t}=k,\,N_t\le a_t\Big]\\
  \nonumber
  &\le\sum_{k=1}^{+\infty}\,\PP\Big[\Delta_k>\delta t/2,\,M_{a_t}^+\ge k\Big]=2\sqrt{a_t}\,\prob\big[\Delta_1>\delta t/2\big].
\end{align}
The same arguments we used for getting at (\ref{Phi_secondo_1}) give
\begin{align}
  \nonumber
  &\sum_{n=a_t+1}^{b_t}\PP\Bigg[\Delta_{R_n}>\delta t/2,\,\bigg|\frac{\mathsf{T}_{n,R_n}}{n}-\mu\bigg|>\epsilon\mu,\,T_{n-1}\le t<T_n,\,R_n\ge 1\Bigg]\\
  \nonumber
  &=\sum_{k=1}^{+\infty}\,\PP\Bigg[\Delta_k>\delta t/2,\,\bigg|\frac{\mathsf{T}_{N_t,k}}{N_t}-\mu\bigg|>\epsilon\mu,\,R_{N_t}=k,\,a_t<N_t\le b_t\Bigg]\\
  \nonumber
  &\le\sum_{k=1}^{+\infty}\,\PP\Bigg[\Delta_k>\delta t/2,\,\max_{a_t\le n\le b_t}
  \bigg\{\bigg|\frac{\mathsf{T}_{n,k}}{n}-\mu\bigg|\bigg\}>\epsilon\mu,\,M_{b_t}^+\ge k\Bigg]=\mathcal{C}_t\{\epsilon\}\,\prob\big[\Delta_1>\delta t/2\big]
\end{align}
with $\mathcal{C}_t\{\epsilon\}$ as in (\ref{defc}). Then, the upper
bound (\ref{Psi_upper_1}) yields for every $y\in[\delta/2,1]$ and
$t>0$
\begin{align}
  \nonumber
  \Psi_t(y)&\le 2\sqrt{a_t}\,\prob\big[\Delta_1>\delta t/2\big]+\mathcal{C}_t\{\epsilon\}\,\prob\big[\Delta_1>\delta t/2\big]\\
  &+\sum_{n=a_t+1}^{+\infty}\PP\Bigg[\Delta_{R_n}>yt,\,\bigg|\frac{\mathsf{T}_{n,R_n}}{n}-\mu\bigg|\le\epsilon\mu,\,T_{n-1}\le t-\mathds{1}_{\{V_n=1\}}yt<T_n-yt,
    \,R_n\ge 1\Bigg].
\label{Psi_upper_2}
\end{align}
We are now in the position to replace $\mathsf{T}_{n,R_n}/n$ with
$\mu$ in the last term. The constraint $T_{n-1}\le
t-\mathds{1}_{\{V_n=1\}}yt<T_n-yt$ reads
$\mathsf{T}_{n,R_n}+\Delta_{R_n}U_{n-1,R_n}\le
t-\mathds{1}_{\{V_n=1\}}yt<\mathsf{T}_{n,R_n}+\Delta_{R_n}U_{n,R_n}-yt$
as $\mathsf{T}_{n-1,R_n}=\mathsf{T}_{n,R_n}$ by (\ref{T_coll_Psi}).
Conditional on $|\mathsf{T}_{n,R_n}/n-\mu|\le\epsilon\mu$, the latter
entails
\begin{equation}
\frac{t-\mathsf{J}_{n,yt}}{(1+\epsilon)\mu}<n\le \frac{t-\mathsf{I}_{n,yt}}{(1-\epsilon)\mu},
\label{n_constraint_1}
\end{equation}
where $\mathsf{I}_{n,z}$ and $\mathsf{J}_{n,z}$ denote for brevity the
random variables defined for $n\ge 1$ and $z\in\Rl$ by
\begin{equation*}
\mathsf{I}_{n,z}:=\Delta_{R_n}U_{n-1,R_n}+\mathds{1}_{\{V_n=1\}}z
\end{equation*}
and
\begin{equation*}
\mathsf{J}_{n,z}:=\Delta_{R_n}U_{n,R_n}-\mathds{1}_{\{V_n=-1\}}z.
\end{equation*}
The constraint (\ref{n_constraint_1}) in turn gives
$n\in(p_{t-\mathsf{J}_{n,yt}},q_{t-\mathsf{I}_{n,yt}}]$ for positive
  $n$ with $p_z\ge 0$ and $q_z\ge p_z$ as in (\ref{def:pz}) and
  (\ref{def:qz}), respectively. Thus, we can conclude that for any
  $y\in[\delta/2,1]$ and $t>0$
\begin{align}
  \nonumber
  \Psi_t(y)&\le 2\sqrt{a_t}\,\prob\big[\Delta_1>\delta t/2\big]+\mathcal{C}_t\{\epsilon\}\,\prob\big[\Delta_1>\delta t/2\big]\\
  \nonumber
  &+\sum_{n=a_t+1}^{+\infty}\PP\Big[\Delta_{R_n}>yt,\,n\in\big(p_{t-\mathsf{J}_{n,yt}},q_{t-\mathsf{I}_{n,yt}}\big],\,R_n\ge 1\Big]\\
    \nonumber
    &=2\sqrt{a_t}\,\prob\big[\Delta_1>\delta t/2\big]+\mathcal{C}_t\{\epsilon\}\,\prob\big[\Delta_1>\delta t/2\big]\\
  &+\sum_{n=a_t+1}^{+\infty}\PP\bigg[\Delta_{R_n}>yt,\,n\in\big(p_{t-\Delta_{R_n}U_{n,R_n}+\mathds{1}_{\{V_n=-1\}}yt},q_{t-\Delta_{R_n}U_{n-1,R_n}-\mathds{1}_{\{V_n=1\}}yt}\big],\,R_n\ge 1\bigg].
\label{Psi_upper_3}
\end{align}
This bound is the upper bound free from $\mathsf{T}_{n,R_n}$ we were
looking for.

Starting from (\ref{Psi_lower_0}) we can get a lower bound similar to
(\ref{Psi_upper_3}) by exchanging $p_z$ and $q_z$. With this purpose,
we claim that the conditions
$n\in(q_{t-\mathsf{J}_{n,yt}},p_{t-\mathsf{I}_{n,yt}})$ and
$|\mathsf{T}_{n,R_n}/n-\mu|\le\epsilon\mu$ imply
$T_{n-1}<t-\mathds{1}_{\{V_n=1\}}yt<T_n-yt$, so that for all
$y\in[\delta/2,1]$ and $t>0$
\begin{equation}
  \Psi_t(y)\ge\sum_{n=a_t+1}^{b_t}\PP\Bigg[\Delta_{R_n}>yt,\,n\in\big(q_{t-\mathsf{J}_{n,yt}},p_{t-\mathsf{I}_{n,yt}}\big),
    \,\bigg|\frac{\mathsf{T}_{n,R_n}}{n}-\mu\bigg|\le\epsilon\mu, \,R_n\ge 1\Bigg].
\label{Psi_lower_1}
\end{equation}
In fact, if there exists an integer $n\ge 1$ such that
$n\in(q_{t-\mathsf{J}_{n,yt}},p_{t-\mathsf{I}_{n,yt}})$, then $n$ must
satisfy
\begin{equation}
\frac{t-\Delta_{R_n} U_{n,R_n}+\mathds{1}_{\{V_n=-1\}}yt}{(1-\epsilon)\mu}<n<\frac{t-\Delta_{R_n} U_{n-1,R_n}-\mathds{1}_{\{V_n=1\}}yt}{(1+\epsilon)\mu}
\label{n_constraint_2}
\end{equation}
since the constraint $1\le n<p_{t-\mathsf{I}_{n,yt}}$ requires
$t-\mathsf{I}_{n,yt}=t-\Delta_{R_n}
U_{n-1,R_n}-\mathds{1}_{\{V_n=1\}}yt>0$. The condition
(\ref{n_constraint_2}) yields
$T_{n-1}<t-\mathds{1}_{\{V_n=1\}}yt<T_n-yt$ when it is combined with
$|\mathsf{T}_{n,R_n}/n-\mu|\le\epsilon\mu$ because
$\mathsf{T}_{n-1,R_n}=\mathsf{T}_{n,R_n}$ by (\ref{T_coll_Psi}).

To make (\ref{Psi_lower_1}) appear as (\ref{Psi_upper_3}) we now need
to eliminate the constraint
$|\mathsf{T}_{n,R_n}/n-\mu|\le\epsilon\mu$. We observe that the
condition (\ref{n_constraint_2}) also entails $\mu(n-1)+\Delta_{R_n}
U_{n-1,R_n}\le t<\mu n+\Delta_{R_n} U_{n,R_n}$. Then, by thinking of
the non-decreasing sequence $\{\mu n+\Delta_k
U_{n,k}\}_{n\in\mathbb{Z}_{\ge 0}}$ as a sequence of renewal times for
a given $k\ge 1$, we can deduce that
\begin{align}
  \nonumber
  &\sum_{n=a_t+1}^{b_t}\PP\Bigg[\Delta_{R_n}>\delta t/2,\,n\in\big(q_{t-\mathsf{J}_{n,yt}},p_{t-\mathsf{I}_{n,yt}}\big),\,\bigg|\frac{\mathsf{T}_{n,R_n}}{n}-\mu\bigg|>\epsilon\mu,
    \,R_n\ge 1\Bigg]\\
  \nonumber
  &\le\sum_{k=1}^{+\infty}\sum_{n=1}^{b_t}\,\PP\Bigg[\Delta_k>\delta t/2,\,\mu(n-1)+\Delta_kU_{n-1,k}\le t<\mu n+\Delta_kU_{n,k},\,
  \bigg|\frac{\mathsf{T}_{n,k}}{n}-\mu\bigg|>\epsilon\mu,\,R_n=k\Bigg]\\
\nonumber
&\le\sum_{k=1}^{+\infty}\sum_{n=1}^{b_t}\,\PP\Bigg[\Delta_k>\delta t/2,\,\mu(n-1)+\Delta_kU_{n-1,k}\le t<\mu n+\Delta_kU_{n,k},\,
  \sup_{a_t\le l\le b_t}\bigg\{\bigg|\frac{\mathsf{T}_{l,k}}{l}-\mu\bigg|\bigg\}>\epsilon\mu,\,M_{b_t}^+\ge k\Bigg]\\
\nonumber
&\le\sum_{k=1}^{+\infty}\,\PP\Bigg[\Delta_k>\delta t/2,\,\sup_{a_t\le l\le b_t}\bigg\{\bigg|\frac{\mathsf{T}_{l,k}}{l}-\mu\bigg|\bigg\}>\epsilon\mu,\,M_{b_t}^+\ge k\Bigg]
=\mathcal{C}_t\{\epsilon\}\,\prob\big[\Delta_1>\delta t/2\big],
\end{align}
where we have again recovered $\mathcal{C}_t\{\epsilon\}$ defined by
(\ref{defc}). Thus, (\ref{Psi_lower_1}) shows that 
\begin{align}
  \nonumber
  \Psi_t(y)&\ge\sum_{n=a_t+1}^{b_t}\PP\Big[\Delta_{R_n}>yt,\,n\in\big(q_{t-\mathsf{J}_{n,yt}},p_{t-\mathsf{I}_{n,yt}}\big),\,R_n\ge 1\Big]\\
  \nonumber
  &-\sum_{n=a_t+1}^{b_t}\PP\Bigg[\Delta_{R_n}>\delta t/2,\,n\in\big(q_{t-\mathsf{J}_{n,yt}},p_{t-\mathsf{I}_{n,yt}}\big),\,\bigg|\frac{\mathsf{T}_{n,R_n}}{n}-\mu\bigg|>\epsilon\mu,
    \,R_n\ge 1\Bigg]\\
  \nonumber
  &\ge\sum_{n=a_t+1}^{b_t}\PP\Big[\Delta_{R_n}>yt,\,n\in\big(q_{t-\mathsf{J}_{n,yt}},p_{t-\mathsf{I}_{n,yt}}\big),\,R_n\ge 1\Big]-
  \mathcal{C}_t\{\epsilon\}\,\prob\big[\Delta_1>\delta t/2\big].
\end{align}
Since the constraint $n<p_{t-\mathsf{I}_{n,yt}}$ implies $n\le
t/\mu\le b_t:=\lfloor 2t/\mu\rfloor+1$, we can drop the upper bound in
the sum over $n$ and state that for every $y\in[\delta/2,1]$ and $t>0$
\begin{align}
  \nonumber
  \Psi_t(y)&\ge\sum_{n=a_t+1}^{+\infty}\PP\bigg[\Delta_{R_n}>yt,\,n\in\big(q_{t-\Delta_{R_n}U_{n,R_n}+\mathds{1}_{\{V_n=-1\}}yt},p_{t-\Delta_{R_n}U_{n-1,R_n}-\mathds{1}_{\{V_n=1\}}yt}\big),
    \,R_n\ge 1\bigg]\\
  &-\mathcal{C}_t\{\epsilon\}\,\prob\big[\Delta_1>\delta t/2\big].
\label{Psi_lower_2}
\end{align}
This lower bound is free from $\mathsf{T}_{n,R_n}$ and conclude the
first part of our program. The term $\mathcal{C}_t\{\epsilon\}$ will
be later tackled by means of Lemma \ref{cboundimp}.

The second part of the program aims to verify (\ref{Psi_Psi_primo}) on
the basis of bounds (\ref{Psi_upper_3}) and (\ref{Psi_lower_2}) by some
explicit calculations. Let us use the identity
$U_{n,R_n}=U_{n-1,R_n}+1$ to rewrite the upper bound
(\ref{Psi_upper_3}) as
\begin{align}
  \nonumber
  \Psi_t(y)&\le 2\sqrt{a_t}\,\prob\big[\Delta_1>\delta t/2\big]+\mathcal{C}_t\{\epsilon\}\,\prob\big[\Delta_1>\delta t/2\big]\\
\nonumber
&+\sum_{n=a_t+1}^{+\infty}\sum_{k=1}^{+\infty}\PP\bigg[\Delta_k>yt,\,n\in\big(p_{t-\Delta_k(U_{n-1,k}+1)+\mathds{1}_{\{V_n=-1\}}yt},
  q_{t-\Delta_kU_{n-1,k}-\mathds{1}_{\{V_n=1\}}yt}\big],\,R_n=k\bigg].
\end{align}
From here, Fubini's theorem and the fact that $\Delta_k$ is
distributed as $\Delta_1$ give for any $y\in[\delta/2,1]$ and $t>0$
\begin{align}
  \nonumber
  \Psi_t(y)&\le 2\sqrt{a_t}\,\prob\big[\Delta_1>\delta t/2\big]+\mathcal{C}_t\{\epsilon\}\,\prob\big[\Delta_1>\delta t/2\big]\\
    &+\Ex\Bigg[\sum_{v\in\mathbb{Z}_2}\sum_{u=0}^{+\infty}\sum_{n=a_t+1}^{+\infty}
      \mathds{1}_{\big\{\Delta_1>yt,\,n\in\big(p_{t-\Delta_1(u+1)+\mathds{1}_{\{v=-1\}}yt},q_{t-\Delta_1u-\mathds{1}_{\{v=1\}}yt}\big]\big\}}\,\mathcal{P}_{u,v}\{n\}\Bigg],
\label{Psi_upper_4}
\end{align}
where for all integers $n\ge 1$, $u\ge 0$, and $v\in\mathbb{Z}_2$ we
have set
\begin{equation}
\mathcal{P}_{u,v}\{n\}:=\sum_{k=1}^{+\infty}\,P\Big[U_{n-1,k}=u,R_n=k,V_n=v\Big]=P\Big[U_{n-1,R_n}=u,R_n\ge 1,V_n=v\Big].
\label{Pcalnuv}
\end{equation}
Similarly, the lower bound (\ref{Psi_lower_2}) yields for every
$y\in[\delta/2,1]$ and $t>0$
\begin{align}
  \nonumber
  \Psi_t(y)&\ge\Ex\Bigg[\sum_{v\in\mathbb{Z}_2}\sum_{u=0}^{+\infty}
    \sum_{n=a_t+1}^{+\infty}\mathds{1}_{\big\{\Delta_1>yt,\,n\in\big(q_{t-\Delta_1(u+1)+\mathds{1}_{\{v=-1\}}yt},p_{t-\Delta_1u-\mathds{1}_{\{v=1\}}yt}\big)\big\}}\,\mathcal{P}_{u,v}\{n\}\Bigg]\\
  &-\mathcal{C}_t\{\epsilon\}\,\prob\big[\Delta_1>\delta t/2\big].
\label{Psi_lower_3}
\end{align}
The following lemma, which only involves the simple symmetric random
walk, provides a formula for $\mathcal{P}_{u,v}\{n\}$ and is proved in
Appendix \ref{proof:sommeUv}.
\begin{lemma}
  \label{sommeUv}
  Let $\mathcal{P}_{u,v}\{n\}$ be defined by (\ref{Pcalnuv}) for every
  integers $n\ge 1$, $u\ge 0$, and $v\in\mathbb{Z}_2$. Then
\begin{equation*}
  \mathcal{P}_{u,v}\{n\}=
  \begin{cases}
    \frac{1}{2^n}{n-1\choose\lfloor\frac{n+u}{2}\rfloor} & \mbox{if $u+\mathds{1}_{\{v=1\}}$ is odd},\\
    0 & \mbox{otherwise}.
  \end{cases}
  \end{equation*}
\end{lemma}

The expression of $\mathcal{P}_{u,v}\{n\}$ can be further simplified
when $u$ is much smaller than $n$. In fact, the inequalities
$\sqrt{2\pi l}\,l^le^{-l}\le l!\le\sqrt{2\pi l}\,l^le^{-l+\frac{1}{12
    l}}$ valid for all positive integers $l$ \cite{Stirling} show that
there exists an integer $n_o\ge 0$ such that
\begin{equation*}
  \frac{1-\epsilon}{\sqrt{2\pi n}}\le\frac{1}{2^n}{n-1\choose\lfloor\frac{n+u}{2}\rfloor}\le \frac{1+\epsilon}{\sqrt{2\pi n}}
\end{equation*}
for all $n>n_o$ and $u<2/\delta$. On the other hand, according to
(\ref{def:pz}) and (\ref{def:qz}), the conditions $1\le n\le
q_{t-\Delta_1u-\mathds{1}_{\{v=1\}}yt}$ in (\ref{Psi_upper_4}) and
$1\le n<p_{t-\Delta_1u-\mathds{1}_{\{v=1\}}yt}$ in (\ref{Psi_lower_3})
require $t-\Delta_1u-\mathds{1}_{\{v=1\}}yt>0$ to be satisfied. The
latter gives $u<\delta/2$ when combined with $\Delta_1>yt$ and
$y\ge\delta/2$. Let $t_o$ be a positive real number such that
$a_t:=\lfloor \epsilon t/\mu\rfloor>n_o$ for $t>t_o$. If $t>t_o$,
$n>a_t$, and $u<2/\delta$, then Lemma \ref{sommeUv} implies
$\frac{1-\epsilon}{\sqrt{2\pi
    n}}\le\mathcal{P}_{u,v}\{n\}\le\frac{1+\epsilon}{\sqrt{2\pi n}}$
or $\mathcal{P}_{u,v}\{n\}=0$ depending on whether
$u+\mathds{1}_{\{v=1\}}$ is odd or even. Let us plug these estimates
in (\ref{Psi_upper_4}) and (\ref{Psi_lower_3}).  Since
$u+\mathds{1}_{\{v=1\}}$ is odd if and only if there exists an integer
$l\ge 0$ such that $u=2l+\mathds{1}_{\{v=-1\}}$, bound
(\ref{Psi_upper_4}) yields for all $y\in[\delta/2,1]$ and $t>t_o$
\begin{align}
  \nonumber
  \Psi_t(y)&\le 2\sqrt{a_t}\,\prob\big[\Delta_1>\delta t/2\big]+\mathcal{C}_t\{\epsilon\}\,\prob\big[\Delta_1>\delta t/2\big]\\
    &+\Ex\Bigg[\mathds{1}_{\{\Delta_1>yt\}}\sum_{v\in\mathbb{Z}_2}\sum_{l=0}^{+\infty}\sum_{n=a_t+1}^{+\infty}
    \mathds{1}_{\{n\in(p_{t-\mathsf{K}_{l,v,yt}},q_{t-\mathsf{H}_{l,v,yt}}]\}}\,\frac{1+\epsilon}{\sqrt{2\pi n}}\Bigg],
\label{Psi_upper_5}
\end{align}
where $\mathsf{H}_{l,v,z}$ and $\mathsf{K}_{l,v,z}$ are defined for
integers $l\ge 0$ and $v\in\mathbb{Z}_2$ and a real number $z$ by
\begin{equation*}
\mathsf{H}_{l,v,z}:=(2l+1)\Delta_1-\mathds{1}_{\{v=1\}}(\Delta_1-z)
\end{equation*}
and
\begin{equation*}
\mathsf{K}_{l,v,z}:=(2l+1)\Delta_1+\mathds{1}_{\{v=-1\}}(\Delta_1-z).
\end{equation*}
We stress that $0\le \mathsf{H}_{l,v,z}\le \mathsf{K}_{l,v,z}$
conditional on $\Delta_1>z>0$.  Similarly, bound (\ref{Psi_lower_3})
entails that for each $y\in[\delta/2,1]$ and $t>t_o$
\begin{equation}
  \Psi_t(y)\ge\Ex\Bigg[\mathds{1}_{\{\Delta_1>yt\}}\sum_{v\in\mathbb{Z}_2}\sum_{l=0}^{+\infty}\sum_{n=a_t+1}^{+\infty}
    \mathds{1}_{\{n\in(q_{t-\mathsf{K}_{l,v,yt}},p_{t-\mathsf{H}_{l,v,yt}})\}}\,\frac{1-\epsilon}{\sqrt{2\pi n}}\Bigg]-\mathcal{C}_t\{\epsilon\}\,\prob\big[\Delta_1>\delta t/2\big].
\label{Psi_lower_4}
\end{equation}

Let us now carry out the sum over $n$ in (\ref{Psi_upper_5}) and
(\ref{Psi_lower_4}). We tackle the upper bound (\ref{Psi_upper_5})
first. Bear in mind the expressions (\ref{def:pz}) and (\ref{def:qz})
of $p_z$ and $q_z$.  We notice that the condition $1\le n\le
q_{t-\mathsf{H}_{l,v,yt}}$ is fulfilled only if
$\mathsf{H}_{l,v,yt}<t$, which yields $l<1/\delta$ if $\Delta_1>yt$
with $y\ge \delta/2$. At the same time, we have
$\sum_{n=p+1}^q1/\sqrt{n}\le 1+2\sqrt{q}-2\sqrt{p+1}$ for every
integers $0\le p\le q$ as a consequence of the inequality
$1/\sqrt{n}\le 2\sqrt{n}-2\sqrt{n-1}$ valid for $n\ge 1$.  Thus, if
$\Delta_1>yt\ge \delta t/2$, then we can state that
\begin{align}
  \nonumber
&\sum_{v\in\mathbb{Z}_2}\sum_{l=0}^{+\infty}\sum_{n=a_t+1}^{+\infty}
\mathds{1}_{\{n\in(p_{t-\mathsf{K}_{l,v,yt}},q_{t-\mathsf{H}_{l,v,yt}}]\}}\,\frac{1+\epsilon}{\sqrt{2\pi n}}\\
  \nonumber
  &\le
  \sum_{v\in\mathbb{Z}_2}\sum_{l=0}^{\lfloor 1/\delta\rfloor}\sum_{n=1}^{+\infty}\mathds{1}_{\{n\in(p_{t-\mathsf{K}_{l,v,yt}},q_{t-\mathsf{H}_{l,v,yt}}]\}}\,\frac{1+\epsilon}{\sqrt{2\pi n}}\\
    \nonumber
    &\le \frac{1+\epsilon}{\sqrt{2\pi}} \sum_{v\in\mathbb{Z}_2}\sum_{l=0}^{\lfloor 1/\delta\rfloor}\Big[1+2\sqrt{q_{t-\mathsf{H}_{l,v,yt}}}-2\sqrt{p_{t-\mathsf{K}_{l,v,yt}}+1}\Big]\\
    \nonumber
    &\le\frac{1+\epsilon}{\sqrt{2\pi}} \sum_{v\in\mathbb{Z}_2}\sum_{l=0}^{\lfloor 1/\delta\rfloor}\Bigg[1+2\sqrt{\frac{(t-\mathsf{H}_{l,v,yt})\vee 0}{(1-\epsilon)\mu}}
      -2\sqrt{\frac{(t-\mathsf{K}_{l,v,yt})\vee 0}{(1+\epsilon)\mu}}~\Bigg].
\end{align}
At this point, by making use of the facts that $\epsilon<1/2$,
$(1+\epsilon)/\sqrt{1-\epsilon}\le 1+3\epsilon$, and
$\sqrt{1+\epsilon}>1$, as well as the fact that
$\mathsf{K}_{l,v,yt}\ge 0$ conditional on $\Delta_1>yt$, we get that
if $\Delta_1>yt\ge \delta t/2$, then
\begin{align}
  \nonumber
  &\sum_{v\in\mathbb{Z}_2}\sum_{l=0}^{+\infty}\sum_{n=a_t+1}^{+\infty}\mathds{1}_{\{n\in(p_{t-\mathsf{K}_{l,v,yt}},q_{t-\mathsf{H}_{l,v,yt}}]\}}\,\frac{1+\epsilon}{\sqrt{2\pi n}}\\
    \nonumber
    &\le \sum_{v\in\mathbb{Z}_2}\sum_{l=0}^{\lfloor 1/\delta\rfloor}\Bigg[1+3\epsilon\sqrt{\frac{t}{\mu}}+\sqrt{\frac{2(t-\mathsf{H}_{l,v,yt})\vee 0}{\pi\mu}}
      -\sqrt{\frac{2(t-\mathsf{K}_{l,v,yt})\vee 0}{\pi\mu}}~\Bigg]\\
    \nonumber
    &=\frac{2}{\delta}+\frac{6\epsilon}{\delta}\sqrt{\frac{t}{\mu}}+\sum_{l=0}^{+\infty}\Bigg[\sqrt{\frac{2(t-\mathsf{H}_{l,1,yt})\vee 0}{\pi\mu}}
      -\sqrt{\frac{2(t-\mathsf{K}_{l,-1,yt})\vee 0}{\pi\mu}}~\Bigg].
\end{align}
For the last step we make use of the identity
\begin{equation}
  \sqrt{\frac{2(t-w)\vee 0}{\pi\mu}}-\sqrt{\frac{2(t-z)\vee 0}{\pi\mu}}
  =\sqrt{\frac{t}{2\pi\mu}}\int_0^1\mathds{1}_{\{w\le\xi t<z\}}\frac{d\xi}{\sqrt{1-\xi}}
\label{identity}
\end{equation}
valid for real numbers $w$ and $z$ such that $0\le w\le z$, as one can
easily verify. As $\mathsf{H}_{l,1,yt}:=(2l)\Delta_1+yt$ and
$\mathsf{K}_{l,-1,yt}:=(2l+2)\Delta_1-yt$, the hypothesis
$\Delta_1>yt$ implies $0\le \mathsf{H}_{l,1,yt}\le
\mathsf{K}_{l,-1,yt}$. Then, the above identity gives for
$\Delta_1>yt\ge \delta t/2$
\begin{align}
  \nonumber
  \sum_{v\in\mathbb{Z}_2}\sum_{l=0}^{+\infty}\sum_{n=a_t+1}^{+\infty}\mathds{1}_{\{n\in(p_{t-\mathsf{K}_{l,v,yt}},q_{t-\mathsf{H}_{l,v,yt}}]\}}\,\frac{1+\epsilon}{\sqrt{2\pi n}}
    &\le \frac{2}{\delta}+\frac{6\epsilon}{\delta}\sqrt{\frac{t}{\mu}}\\
    &+\sqrt{\frac{t}{2\pi\mu}}\,\sum_{l=0}^{+\infty}\int_0^1\mathds{1}_{\{(2l)\Delta_1+yt\le\xi t<(2l+2)\Delta_1-yt\}}\frac{d\xi}{\sqrt{1-\xi}}.
\label{Psi_upper_6}
\end{align}
By plugging bound (\ref{Psi_upper_6}) in (\ref{Psi_upper_5}), Fubini's
theorem finally allows to obtain for $y\in[\delta/2,1]$ and $t>t_o$
\begin{align}
  \nonumber
  \Psi_t(y)&\le 2\sqrt{a_t}\,\prob\big[\Delta_1>\delta t/2\big]+\mathcal{C}_t\{\epsilon\}\,\prob\big[\Delta_1>\delta t/2\big]+
  \bigg(\frac{2}{\delta}+\frac{6\epsilon}{\delta}\sqrt{\frac{t}{\mu}}\bigg)\prob\big[\Delta_1>\delta t/2\big]\\
  \nonumber
    &+\Ex\Bigg[\sqrt{\frac{t}{2\pi\mu}}\,\sum_{l=0}^{+\infty}\int_0^1\mathds{1}_{\{(2l)\Delta_1+yt\le\xi t<(2l+2)\Delta_1-yt\}}\frac{d\xi}{\sqrt{1-\xi}}\Bigg]\\
  &= 2\sqrt{a_t}\,\prob\big[\Delta_1>\delta t/2\big]+\mathcal{C}_t\{\epsilon\}\,\prob\big[\Delta_1>\delta t/2\big]+
  \bigg(\frac{2}{\delta}+\frac{6\epsilon}{\delta}\sqrt{\frac{t}{\mu}}\bigg)\prob\big[\Delta_1>\delta t/2\big]+\Psi'_t(y),
  \label{Psi_upper_7}
\end{align}
with $\Psi'_t(y)$ as in (\ref{def_Psi_primo}). We stress that the
condition $\Delta_1>yt$ in the expectation can be dropped in the
transition from (\ref{Psi_upper_5}) to (\ref{Psi_upper_7}) because if
$(2l)\Delta_1+yt\le\xi t<(2l+2)\Delta_1-yt$ for some $l\ge 0$ and
$\xi\in[0,1]$, then automatically $\Delta_1>yt$. This remark is
important to obtain a lower bound for $\Psi_t(y)$ similar to
(\ref{Psi_upper_7}).

Let us move to the lower bound (\ref{Psi_lower_4}). We have
$\sum_{n=1}^{a_t}1/\sqrt{n}\le 2\sqrt{a_t}$ since $1/\sqrt{n}\le
2\sqrt{n}-2\sqrt{n-1}$ for $n\ge 1$. We also have
$\sum_{n=q+1}^{p-1}1/\sqrt{n}\ge 2\sqrt{p+1}-2\sqrt{q}-3$ for all
integers $p\ge 0$ and $q\ge 0$ due to the inequality $1/\sqrt{n}\ge
2\sqrt{n+1}-2\sqrt{n}$ valid for $n\ge 1$. Then
\begin{align}
  \nonumber
&\sum_{v\in\mathbb{Z}_2}\sum_{l=0}^{+\infty}\sum_{n=a_t+1}^{+\infty}
\mathds{1}_{\{n\in(q_{t-\mathsf{K}_{l,v,yt}},p_{t-\mathsf{H}_{l,v,yt}})\}}\,\frac{1-\epsilon}{\sqrt{2\pi n}}\\
\nonumber
&\ge\sum_{v\in\mathbb{Z}_2}\sum_{l=0}^{\lfloor 1/\delta\rfloor}\sum_{n=1}^{+\infty}\mathds{1}_{\{n\in(q_{t-\mathsf{K}_{l,v,yt}},p_{t-\mathsf{H}_{l,v,yt}})\}}\,\frac{1-\epsilon}{\sqrt{2\pi n}}
  -\sum_{v\in\mathbb{Z}_2}\sum_{l=0}^{\lfloor 1/\delta\rfloor}\sum_{n=1}^{a_t}\frac{1-\epsilon}{\sqrt{2\pi n}}\\
    \nonumber
    &\ge \frac{1-\epsilon}{\sqrt{2\pi}} \sum_{v\in\mathbb{Z}_2}\sum_{l=0}^{\lfloor 1/\delta\rfloor}\Big[2\sqrt{p_{t-\mathsf{H}_{l,v,yt}}+1}-2\sqrt{q_{t-\mathsf{K}_{l,v,yt}}}-3\Big]
    -\frac{2}{\delta}\sqrt{a_t}\\
    \nonumber
    &\ge\frac{1-\epsilon}{\sqrt{2\pi}} \sum_{v\in\mathbb{Z}_2}\sum_{l=0}^{\lfloor 1/\delta\rfloor}\Bigg[2\sqrt{\frac{(t-\mathsf{H}_{l,v,yt})\vee 0}{(1+\epsilon)\mu}}
      -2\sqrt{\frac{(t-\mathsf{K}_{l,v,yt})\vee 0}{(1-\epsilon)\mu}}-3\Bigg]-\frac{2}{\delta}\sqrt{a_t}.
\end{align}
At this point, by making use of the facts that
$(1-\epsilon)/\sqrt{1+\epsilon}\ge 1-2\epsilon$ as $\epsilon<1/2$,
$\sqrt{1-\epsilon}<1$, and $1-\epsilon<1$, as well as the fact that
$\mathsf{H}_{l,v,yt}\ge 0$ conditional on $\Delta_1>yt$, we get that
if $\Delta_1>yt$, then
\begin{align}
  \nonumber
  &\sum_{v\in\mathbb{Z}_2}\sum_{l=0}^{+\infty}\sum_{n=a_t+1}^{+\infty}\mathds{1}_{\{n\in(q_{t-\mathsf{K}_{l,v,yt}},p_{t-\mathsf{H}_{l,v,yt}})\}}\,\frac{1-\epsilon}{\sqrt{2\pi n}}\\
  \nonumber
    &\ge \sum_{v\in\mathbb{Z}_2}\sum_{l=0}^{\lfloor 1/\delta\rfloor}\Bigg[\sqrt{\frac{2(t-\mathsf{H}_{l,v,yt})\vee 0}{\pi\mu}}
      -\sqrt{\frac{2(t-\mathsf{K}_{l,v,yt})\vee 0}{\pi\mu}}-2-2\epsilon\sqrt{\frac{t}{\mu}}~\Bigg]-\frac{2}{\delta}\sqrt{a_t}\\
    \nonumber
    &=\sum_{l=0}^{\lfloor 1/\delta\rfloor}\Bigg[\sqrt{\frac{2(t-\mathsf{H}_{l,1,yt})\vee 0}{\pi\mu}}
      -\sqrt{\frac{2(t-\mathsf{K}_{l,-1,yt})\vee 0}{\pi\mu}}~\Bigg]-\frac{4}{\delta}-\frac{4\epsilon}{\delta}\sqrt{\frac{t}{\mu}}-\frac{2}{\delta}\sqrt{a_t}.
\end{align}
We point out that if $\Delta_1>yt\ge \delta t/2$, then
$\mathsf{H}_{l,1,yt}>t$ and $\mathsf{K}_{l,-1,yt}>t$ when $l>\lfloor
1/\delta\rfloor$, so that we can drop the upper bound in the above sum
and write for $\Delta_1>yt\ge \delta t/2$
\begin{align}
  \nonumber
  \sum_{v\in\mathbb{Z}_2}\sum_{l=0}^{+\infty}\sum_{n=a_t+1}^{+\infty}\mathds{1}_{\{n\in(q_{t-\mathsf{K}_{l,v,yt}},p_{t-\mathsf{H}_{l,v,yt}})\}}\,\frac{1-\epsilon}{\sqrt{2\pi n}}
  &\ge\sum_{l=0}^{+\infty}\Bigg[\sqrt{\frac{2(t-\mathsf{H}_{l,1,yt})\vee 0}{\pi\mu}}
  -\sqrt{\frac{2(t-\mathsf{K}_{l,-1,yt})\vee 0}{\pi\mu}}~\Bigg]\\
\nonumber
&-\frac{4}{\delta}-\frac{4\epsilon}{\delta}\sqrt{\frac{t}{\mu}}-\frac{2}{\delta}\sqrt{a_t}.
\end{align}
As before, identity (\ref{identity}) yields for $\Delta_1>yt\ge \delta
t/2$
\begin{align}
  \nonumber
  \sum_{v\in\mathbb{Z}_2}\sum_{l=0}^{+\infty}\sum_{n=a_t+1}^{+\infty}\mathds{1}_{\{n\in(q_{t-\mathsf{K}_{l,v,yt}},p_{t-\mathsf{H}_{l,v,yt}})\}}\,\frac{1-\epsilon}{\sqrt{2\pi n}}
\nonumber
&\ge \sqrt{\frac{t}{2\pi\mu}}\,\sum_{l=0}^{+\infty}\int_0^1\mathds{1}_{\{(2l)\Delta_1+yt\le\xi t<(2l+2)\Delta_1-yt\}}\frac{d\xi}{\sqrt{1-\xi}}\\
&-\frac{4}{\delta}-\frac{4\epsilon}{\delta}\sqrt{\frac{t}{\mu}}-\frac{2}{\delta}\sqrt{a_t}.
\label{Psi_lower_5}
\end{align}
By plugging bound (\ref{Psi_lower_5}) in (\ref{Psi_lower_4}) and by
recalling that the condition $(2l)\Delta_1+yt\le\xi
t<(2l+2)\Delta_1-yt$ for some $l\ge 0$ and $\xi\in[0,1]$ implies
$\Delta_1>yt$, an application of Fubini's theorem finally gives for
$y\in[\delta/2,1]$ and $t>t_o$
\begin{align}
  \nonumber
  \Psi_t(y)&\ge \Ex\Bigg[\sqrt{\frac{t}{2\pi\mu}}\,\sum_{l=0}^{+\infty}\int_0^1\mathds{1}_{\{(2l)\Delta_1+yt\le\xi t<(2l+2)\Delta_1-yt\}}\frac{d\xi}{\sqrt{1-\xi}}\Bigg]\\
\nonumber
  &-\mathcal{C}_t\{\epsilon\}\,\prob\big[\Delta_1>\delta t/2\big]-\bigg(\frac{4}{\delta}+\frac{4\epsilon}{\delta}\sqrt{\frac{t}{\mu}}+\frac{2}{\delta}\sqrt{a_t}\bigg)
  \prob\big[\Delta_1>\delta t/2\big]\\
  &= \Psi'_t(y)-\mathcal{C}_t\{\epsilon\}\,\prob\big[\Delta_1>\delta t/2\big]-\bigg(\frac{4}{\delta}+\frac{4\epsilon}{\delta}\sqrt{\frac{t}{\mu}}
  +\frac{2}{\delta}\sqrt{a_t}\bigg)\prob\big[\Delta_1>\delta t/2\big].
  \label{Psi_lower_6}
\end{align}

Since $a_t:=\lfloor\epsilon t/\mu\rfloor$,
$\lim_{t\uparrow+\infty}\prob[\Delta_1>\delta
  t/2]/\prob[\Delta_1>t]=2^\alpha\delta^{-\alpha}$ by regular
variation, and
$\lim_{t\uparrow+\infty}\mathcal{C}_t\{\epsilon\}/\sqrt{t}=0$ by Lemma
\ref{cboundimp}, bounds (\ref{Psi_upper_7}) and (\ref{Psi_lower_6})
show that
\begin{equation*}
  \adjustlimits\limsup_{t\uparrow+\infty}\sup_{~y\in[\delta/2,1]~}\!
  \Bigg\{\bigg|\frac{\Psi_t(y)}{\sqrt{t}\,\prob[\Delta_1>t]}-\frac{\Psi_t'(y)}{\sqrt{t}\,\prob[\Delta_1>t]}\bigg|\Bigg\}\le
  \big[(1+\delta)\sqrt{\epsilon}+3\epsilon\big]\frac{2^{\alpha+1}}{\delta^{\alpha+1}\sqrt{\mu}}.
\end{equation*}
Limit (\ref{Psi_Psi_primo}) follows from here by sending $\epsilon$ to
zero. 

\begin{remark}
  The procedure to demonstrate the limit (\ref{Psi1}) explains the sum
  over the index $l$ in the expression of the function $f_\alpha$. In
  fact, it is now clear that this index basically counts the number of
  passages of the particle over the current edge. The sum over $l$
  then provides the total statistical weight of the passages on the
  longest edge.
\end{remark}

\subsection{The bound (\ref{F(hx)}) for the function $F$}
\label{sec:F(hx)}

We conclude the proof of Theorem \ref{atail} by verifying the bound
(\ref{F(hx)}).  Pick $x\in[\delta,1]$.  Starting from (\ref{defFy}),
the change of variable $\xi\mapsto h\xi$ gives
\begin{align}
\nonumber
F(hx)&:=\frac{1}{\sqrt{2\pi\mu}}\int_{hx}^1 f_\alpha\bigg(\frac{hx}{\xi}\bigg)\frac{d\xi}{\xi^\alpha\sqrt{1-\xi}}\\
\nonumber
&=\frac{1}{\sqrt{2\pi\mu}}\int_{hx}^h f_\alpha\bigg(\frac{hx}{\xi}\bigg)\frac{d\xi}{\xi^\alpha\sqrt{1-\xi}}+
\frac{1}{\sqrt{2\pi\mu}}\int_{h}^1 f_\alpha\bigg(\frac{hx}{\xi}\bigg)\frac{d\xi}{\xi^\alpha\sqrt{1-\xi}}\\
\nonumber
  &=\frac{h^{1-\alpha}}{\sqrt{2\pi\mu}}\int_x^1 f_\alpha\bigg(\frac{x}{\xi}\bigg)\frac{d\xi}{\xi^\alpha\sqrt{1-h\xi}}+
  \frac{1}{\sqrt{2\pi\mu}}\int_h^1 f_\alpha\bigg(\frac{hx}{\xi}\bigg)\frac{d\xi}{\xi^\alpha\sqrt{1-\xi}}.
\end{align}
The function $f_\alpha$ satisfies $0<f_\alpha(y)\le y^{-\alpha}$ for
all $y\in(0,1)$ by Lemma \ref{prop_funalpha}. Then, by observing that
$\sqrt{1-h\xi}>\sqrt{1-\xi}$ for $\xi\in[0,1]$, we find
\begin{align}
  \nonumber
  F(hx)&\le F(x)+\frac{h^{1-\alpha}-1}{\sqrt{2\pi\mu}\,x^\alpha}\int_0^1\frac{d\xi}{\sqrt{1-\xi}}+
  \frac{1}{\sqrt{2\pi\mu}\,(hx)^\alpha}\int_h^1\frac{d\xi}{\sqrt{1-\xi}}\\
  \nonumber
  &=F(x)+2\,\frac{h^{1-\alpha}-1}{\sqrt{2\pi\mu}\,x^\alpha}+2\,\frac{\sqrt{1-h}}{\sqrt{2\pi\mu}\,(hx)^\alpha}.
\end{align}
The arbitrariness of $x\in[\delta,1]$, together with the fact that
$h>1/2$, proves (\ref{F(hx)}).

\section{Annealed moments: proof of Theorem \ref{annealed_moment}}
\label{proof_amoment}

Pick three real numbers $q>0$, $L\ge 1$, and $\delta\in(0,1]$.  We
  investigate the annealed $q$-order moment of the particle
  displacement by isolating the contribution of the normal
  fluctuations up to the spatial scale $L\sqrt{\mu t}$, the
  contribution of the large fluctuations above the threshold $\delta
  t$, and the contribution of the fluctuations in between these two
  regimes. Focusing on times $t$ that satisfy $L\sqrt{\mu t}<\delta
  t$, that is $t>\mu L^2/\delta^2$, our starting point to estimate the
  moment $\EE[|X_t|^q]$ is the formula
\begin{equation}
  \EE\big[|X_t|^q\big]=\EE\Big[|X_t|^q\mathds{1}_{\{|X_t|\le L\sqrt{\mu t}\}}\Big]+\EE\Big[|X_t|^q\mathds{1}_{\{|X_t|>\delta t\}}\Big]
  +\EE\Big[|X_t|^q\mathds{1}_{\{L\sqrt{\mu t}<|X_t|\le\delta t\}}\Big].
\label{amoment_0}
\end{equation}
We first study the large-$t$ behavior of each term in
(\ref{amoment_0}), and then we send $L$ to infinity and $\delta$ to
zero.

The contribution of the normal fluctuations is controlled by the
annealed central limit theorem stated by Corollary \ref{CLTa} in the
following way. The identity $|X_t|^q\mathds{1}_{\{|X_t|\le L\sqrt{\mu
    t}\}}=(\mu t)^{\frac{q}{2}}\int_0^L
  qx^{q-1}\mathds{1}_{\{x<|X_t|/\sqrt{\mu t}\le L\}}dx$, in
  combination with Fubini's theorem, yields
\begin{align}
  \nonumber
  \EE\Big[|X_t|^q\mathds{1}_{\{|X_t|\le L\sqrt{\mu t}\}}\Big]&=\EE\Bigg[(\mu t)^{\frac{q}{2}}\int_0^L qx^{q-1}\mathds{1}_{\{x<|X_t|/\sqrt{\mu t}\le
    L\}}dx\Bigg]\\
  \nonumber
  &=(\mu t)^{\frac{q}{2}}\int_0^L qx^{q-1}\,\PP\bigg[x<\frac{|X_t|}{\sqrt{\mu t}}\le L\bigg]dx.
  \end{align}
Thus, Corollary \ref{CLTa} together with the Portmanteau theorem
and the dominated convergence theorem shows that
\begin{align}
  \nonumber
  \lim_{t\uparrow+\infty}\frac{\EE[|X_t|^q\mathds{1}_{\{|X_t|\le L\sqrt{\mu t}\}}]}{t^{\frac{q}{2}}}
  &=\sqrt{\frac{\mu^q}{2\pi}}\int_0^L qx^{q-1}\bigg(\int_{-\infty}^{+\infty} \mathds{1}_{\{x<|\xi|\le L\}}e^{-\frac{1}{2}\xi^2}d\xi\bigg)dx\\
  &=\sqrt{\frac{2\mu^q}{\pi}}\int_0^L \xi^q e^{-\frac{1}{2}\xi^2} d\xi.
  \label{normalfl}
  \end{align}

The contribution of the large fluctuations is tackled by means of the
precise large deviation principle stated by Theorem \ref{atail}.  The
identity $|X_t|^q\mathds{1}_{\{|X_t|>\delta t\}}=(\delta
t)^q\mathds{1}_{\{|X_t|>\delta t\}}+ t^q\int_\delta^1
qx^{q-1}\mathds{1}_{\{|X_t|>xt\}}dx$, Fubini's theorem, and the
left-right symmetry of the model give
\begin{equation*}
  \EE\Big[|X_t|^q\mathds{1}_{\{|X_t|>\delta t\}}\Big]=2(\delta t)^q\,\PP\big[X_t>\delta t\big]
  +2t^q\int_\delta^1 qx^{q-1}\PP\big[X_t>xt\big]dx.
  \end{equation*}
It follows from Theorem \ref{atail} that
\begin{equation}
  \lim_{t\uparrow+\infty} \frac{\EE[|X_t|^q\mathds{1}_{\{|X_t|>\delta t\}}]}{t^{q+\frac{1}{2}}\prob[\Delta_1>t]}
  =2\delta^qF(\delta)+2\int_\delta^1 qx^{q-1}F(x)\,dx.
\label{largefl}
\end{equation}

We shall prove Theorem \ref{annealed_moment} by combining formula
(\ref{amoment_0}) with the limits (\ref{normalfl}) and (\ref{largefl})
and by demonstrating that the contribution of intermediate
fluctuations is negligible, addressing the case $q<2\alpha-1$ or
$q=2\alpha-1$ and $\alpha=1$ in Section \ref{sec:qsmall} and the case
$q=2\alpha-1$ and $\alpha>1$ or $q>2\alpha-1$ in Section
\ref{sec:qlarge}. The intermediate fluctuations are characterized by
the following lemma, which is based on Rosenthal's inequalities for
the moments of a sum of i.i.d.\ random variables and does not require
any assumption on the tail probability of $\Delta_1$ other than
$\mu<+\infty$.
\begin{lemma}
  \label{bound_aux_ql}
  Assume that $\mu<+\infty$ and fix $q>0$, $L\ge 1$, and
  $\delta\in(0,1]$. Then, for all sufficiently large $t$
\begin{equation*}
  \EE\Big[|X_t|^q\mathds{1}_{\{L\sqrt{\mu t}<|X_t|\le\delta t\}}\Big]\le
  \begin{cases}
\Big(2^9\frac{\sqrt{\mu t}}{L}\Big)^q & \mbox{if }q\le 1,\\[0.3em]
2^{9q^2}\Big\{\frac{(\mu t)^{\frac{q}{2}}}{L}+\sqrt{\frac{t}{\mu}}\int_0^{\delta t}qx^{q-1}\prob[\Delta_1>x]\,dx\Big\} & \mbox{if }q>1.
\end{cases}
\end{equation*}
\end{lemma}

\begin{proof}[Proof of Lemma \ref{bound_aux_ql}]
In order to solve the case $q\le 1$ it suffices to address the
instance $q=1$, so that we only need to consider situations where
$q\ge 1$. In fact, Jensen's inequality gives for each positive $q<1$
\begin{equation*}
\EE\Big[|X_t|^q\mathds{1}_{\{L\sqrt{\mu t}<|X_t|\le\delta t\}}\Big]\le\Big(\EE\big[|X_t|\mathds{1}_{\{L\sqrt{\mu t}<|X_t|\le\delta t\}}\big]\Big)^q.
\end{equation*}

Assume $q\ge 1$ and set $b_t=:\lfloor 2t/\mu\rfloor+1$. As in the
proof of Theorem \ref{atail}, we use $b_t$ to cut off the number of
collisions by time $t$. The left-right symmetry of the model and the
fact that $|X_t|\le t$ give for all $t>0$
\begin{align}
\nonumber
\EE\Big[|X_t|^q\,\mathds{1}_{\{L\sqrt{\mu t}<|X_t|\le\delta t\}}\Big]&=
2\,\EE\Big[X_t^q\,\mathds{1}_{\big\{L\sqrt{\mu t}<X_t\le\delta t,\,X_t>0\big\}}\Big]\\
&\le 2\,\EE\Big[X_t^q\,\mathds{1}_{\big\{L\sqrt{\mu t}<X_t\le\delta t,\,X_t>0,\,N_t\le b_t\big\}}\Big]+2t^q\,\PP\big[N_t>b_t\big].
\label{bound_aux_ql_11}
\end{align}
Let us recall that the condition $X_t>0$ entails $R_{N_t}\ge 1$ and
$X_t\le\sum_{k=1}^{R_{N_t}}\Delta_k$, as discussed at the beginning of
Section \ref{proof:atail}. It follows that
$X_t\le\sum_{k=1}^{R_{N_t}}\min\{\Delta_k,\delta t\}$ if
$0<X_t\le\delta t$. Set $\mathsf{D}_{k,z}:=\min\{\Delta_k,z\}$ for
brevity, $z$ being a real number.  These arguments in combination with
(\ref{bound_aux_ql_11}) yield the bound
\begin{align}
\nonumber
\EE\Big[|X_t|^q\,\mathds{1}_{\{L\sqrt{\mu t}<|X_t|\le\delta t\}}\Big]
&\le 2\,\EE\Bigg[\Bigg(\sum_{k=1}^{R_{N_t}}\mathsf{D}_{k,\delta t}\Bigg)^q
  \mathds{1}_{\big\{\sum_{k=1}^{R_{N_t}}\mathsf{D}_{k,\delta t}>L\sqrt{\mu t},\,R_{N_t}\ge 1,\,N_t\le b_t\big\}}\Bigg]\\
\nonumber
    &+2t^q\,\PP\big[N_t>b_t\big]\\
    \nonumber
    &\le 2\,\EE\Bigg[\Bigg(\sum_{k=1}^{M_{b_t}^+}\mathsf{D}_{k,\delta t}\Bigg)^q
      \mathds{1}_{\big\{\sum_{k=1}^{M_{b_t}^+}\mathsf{D}_{k,\delta t}>L\sqrt{\mu t},\,M_{b_t}^+\ge 1\big\}}\Bigg]\\
    \nonumber
    &+2t^q\,\PP\big[N_t>b_t\big]\\
    &=2^q\sum_{m=1}^{+\infty}\mathcal{H}_t\{m\}P\big[M_{b_t}^+=m\big]+2t^q\,\PP\big[N_t>b_t\big],
 \label{amoment_ql1}
\end{align}
where we have set for each $m\ge 1$ and $t>0$
\begin{equation}
  \mathcal{H}_t\{m\}:=2^{1-q}\,\Ex\Bigg[\Bigg(\sum_{k=1}^m\mathsf{D}_{k,\delta t}\Bigg)^q\mathds{1}_{\big\{\sum_{k=1}^m\mathsf{D}_{k,\delta t}>L\sqrt{\mu t}\big\}}\Bigg].
\label{def:Hmoment}
\end{equation}
Bound (\ref{amoment_ql1}) is the starting point to prove the lemma,
and we now need to find an estimate of $\mathcal{H}_t\{m\}$.

One estimate for the moment of order $q\ge 1$ of a sum of
i.i.d.\ random variables has been suggested by Rosenthal.  An
application of Rosenthal's inequality (see \cite{Rosenthal1}, Lemma 1)
to the i.i.d.\ positive bounded variables $\mathsf{D}_{1,\delta
  t},\mathsf{D}_{2,\delta t},\ldots$ yields for every $m\ge 1$ and
$t>0$
\begin{equation}
\Ex\Bigg[\Bigg(\sum_{k=1}^m\mathsf{D}_{k,\delta t}\Bigg)^q\,\Bigg]
\le 2^{q^2}\max\bigg\{\Big(m\,\Ex\big[\mathsf{D}_{1,\delta t}\big]\Big)^q,m\,\Ex\big[\mathsf{D}_{1,\delta t}^q\big]\bigg\}
\le 2^{q^2}\bigg\{(\mu m)^q+m\,\Ex\big[\mathsf{D}_{1,\delta t}^q\big]\bigg\},
  \label{Rosenthal_ine}
\end{equation}
where we used the fact that $\Ex[\mathsf{D}_{1,\delta
    t}]\le\Ex[\Delta_1]=\mu$ for all $t$ to get the last bound.
Unfortunately, this inequality does not consider the constraint
$\sum_{k=1}^m\mathsf{D}_{k,\delta t}>L\sqrt{\mu t}$ of definition
(\ref{def:Hmoment}), which cannot be lost to get at the lemma. We make
some preliminary manipulations to the Rosenthal's inequality
(\ref{Rosenthal_ine}) in order to keep this constraint.  We resort at
first to the inequality $(w+z)^{q-1}\le 2^{q-1}(w^{q-1}+z^{q-1})$
valid for all numbers $w\ge 0$ and $z\ge 0$ as $q\ge 1$ to write down
for each $m\ge 2$ and $t>0$ the bound
\begin{align}
\nonumber
\mathcal{H}_t\{m\}&=2^{1-q}\sum_{l=1}^m\Ex\Bigg[\Bigg(\sum_{k=1}^m\mathsf{D}_{k,\delta t}\Bigg)^{q-1}
  \mathsf{D}_{l,\delta t}\,\mathds{1}_{\big\{\sum_{k=1}^m\mathsf{D}_{k,\delta t}>L\sqrt{\mu t}\big\}}\Bigg]\\
\nonumber
&=2^{1-q}\,m\,\Ex\Bigg[\Bigg(\sum_{k=1}^m\mathsf{D}_{k,\delta t}\Bigg)^{q-1}
 \mathsf{D}_{m,\delta t}\,\mathds{1}_{\big\{\sum_{k=1}^m\mathsf{D}_{k,\delta t}>L\sqrt{\mu t}\big\}}\Bigg]\\
\nonumber
&\le m\,\Ex\Bigg[\Bigg(\sum_{k=1}^{m-1}\mathsf{D}_{k,\delta t}\Bigg)^{q-1}
 \mathsf{D}_{m,\delta t}\,\mathds{1}_{\big\{\sum_{k=1}^m\mathsf{D}_{k,\delta t}>L\sqrt{\mu t}\big\}}\Bigg]
+m\,\Ex\Bigg[\mathsf{D}_{m,\delta t}^q\mathds{1}_{\big\{\sum_{k=1}^m\mathsf{D}_{k,\delta t}>L\sqrt{\mu t}\big\}}\Bigg].
\end{align}
The constraint $\sum_{k=1}^m\mathsf{D}_{k,\delta t}>L\sqrt{\mu t}$
implies that either $\sum_{k=1}^{m-1}\mathsf{D}_{k,\delta
  t}>L\sqrt{\mu t}/2$ or $\Delta_m\ge\mathsf{D}_{m,\delta
  t}>L\sqrt{\mu t}/2$. This observation together with the inequality
$\mathds{1}_{\{\sum_{k=1}^{m-1}\mathsf{D}_{k,\delta t}>L\sqrt{\mu
    t}/2\}}\le(2/L\sqrt{\mu t})\sum_{k=1}^{m-1}\mathsf{D}_{k,\delta
  t}$ yields
\begin{align}
\nonumber
\mathcal{H}_t\{m\}&\le m\,\Ex\Bigg[\Bigg(\sum_{k=1}^{m-1}\mathsf{D}_{k,\delta t}\Bigg)^{q-1}
  \Delta_m\,\mathds{1}_{\big\{\sum_{k=1}^{m-1}\mathsf{D}_{k,\delta t}>L\sqrt{\mu t}/2\big\}}\Bigg]+m\,\Ex\Bigg[\Bigg(\sum_{k=1}^{m-1}\mathsf{D}_{k,\delta t}\Bigg)^{q-1}
  \Delta_m\,\mathds{1}_{\{\Delta_m>L\sqrt{\mu t}/2\}}\Bigg]\\
\nonumber
&+m\,\Ex\bigg[\mathsf{D}_{m,\delta t}^q\mathds{1}_{\big\{\sum_{k=1}^{m-1}\mathsf{D}_{k,\delta t}>L\sqrt{\mu t}/2\big\}}\bigg]
+m\,\Ex\bigg[\mathsf{D}_{m,\delta t}^q\mathds{1}_{\{\Delta_m>L\sqrt{\mu t}/2\}}\bigg]\\
\nonumber
&\le \frac{2}{L}\sqrt{\frac{\mu}{t}}\, m\,\Ex\Bigg[\Bigg(\sum_{k=1}^{m-1}\mathsf{D}_{k,\delta t}\Bigg)^q\,\Bigg]
+m\,\Ex\Big[\Delta_1\,\mathds{1}_{\{\Delta_1>L\sqrt{\mu t}/2\}}\Big]\,\Ex\Bigg[\Bigg(\sum_{k=1}^{m-1}\mathsf{D}_{k,\delta t}\Bigg)^{q-1}\Bigg]\\
&+\frac{2}{L}\sqrt{\frac{\mu}{t}}\, m(m-1)\,\Ex\Big[\mathsf{D}_{1,\delta t}^q\Big]
+m\,\Ex\Big[\mathsf{D}_{1,\delta t}^q\mathds{1}_{\{\Delta_1>L\sqrt{\mu t}/2\}}\Big].
\label{amoment_ql2}
\end{align}
We notice that for any $m\ge 2$ and $t>0$
\begin{equation}
  m\,\Ex\Bigg[\Bigg(\sum_{k=1}^{m-1}\mathsf{D}_{k,\delta t}\Bigg)^{q-1}\Bigg]\le
  \frac{\Ex\Big[\big(\sum_{k=1}^m\mathsf{D}_{k,\delta t}\big)^q\Big]}{\Ex[\mathsf{D}_{1,\delta t}]},
\label{amoment_ql3}
\end{equation}
which follows from the lower bound
\begin{align}
 \nonumber
 \Ex\Bigg[\Bigg(\sum_{k=1}^m\mathsf{D}_{k,\delta t}\Bigg)^q\,\Bigg]&=\sum_{l=1}^m\Ex\Bigg[\Bigg(\sum_{k=1}^m\mathsf{D}_{k,\delta t}\Bigg)^{q-1}\mathsf{D}_{l,\delta t}\Bigg]
 \ge\sum_{l=1}^m\Ex\Bigg[\Bigg(\sum_{\substack{k=1\\k\ne l}}^m\mathsf{D}_{k,\delta t}\Bigg)^{q-1}\mathsf{D}_{l,\delta t}\Bigg]\\
\nonumber
&=m\,\Ex\big[\mathsf{D}_{1,\delta t}\big]\,\Ex\Bigg[\Bigg(\sum_{k=1}^{m-1}\mathsf{D}_{k,\delta t}\Bigg)^{q-1}\Bigg].
\end{align}
Then, by combining (\ref{amoment_ql2}) with (\ref{amoment_ql3}) we
obtain for all $m\ge 2$ and $t>0$
\begin{align}
\nonumber
\mathcal{H}_t\{m\}&\le\Bigg(\frac{2}{L}\sqrt{\frac{\mu}{t}}\, m+\frac{\Ex\big[\Delta_1\,\mathds{1}_{\{\Delta_1>L\sqrt{\mu t}/2\}}\big]}
        {\Ex[\mathsf{D}_{1,\delta t}]}\Bigg)\Ex\Bigg[\Bigg(\sum_{k=1}^m\mathsf{D}_{k,\delta t}\Bigg)^q\,\Bigg]\\
&+\frac{2}{L}\sqrt{\frac{\mu}{t}}\, m^2\,\Ex\Big[\mathsf{D}_{1,\delta t}^q\Big]
        +m\,\Ex\Big[\mathsf{D}_{1,\delta t}^q\mathds{1}_{\{\Delta_1>L\sqrt{\mu t}/2\}}\Big].
        \label{amoment_ql4}
\end{align}
This estimate also holds for $m=1$ due to the last term and the fact
that $q\ge 1$.  We are now in the condition to invoke Rosenthal's
inequality. By plugging bound (\ref{Rosenthal_ine}) in
(\ref{amoment_ql4}) we find for every $m\ge 1$ and $t>0$
\begin{align}
\nonumber
\mathcal{H}_t\{m\}&\le\Bigg(\frac{2^{q^2+1}}{L}\sqrt{\frac{\mu}{t}}\, m+2^{q^2}\frac{\Ex\big[\Delta_1\,\mathds{1}_{\{\Delta_1>L\sqrt{\mu t}/2\}}\big]}
        {\Ex[\mathsf{D}_{1,\delta t}]}\Bigg)\bigg\{(\mu m)^q+m\,\Ex\big[\mathsf{D}_{1,\delta t}^q\big]\bigg\}\\
        &+\frac{2}{L}\sqrt{\frac{\mu}{t}}\, m^2\,\Ex\Big[\mathsf{D}_{1,\delta t}^q\Big]
        +m\,\Ex\Big[\mathsf{D}_{1,\delta t}^q\mathds{1}_{\{\Delta_1>L\sqrt{\mu t}/2\}}\Big].
        \label{amoment_ql5}
\end{align}
This estimate improves the Rosenthal's inequality as it keeps track of
the constraint $\sum_{k=1}^m\mathsf{D}_{k,\delta t}>L\sqrt{\mu t}$.

We come back to (\ref{amoment_ql1}) with the strength of
(\ref{amoment_ql5}). When combining the upper bounds
(\ref{amoment_ql5}) with (\ref{amoment_ql1}) we use the inequality
$E[(M_{b_t}^+)^p]\le 2(pb_t)^{\frac{p}{2}}$ valid for any $p>0$ by
Lemma \ref{moment:Mn}.  Thus, we obtain for all $t>0$
\begin{align}
\nonumber
\EE\Big[|X_t|^q\,\mathds{1}_{\{L\sqrt{\mu t}<|X_t|\le\delta t\}}\Big]&\le\frac{2^{q^2+q+2}}{L}\sqrt{\frac{\mu}{t}}
\bigg\{\mu^q(q+1)^{\frac{q+1}{2}}b_t^{\frac{q+1}{2}}+2b_t\,\Ex\big[\mathsf{D}_{1,\delta t}^q\big]\bigg\}\\
        \nonumber
        &+2^{q^2+q+1}\frac{\Ex\big[\Delta_1\,\mathds{1}_{\{\Delta_1>L\sqrt{\mu t}/2\}}\big]}
        {\Ex[\mathsf{D}_{1,\delta t}]}\bigg\{\mu^qq^{\frac{q}{2}}b_t^{\frac{q}{2}}+\sqrt{b_t}\,\Ex\big[\mathsf{D}_{1,\delta t}^q\big]\bigg\}\\
\nonumber
        &+\frac{2^{q+3}}{L}\sqrt{\frac{\mu}{t}}\,b_t\,\Ex\Big[\mathsf{D}_{1,\delta t}^q\Big]
+2^{q+1}\sqrt{b_t}\,\Ex\Big[\mathsf{D}_{1,\delta t}^q\mathds{1}_{\{\Delta_1>L\sqrt{\mu t}/2\}}\Big]\\
&+2t^q\,\PP\big[N_t>b_t\big],
\label{amoment_ql6}
\end{align}
where $\mathsf{D}_{1,\delta t}:=\min\{\Delta_1,\delta t\}$. Since
$\lim_{t\uparrow+\infty}\Ex[\mathsf{D}_{1,\delta t}]=\mu$ and
$\lim_{t\uparrow+\infty}\Ex[\Delta_1\mathds{1}_{\{\Delta_1>L\sqrt{\mu
      t}/2\}}]=0$ by the dominated convergence theorem, we can find a
positive real number $t_o$ with the property that
$\Ex[\Delta_1\mathds{1}_{\{\Delta_1>L\sqrt{\mu
      t}/2\}}]/\Ex[\mathsf{D}_{1,\delta t}]\le 1/L$ and $b_t:=\lfloor
2t/\mu\rfloor+1\le 4t/\mu$ for $t>t_o$. Since there exists a real
number $\kappa>0$ such that
\begin{equation*}
\PP\big[N_t>b_t\big]=\PP\big[T_{b_t}\le t\big]\le \PP\big[T_{b_t}\le \mu b_t/2\big]\le e^{-\kappa b_t^{1/3}}
\end{equation*}
for all sufficiently large $t$ by Lemma \ref{Ttailann}, we can take
$t_o$ large enough to have $2t^{\frac{q}{2}}\PP[T_{b_t}\le t]\le
\mu^{\frac{q}{2}}/L$ as well for $t>t_o$. These arguments are used to
simplify (\ref{amoment_ql6}) and to finally state that for all $t>t_o$
\begin{align}
\nonumber
\EE\Big[|X_t|^q\,\mathds{1}_{\{L\sqrt{\mu t}<|X_t|\le\delta t\}}\Big]&\le\frac{2^{q^2+q+3}}{L}
\bigg\{2^q(q+1)^{\frac{q+1}{2}}(\mu t)^{\frac{q}{2}}+4\sqrt{\frac{t}{\mu}}\,\Ex\big[\mathsf{D}_{1,\delta t}^q\big]\bigg\}\\
        \nonumber
        &+\frac{2^{q^2+q+1}}{L}\bigg\{2^q q^{\frac{q}{2}}  (\mu t)^{\frac{q}{2}}+2\sqrt{\frac{t}{\mu}}\,\Ex\big[\mathsf{D}_{1,\delta t}^q\big]\bigg\}\\
        &+\frac{2^{q+5}}{L}\sqrt{\frac{t}{\mu}}\,\Ex\Big[\mathsf{D}_{1,\delta t}^q\Big]
        +2^{q+2}\sqrt{\frac{t}{\mu}}\,\Ex\Big[\mathsf{D}_{1,\delta t}^q\mathds{1}_{\{\Delta_1>L\sqrt{\mu t}/2\}}\Big]+\frac{(\mu t)^{\frac{q}{2}}}{L}.
\label{amoment_ql7}
\end{align}

Let us verify the statement of the lemma.  When $q=1$ and $t>t_o$, the
inequality $\Ex[\Delta_1\mathds{1}_{\{\Delta_1>L\sqrt{\mu
      t}/2\}}]/\Ex[\mathsf{D}_{1,\delta t}]\le 1/L$ gives
$\Ex[\mathsf{D}_{1,\delta t}^q\mathds{1}_{\{\Delta_1>L\sqrt{\mu
      t}/2\}}]\le \mu/L$, and (\ref{amoment_ql7}) shows that
\begin{equation}
\EE\Big[|X_t|\,\mathds{1}_{\{L\sqrt{\mu t}<|X_t|\le\delta t\}}\Big]\le\frac{2^8+2^6+2^5+2^3+1}{L}\sqrt{\mu t}\le 2^9\frac{\sqrt{\mu t}}{L}.
\label{amoment_ql8}
\end{equation}
When $q>1$ and $t>t_o$, (\ref{amoment_ql7}) in combination with the
inequality $\Ex[\mathsf{D}_{1,\delta
    t}^q\mathds{1}_{\{\Delta_1>L\sqrt{\mu t}/2\}}]\le
\Ex[\mathsf{D}_{1,\delta t}^q]$ and the fact that $L\ge 1$ yields
\begin{align}
\nonumber
\EE\Big[|X_t|^q\,\mathds{1}_{\{L\sqrt{\mu t}<|X_t|\le\delta t\}}\Big]&\le\frac{2^{q^2+2q+3}(q+1)^{\frac{q+1}{2}}+2^{q^2+2q+1}q^{\frac{q}{2}}+1}{L}(\mu t)^{\frac{q}{2}}\\
\nonumber
&+(2^{q^2+q+5}+2^{q^2+q+2}+2^{q+5}+2^{q+2})\sqrt{\frac{t}{\mu}}\,\Ex\big[\mathsf{D}_{1,\delta t}^q\big]\\
&\le 2^{9q^2}\bigg\{\frac{(\mu t)^{\frac{q}{2}}}{L}+\sqrt{\frac{t}{\mu}}\,\Ex\big[\mathsf{D}_{1,\delta t}^q\big]\bigg\}.
\label{amoment_ql9}
\end{align}
Bounds (\ref{amoment_ql8}) and (\ref{amoment_ql9}) prove the lemma as
Fubini's theorem tells us that
\begin{equation*}
  \Ex\big[\mathsf{D}_{1,\delta t}^q\big]=\Ex\big[\min\{\Delta_1,\delta t\}^q\big]=
  \Ex\bigg[\int_0^{\delta t}qx^{q-1}\mathds{1}_{\{\Delta_1>x\}}dx\bigg]=\int_0^{\delta t}qx^{q-1}\prob[\Delta_1>x]\,dx.
\qedhere
\end{equation*}
\end{proof}

\subsection{The case $q<2\alpha-1$ or $q=2\alpha-1$ and $\alpha=1$}
\label{sec:qsmall}

Assume that $q<2\alpha-1$ or that $q=2\alpha-1$ and $\alpha=1$, namely
that $q\le 1$ or that $q\in(1,2\alpha-1)$ if $\alpha>1$. Based on
Lemma \ref{bound_aux_ql}, we claim that for each $L\ge 1$ and all
sufficiently large $t$
\begin{equation}
  \EE\Big[|X_t|^q\,\mathds{1}_{\{|X_t|>L\sqrt{\mu t}\}}\Big]\le 2^{9q^2+9q}\frac{(\mu t)^{\frac{q}{2}}}{L^{q\wedge 1}}.
\label{bound_aux_ql_mod}
\end{equation}
In fact, recalling that $|X_t|\le t$, for $q\le 1$ this bound
immediately follows from Lemma \ref{bound_aux_ql} with $\delta=1$. For
$\alpha>1$ and $q\in(1,2\alpha-1)$, when invoking Lemma
\ref{bound_aux_ql} with $\delta=1$ we need to observe that
$\sqrt{t/\mu}\,\int_0^t qx^{q-1}\prob[\Delta_1>x]\,dx\le(\mu
t)^{\frac{q}{2}}/L$ for each $L\ge 1$ and all sufficiently large
$t$. Indeed, in this case there exists a positive number $\eta$ such
that $q-1-\eta>0$ and $2\alpha-1-q-\eta>0$.  Then
\begin{equation*}
\int_0^t qx^{q-1}\prob\big[\Delta_1>x\big]dx\le t^{\frac{q-1-\eta}{2}}\int_0^{+\infty} q\,x^{-\frac{2\alpha-1-q-\eta}{2}-1}\ell(x)\,dx,
\end{equation*}
the integral in the r.h.s.\ being finite as
$\lim_{x\uparrow+\infty}x^{-\gamma}\ell(x)=0$ for every $\gamma>0$ (see
\cite{RV}, Proposition 1.3.6).

We are now ready to prove Theorem \ref{annealed_moment} for
$q<2\alpha-1$ or $q=2\alpha-1$ and $\alpha=1$. Formula
(\ref{amoment_0}) with $\delta=1$, limit (\ref{normalfl}), and the
bound (\ref{bound_aux_ql_mod}) give
\begin{equation*}
\limsup_{t\uparrow+\infty}\frac{\EE[|X_t|^q]}{t^{\frac{q}{2}}}\le\sqrt{\frac{2\mu^q}{\pi}}\int_0^L \xi^q e^{-\frac{1}{2}\xi^2} d\xi
+2^{9q^2+9q}\frac{\mu^{\frac{q}{2}}}{L^{q\wedge 1}}
\end{equation*}
and
\begin{equation*}
\liminf_{t\uparrow+\infty}\frac{\EE[|X_t|^q]}{t^{\frac{q}{2}}}\ge\sqrt{\frac{2\mu^q}{\pi}}\int_0^L \xi^q e^{-\frac{1}{2}\xi^2} d\xi.
\end{equation*}
Theorem \ref{annealed_moment} is then demonstrated by sending $L$ to
infinity:
\begin{equation*}
\lim_{t\uparrow+\infty}\frac{\EE[|X_t|^q]}{t^{\frac{q}{2}}}
  =\sqrt{\frac{2\mu^q}{\pi}}\int_0^{+\infty} \xi^q e^{-\frac{1}{2}\xi^2} d\xi=\sqrt{\frac{(2\mu)^q}{\pi}}\,\Gamma\bigg(\frac{q+1}{2}\bigg)=:g_q.
  \end{equation*}

\subsection{The case $q=2\alpha-1$ and $\alpha>1$ or $q>2\alpha-1$}
\label{sec:qlarge}

Assume that $q=2\alpha-1$ and $\alpha>1$ or that $q>2\alpha-1$. In
this case the large fluctuations of the displacement come into play as
here we show that
\begin{equation}
\lim_{t\uparrow+\infty}\frac{\EE[|X_t|^q]}{g_qt^{\frac{q}{2}}+d_qt^{q+\frac{1}{2}}\prob[\Delta_1>t]}=1,
\label{aux_qlarge}
\end{equation}
$g_q$ and $d_q$ being the positive coefficients introduced by Theorem
\ref{annealed_moment}.  This limit is exactly Theorem
\ref{annealed_moment} when $q=2\alpha-1$ and $\alpha>1$. Regarding the
instance $q>2\alpha-1$, Theorem \ref{annealed_moment} follows from
(\ref{aux_qlarge}) by observing that if $q>2\alpha-1$, then the fact
that $\lim_{x\uparrow+\infty}x^{\gamma}\ell(x)=+\infty$ for every
$\gamma>0$ (see \cite{RV}, Proposition 1.3.6) gives
\begin{equation*}
  \lim_{t\uparrow+\infty}\frac{t^{\frac{q}{2}}}{t^{q+\frac{1}{2}}\prob[\Delta_1>t]}=
  \lim_{t\uparrow+\infty}\frac{1}{t^{\frac{q-2\alpha+1}{2}}\ell(t)}=0.
\end{equation*}

Let us verify (\ref{aux_qlarge}). To begin with, we notice that both
when $q=2\alpha-1$ and $\alpha>1$ or when $q>2\alpha-1$ we have
$q>\alpha\ge 1$. Under the condition $q>\alpha$ we find
\begin{equation}
  \lim_{\delta\downarrow 0}\lim_{t\uparrow+\infty} \frac{\EE[|X_t|^q\mathds{1}_{\{|X_t|>\delta t\}}]}{t^{q+\frac{1}{2}}\prob[\Delta_1>t]}=2\int_0^1 qx^{q-1}F(x)\,dx=d_q<+\infty,
\label{aux_qlarge_2}
\end{equation}
where $F$ is the function of Theorem \ref{atail} that describes the
large fluctuations. In fact, since $0<f_\alpha(x)\le x^{-\alpha}$ for
all $x\in(0,1)$ by Lemma \ref{prop_funalpha}, Fubini's theorem and the
change of variable $x\mapsto\xi x$ yield
\begin{align}
  \nonumber
  2\int_0^1 qx^{q-1}F(x)\,dx&=\int_0^1 dx\,qx^{q-1}\sqrt{\frac{2}{\pi\mu}}\int_x^1\frac{d\xi}{\xi^\alpha\sqrt{1-\xi}}\,
      f_\alpha\bigg(\frac{x}{\xi}\bigg)\\
\nonumber
&=\sqrt{\frac{2}{\pi\mu}}\int_0^1\frac{d\xi}{\xi^\alpha\sqrt{1-\xi}}\int_0^\xi dx\,qx^{q-1}f_\alpha\bigg(\frac{x}{\xi}\bigg)\\
\nonumber
&=\sqrt{\frac{2}{\pi\mu}}\int_0^1\frac{\xi^{q-\alpha}d\xi}{\sqrt{1-\xi}}\int_0^1 dx\,qx^{q-1}f_\alpha(x)<+\infty.
\end{align}
On the other hand, the theory of the Euler beta function $\mathrm{B}$
tells us that
\begin{equation*}
\int_0^1\frac{\xi^{q-\alpha}d\xi}{\sqrt{1-\xi}}=\mathrm{B}(q-\alpha+1,1/2)=\frac{\Gamma(q-\alpha+1)\sqrt{\pi}}{\Gamma(q-\alpha+3/2)}
\end{equation*}
since for all reals $z>0$ and $\zeta>0$
\begin{equation*}
\mathrm{B}(z,\zeta):=\int_0^1\xi^{z-1}(1-\xi)^{\zeta-1} d\xi=\frac{\Gamma(z)\Gamma(\zeta)}{\Gamma(z+\zeta)}.
\end{equation*}
Thus, we have
\begin{equation*}
  2\int_0^1 qx^{q-1}F(x)\,dx=\sqrt{\frac{2}{\mu}}\frac{\Gamma(q-\alpha+1)}{\Gamma(q-\alpha+3/2)}\int_0^1 qx^{q-1}f_\alpha(x)\,dx=:
  d_q<+\infty.
\end{equation*}
By combining limit (\ref{largefl}) with this identity we obtain
(\ref{aux_qlarge_2}) since $\lim_{\delta\downarrow
  0}\delta^qF(\delta)=0$ as $q>\alpha$ and
$F(\delta)\le\delta^{-\alpha}/\sqrt{\mu}$ due to the bound
$f_\alpha(x)\le x^{-\alpha}$ for $x\in(0,1)$.

Another consequence of the fact that $q>\alpha$ is the following,
which is a general result about truncated moments of slowly varying
functions (see \cite{Feller2} chapter VIII.9, Theorem 1):
\begin{equation}
 \lim_{t\uparrow+\infty}\frac{\int_0^t qx^{q-1}\prob[\Delta_1>x]\,dx}{t^q\,\prob[\Delta_1>t]}=
   \lim_{t\uparrow+\infty}\frac{\int_0^t qx^{q-\alpha-1}\ell(x)\,dx}{t^{q-\alpha}\ell(t)}=\frac{q}{q-\alpha}.
     \label{limFeller}
\end{equation}
Limit (\ref{normalfl}) shows that
\begin{equation}
  \lim_{L\uparrow+\infty}\lim_{t\uparrow+\infty}\frac{\EE[|X_t|^q\mathds{1}_{\{|X_t|\le L\sqrt{\mu t}\}}]}{t^{\frac{q}{2}}}
  =\sqrt{\frac{(2\mu)^q}{\pi}}\,\Gamma\bigg(\frac{q+1}{2}\bigg)=:g_q.
\label{aux_qlarge_3}
\end{equation}

We are now ready to prove (\ref{aux_qlarge}).  Pick $\epsilon>0$. Due
to limits (\ref{aux_qlarge_3}), (\ref{aux_qlarge_2}), and
(\ref{limFeller}) and due to the asymptotic scale invariance of slowly
varying functions, there exist three positive numbers $L\ge 1\vee
(1/\epsilon)$, $\delta<1\wedge (\epsilon^{\frac{1}{q-\alpha}})$, and
$t_o$ such that for all $t>t_o$
\begin{enumerate}[(a)]
\item $(1-\epsilon)g_qt^{\frac{q}{2}}\le\EE[|X_t|^q\mathds{1}_{\{|X_t|\le L\sqrt{\mu t}\}}]\le(1+\epsilon)g_qt^{\frac{q}{2}}$;
\item $(1-\epsilon)d_qt^{q+\frac{1}{2}}\prob[\Delta_1>t]\le\EE[|X_t|^q\mathds{1}_{\{|X_t|>\delta t\}}]\le(1+\epsilon)d_qt^{q+\frac{1}{2}}\prob[\Delta_1>t]$;
\item $\int_0^{\delta t}qx^{q-1}\prob[\Delta_1>x]\,dx\le\frac{2q}{q-\alpha}(\delta t)^q\,\prob[\Delta_1>\delta t]$ and
  $\prob[\Delta_1>\delta t]\le 2\delta^{-\alpha}\prob[\Delta_1>t]$.
\end{enumerate}
Since $q>1$ as $q>\alpha$, Lemma \ref{bound_aux_ql} and property (c)
give for every $t>t_o$ large enough
\begin{align}
  \nonumber
  \EE\Big[|X_t|^q\,\mathds{1}_{\{L\sqrt{\mu t}<|X_t|\le\delta t\}}\Big]&\le
  2^{9q^2}\bigg\{\frac{(\mu t)^{\frac{q}{2}}}{L}+\frac{1}{\sqrt{\mu}}\frac{4q}{q-\alpha}\delta^{q-\alpha} t^{q+\frac{1}{2}}\prob[\Delta_1>t]\bigg\}\\
\nonumber
&\le \epsilon\,2^{9q^2}\bigg\{(\mu t)^{\frac{q}{2}}+\frac{1}{\sqrt{\mu}}\frac{4q}{q-\alpha}t^{q+\frac{1}{2}}\prob[\Delta_1>t]\bigg\}\\
\nonumber
&\le\epsilon\max\bigg\{\frac{2^{9q^2}\mu^{\frac{q}{2}}}{g_q},\frac{2^{9q^2}}{d_q\sqrt{\mu}}\frac{4q}{q-\alpha}\bigg\}
\bigg\{g_qt^{\frac{q}{2}}+d_qt^{q+\frac{1}{2}}\prob[\Delta_1>t]\bigg\},
\end{align}
the second inequality being a consequence of the fact that $L\ge
1/\epsilon$ and $\delta<\epsilon^{\frac{1}{q-\alpha}}$ by
hypothesis. Then, by also invoking properties (a) and (b), it follows
from formula (\ref{amoment_0}) that for all sufficiently large $t>t_o$
\begin{equation*}
  1-\epsilon\le\frac{\EE[|X_t|^q]}{g_qt^{\frac{q}{2}}+d_qt^{q+\frac{1}{2}}\prob[\Delta_1>t]}\le 1+\epsilon+
  \epsilon \max\bigg\{\frac{2^{9q^2}\mu^{\frac{q}{2}}}{g_q},\frac{2^{9q^2}}{d_q\sqrt{\mu}}\frac{4q}{q-\alpha}\bigg\}.
\end{equation*}
The arbitrariness of $\epsilon$ proves (\ref{aux_qlarge}).

\section{Quenched fluctuations: proof of Theorem \ref{qtail}}
\label{qtail_proof}

Here we prove Theorem \ref{qtail}.  The strong law of large numbers
for i.i.d.\ random variables and Lemma \ref{Ttailq} ensure us that
there exist an event $\Omega_o\in\mathscr{F}$ with $\prob[\Omega_o]=1$
and a positive number $\kappa_o$ such that for each
$\omega\in\Omega_o$
\begin{enumerate}[$(a)$]
\item $\lim_{r\uparrow+\infty}(1/r)\sum_{k=1}^r\Delta_k^\omega=\mu$ and $\lim_{r\uparrow+\infty}(1/r)\sum_{k=1-r}^0\Delta_k^\omega=\mu$;
\item $P[T_n^\omega\le \mu n/2]\le e^{-\kappa_o n^{1/3}}$ for all sufficiently large $n$. 
\end{enumerate}
Let $\kappa$ be the minimum between $1/(8\sqrt{\mu})$ and $\kappa_o/\sqrt{\mu}$. We
are going to show that for every $\omega\in\Omega_o$ and $x\in(0,1]$
\begin{equation}
  \limsup_{t\uparrow+\infty}\frac{1}{\sqrt{xt}}\ln P\big[|X_t^\omega|>xt\big]\le -\kappa.
\label{qtail7}
\end{equation}

Fix $\omega\in\Omega_o$ and $x\in(0,1]$. Consider an integer $b_t\ge
  2t/\mu$ to cut off the number of collisions by time $t$, which will
  be chosen later on. To begin with, let us recall that $0\le
  X_t^\omega\le \sum_{k=1}^{R_{N_t^\omega}}\Delta_k^\omega$ or
  $-\sum_{k=R_{N_t^\omega}}^0\Delta_k^\omega\le X_t^\omega\le 0$
  according to $R_{N_t^\omega}\ge 1$ or $R_{N_t^\omega}\le 0$, as we
  have seen at the beginning of Section \ref{proof:atail}. Then, for
  all $t>0$ we can write down the bound
\begin{align}
  \nonumber
  P\big[|X_t^\omega|>xt\big]&\le P\big[|X_t^\omega|>xt,\,N_t^\omega\le b_t\big]+P\big[N_t^\omega>b_t\big]\\
  \nonumber
  &\le P\Bigg[\sum_{k=1}^{R_{N_t^\omega}}\Delta_k^\omega>xt,\,R_{N_t^\omega}\ge 1,\,N_t^\omega\le b_t\Bigg]\\
  &+P\Bigg[\sum_{k=R_{N_t^\omega}}^0\Delta_k^\omega>xt,\,R_{N_t^\omega}\le 0,\,N_t^\omega\le b_t\Bigg]
  +P\big[N_t^\omega>b_t\big].
  \label{qtail1}
\end{align}
Property $(a)$ entails that there exists a positive integer $r_o$ such
that $\sum_{k=1}^r\Delta_k^\omega\le 2\mu r$ and
$\sum_{k=1-r}^0\Delta_k^\omega\le 2\mu r$ for $r>r_o$.  Let $t_o$ be
the maximum between $(1/x)\sum_{k=1}^{r_o}\Delta_k^\omega$ and
$(1/x)\sum_{k=1-r_o}^0\Delta_k^\omega$. If $t>t_o$, then
$xt>xt_o\ge\sum_{k=1}^{r_o}\Delta_k^\omega$, so that the conditions
$R_{N_t^\omega}\ge 1$ and
$\sum_{k=1}^{R_{N_t^\omega}}\Delta_k^\omega>xt$ necessarily imply
$R_{N_t^\omega}>r_o$, which in turn gives
$\sum_{k=1}^{R_{N_t^\omega}}\Delta_k^\omega\le 2\mu
R_{N_t^\omega}$. Similarly, the conditions $R_{N_t^\omega}\le 0$ and
$\sum_{k=R_{N_t^\omega}}^0\Delta_k^\omega>xt$ with $t>t_o$ entail
$R_{N_t^\omega}<1-r_o$ as $xt>xt_o\ge\sum_{k=1-r_o}^0\Delta_k^\omega$,
which in turn yields $\sum_{k=R_{N_t^\omega}}^0\Delta_k^\omega\le
2\mu(1-R_{N_t^\omega})$. Thus, starting from (\ref{qtail1}), we see
that for every $t>t_o$
\begin{align}
  \nonumber
  P\big[|X_t^\omega|>xt\big]&\le P\big[2\mu R_{N_t^\omega}>xt,\,N_t^\omega\le b_t\big]
  +P\big[2\mu(1-R_{N_t^\omega})>xt,\,N_t^\omega\le b_t\big]+P\big[N_t^\omega>b_t\big]\\
  &\le P\big[2\mu M_{b_t}^+>xt\big]+P\big[2\mu(1-M_{b_t}^-)>xt\big]+P\big[N_t^\omega>b_t\big].
\label{qtail2}
\end{align}
Lemma \ref{tail:Mn} gives $P[2\mu M_{b_t}^+>xt]\le
2e^{-\frac{(xt)^2}{8\mu^2b_t}}$ and part $(i)$ of Proposition
\ref{dist:Mn} provides the equality $P[2\mu(1-M_{b_t}^-)>xt]=P[2\mu
  M_{b_t}^+>xt]$. Moreover, since $t\le\mu b_t/2$ by hypothesis,
property $(b)$ yields for all sufficiently large $t$
\begin{equation*}
P\big[N_t^\omega>b_t\big]=P\big[T_{b_t}^\omega\le t\big]\le P\big[T_{b_t}^\omega\le \mu b_t/2\big]\le e^{-\kappa_o b_t^{1/3}}.
\end{equation*}
In conclusion, (\ref{qtail2}) shows that for all $t>t_o$ large enough
\begin{equation}
  P\big[|X_t^\omega|>xt\big]\le  4e^{-\frac{(xt)^2}{8\mu^2b_t}}+e^{-\kappa_o b_t^{1/3}}.
\label{qtail3}
\end{equation}
The value of $b_t$ that minimizes, at the same time, the contribution
$e^{-\frac{(xt)^2}{8\mu^2b_t}}$ coming from a large fluctuation of the
simple symmetric random walk and the contribution $e^{-\kappa_o
  b_t^{1/3}}$ coming from a large number of collisions is of the order
of magnitude of $(xt/\mu)^{3/2}$. By setting for instance
$b_t:=\lfloor 2t/\mu+ (xt/\mu)^{3/2}\rfloor+1\ge 2t/\mu$ in
(\ref{qtail3}) we get
\begin{equation*}
  \limsup_{t\uparrow+\infty}\frac{1}{\sqrt{xt}}\ln P\big[|X_t^\omega|>xt\big]\le -\min\bigg\{\frac{1}{8\sqrt{\mu}},\frac{\kappa_o}{\sqrt{\mu}}\bigg\}.
\end{equation*}

\begin{remark}
The stretching exponent of the estimate (\ref{qtail7}) can be improved
with no effort when there exists a real number $\delta_o>0$ such that
$\prob[\Delta_1\ge\delta_o]=1$. In such case, it is possible to choose
the set $\Omega_o$ of full probability with the further property that
$T_n^\omega\ge\delta_on$ for all $\omega\in\Omega_o$ and $n\ge 0$. It
follows that $\delta_o(N_t^\omega-1)\le T_{N_t^\omega-1}^\omega\le t$,
i.e.\ $N_t^\omega\le t/\delta_o+1$, for each $\omega\in\Omega_o$. By
making use of the integer $b_t:=\lfloor t/\delta_o\rfloor+1$ in
(\ref{qtail2}) we realize that for $\pae$
\begin{equation*}
  \limsup_{t\uparrow+\infty}\frac{1}{t}\ln P\big[|X_t^\omega|>xt\big]\le -\frac{\delta_ox^2}{8\mu^2}.
\end{equation*}
\end{remark}

\section{Quenched moments: proof of Theorem \ref{quenched_moment}}
\label{proof_qmoments}

The proof of Theorem \ref{quenched_moment} relies on the following
bound for the quenched moments.
\begin{lemma}
  \label{SupX}
  Assume that $\mu<+\infty$.  The following property holds
  for $\pae$: for each $q>0$ there exists a constant $C$ such that for all
  $t>0$
\begin{equation*}
E\big[|X_t^\omega|^q\big]\le C (\mu t)^{\frac{q}{2}}.
\end{equation*}
\end{lemma}

\begin{proof}[Proof of Lemma \ref{SupX}]
  Set $b_t:=\lfloor 2t/\mu\rfloor+1$, which is used to cut off the
  number of collisions by time $t$. The strong law of large numbers
  for i.i.d.\ random variables and Lemma \ref{Ttailq} tell us that
  there exist $\Omega_o\in\mathscr{F}$ with $\prob[\Omega_o]=1$ and
  $\kappa>0$ such that for each $\omega\in\Omega_o$
\begin{enumerate}[$(a)$]
\item $\lim_{r\uparrow+\infty}(1/r)\sum_{k=1}^r\Delta_k^\omega=\mu$ and $\lim_{r\uparrow+\infty}(1/r)\sum_{k=1-r}^0\Delta_k^\omega=\mu$;
\item $P[N_t^\omega>b_t]=P[T_{b_t}^\omega\le t]\le P[T_{b_t}^\omega\le \mu b_t/2]\le e^{-\kappa b_t^{1/3}}$ for all sufficiently large $t$. 
\end{enumerate}
Fix $\omega\in\Omega_o$ and $q>0$. Property $(a)$ entails that a
positive constant $\lambda$ can be found in such a way that
$\sum_{k=1}^r\Delta_k^\omega\le \lambda r$ and
$\sum_{k=1-r}^0\Delta_k^\omega\le \lambda r$ for all $r\ge 1$. Since
$0\le X_t^\omega\le \sum_{k=1}^{R_{N_t^\omega}}\Delta_k^\omega$ or
$-\sum_{k=R_{N_t^\omega}}^0\Delta_k^\omega\le X_t^\omega\le 0$
according to $R_{N_t^\omega}\ge 1$ or $R_{N_t^\omega}\le 0$, as we
have seen at the beginning of Section \ref{proof:atail}, it follows
that $|X_t^\omega|\le \lambda R_{N_t^\omega}$ or $|X_t^\omega|\le
\lambda (1-R_{N_t^\omega})$ according to $R_{N_t^\omega}\ge 1$ or
$R_{N_t^\omega}\le 0$. Thus, bearing in mind that $|X_t^\omega|\le t$,
for all $t>0$ we get
\begin{align}
\nonumber
E\big[|X_t^\omega|^q\big]&\le E\Big[|X_t^\omega|^q\mathds{1}_{\{R_{N_t^\omega}\ge 1,\,N_t\le b_t\}}\Big]+
E\Big[|X_t^\omega|^q\mathds{1}_{\{R_{N_t^\omega}\le 0,\,N_t\le b_t\}}\Big]
+t^qP\big[N_t^\omega>b_t\big]\\
\nonumber
&\le \lambda^q E\Big[R_{N_t^\omega}^q\mathds{1}_{\{R_{N_t^\omega}\ge 1,\,N_t\le b_t\}}\Big]+
\lambda^q E\Big[(1-R_{N_t^\omega})^q\mathds{1}_{\{R_{N_t^\omega}\le 0,\,N_t\le b_t\}}\Big]
+t^qP\big[N_t^\omega>b_t\big]\\
\nonumber
&\le\lambda^qE\big[(M_{b_t}^+)^q\big]+\lambda^qE\big[(1-M_{b_t}^-)^q\big]+t^qP\big[N_t^\omega>b_t\big].
\end{align}
This bound combined with property $(b)$ proves the lemma since Lemma
\ref{moment:Mn} tells us that $E[(M_{b_t}^+)^q]\le
2(qb_t)^{\frac{q}{2}}$ and part $(i)$ of Proposition \ref{dist:Mn}
gives $E[(1-M_{b_t}^-)^q]=E[(M_{b_t}^+)^q]$.
\end{proof}

We can now prove Theorem \ref{quenched_moment} by treating normal
fluctuations as in Section \ref{proof_amoment}.  The quenched central
limit theorem stated by Theorem \ref{CLTq} and Lemma \ref{SupX} ensure
us that there exists $\Omega_o\in\mathscr{F}$ with $\prob[\Omega_o]=1$
such that, for every $\omega\in\Omega_o$, the scaled displacement
$X_t^\omega/\sqrt{\mu t}$ converges in distribution to a standard Gaussian
variable as $t$ is sent to infinity and
$\sup_{t\in\Rl_{>0}}\{E[|X_t^\omega/\sqrt{\mu t}|^q]\}<+\infty$ for
all $q>0$. Fix $\omega\in\Omega_o$, $q>0$, and a real number $L\ge 1$.
The same arguments that led to (\ref{normalfl}) show that
\begin{equation*}
  \lim_{t\uparrow+\infty}\frac{E[|X_t^\omega|^q\mathds{1}_{\{|X_t|\le L\sqrt{\mu t}\}}]}{t^{\frac{q}{2}}}
 =\sqrt{\frac{2\mu^q}{\pi}}\int_0^L \xi^q e^{-\frac{1}{2}\xi^2} d\xi.
\end{equation*}
On the other hand, the bound $\mathds{1}_{\{|X_t^\omega|>L\sqrt{\mu
    t}\}}\le |X_t^\omega|/(L\sqrt{\mu t})$ gives for all $t>0$
\begin{equation*}
  E\Big[|X_t^\omega|^q\mathds{1}_{\{|X_t^\omega|>L\sqrt{\mu t}\}}\Big]\le C\frac{(\mu t)^{\frac{q}{2}}}{L}
\end{equation*}
with $C:=\sup_{t\in\Rl_{>0}}\{E[|X_t^\omega/\sqrt{\mu t}|^{q+1}]\}<+\infty$.
This way, as in Section \ref{sec:qsmall}, the arbitrariness of $L$
results in
\begin{equation*}
  \lim_{t\uparrow+\infty}\frac{E[|X_t^\omega|^q]}{t^{\frac{q}{2}}}
 =\sqrt{\frac{2\mu^q}{\pi}}\int_0^{+\infty} \xi^q e^{-\frac{1}{2}\xi^2} d\xi=g_q.
\end{equation*}

%%%%%%%%%%%%%%%%%%%%%%%%%%%%%%%%%%%%%%%%%%%%%%
%% Single Appendix:                         %%
%%%%%%%%%%%%%%%%%%%%%%%%%%%%%%%%%%%%%%%%%%%%%%
%\begin{appendix}
%\section*{???}%% if no title is needed, leave empty \section*{}.
%\end{appendix}
%%%%%%%%%%%%%%%%%%%%%%%%%%%%%%%%%%%%%%%%%%%%%%
%% Multiple Appendixes:                     %%
%%%%%%%%%%%%%%%%%%%%%%%%%%%%%%%%%%%%%%%%%%%%%%
\begin{appendix}

\section{Proof of Lemma \ref{prop_funalpha}}
\label{proof:prop_funalpha}

Let us verify that $0<x^\alpha f_\alpha(x)\le 1$ for all
$x\in(0,1)$. Positivity of $f_\alpha(x)$ is explained by the fact that
for any $x\in(0,1)$ and $l<(1-x)/(2x)$ we have
\begin{equation*}
  \frac{2l+2}{1+x}=\frac{2l}{1-x}+2\frac{1-x-2lx}{1-x^2}>\frac{2l}{1-x}.
\end{equation*}
The bound from above is deduced as follows. Pick $x\in(0,1)$ and set
$\lambda:=2x/(1+x)\in(0,1)$. Since $\alpha\ge 1$, we find by convexity
\begin{equation*}
  x^\alpha\bigg(\frac{2l+2}{1+x}\bigg)^\alpha=\bigg[\lambda+(1-\lambda)x\frac{2l}{1-x}\bigg]^\alpha\le\lambda+(1-\lambda)x^\alpha\bigg(\frac{2l}{1-x}\bigg)^\alpha
  \le\lambda+x^\alpha\bigg(\frac{2l}{1-x}\bigg)^\alpha
\end{equation*}
for all $l$. It follows that
\begin{equation*}
  x^\alpha f_\alpha(x)=\sum_{l=0}^{+\infty}\mathds{1}_{\big\{l<\frac{1-x}{2x}\big\}}\Bigg[x^\alpha\bigg(\frac{2l+2}{1+x}\bigg)^\alpha
    -x^\alpha\bigg(\frac{2l}{1-x}\bigg)^\alpha\Bigg]\le\lambda\sum_{l=0}^{+\infty}\mathds{1}_{\big\{l<\frac{1-x}{2x}\big\}}\le\lambda\frac{1+x}{2x}=1.
\end{equation*}

The limit $\lim_{x\downarrow 0}(\alpha+1)\,x^\alpha f_\alpha(x)=1$
needs some more effort to be proved. For every $x\in(0,1/2]$ we have
\begin{equation}
  1-\alpha x\le(1+x)^{-\alpha}\le 1-\alpha x+\alpha(\alpha+1)x^2
\label{fa_lemma_1}
\end{equation}
and
\begin{equation}
1+\alpha x\le (1-x)^{-\alpha}\le 1+\alpha x+\alpha(\alpha+1)2^{\alpha+1}x^2.
\label{fa_lemma_2}
\end{equation}
Moreover, the inequality
$[l^{\alpha+1}-(l-1)^{\alpha+1}]/(\alpha+1)\le l^\alpha\le
[(l+1)^{\alpha+1}-l^{\alpha+1}]/(\alpha+1)$ valid for $l\ge 1$ gives
for each $x\in(0,1/3]$
\begin{equation}
  \frac{1}{\alpha+1}\bigg(\frac{1-3x}{2x}\bigg)^{\alpha+1}\le\sum_{l=1}^{+\infty}\mathds{1}_{\big\{l<\frac{1-x}{2x}\big\}}l^\alpha
  \le\sum_{l=1}^{+\infty}\mathds{1}_{\big\{l<\frac{1+x}{2x}\big\}}l^\alpha\le\frac{1}{\alpha+1}\bigg(\frac{1+3x}{2x}\bigg)^{\alpha+1}.
\label{fa_lemma_3}
\end{equation}
By appealing to (\ref{fa_lemma_1}), (\ref{fa_lemma_2}), and
(\ref{fa_lemma_3}) we can state that for $x\in(0,1/3]$
\begin{align}
  \nonumber
  f_\alpha(x)&=\frac{2^\alpha}{(1+x)^\alpha}\sum_{l=1}^{+\infty}\mathds{1}_{\big\{l<\frac{1+x}{2x}\big\}}l^\alpha
  -\frac{2^\alpha}{(1-x)^\alpha}\sum_{l=1}^{+\infty}\mathds{1}_{\big\{l<\frac{1-x}{2x}\big\}}l^\alpha\\
  \nonumber
  &\le2^\alpha\Big[1-\alpha x+\alpha(\alpha+1)x^2\Big]\sum_{l=1}^{+\infty}\mathds{1}_{\big\{l<\frac{1+x}{2x}\big\}}l^\alpha
  -2^\alpha(1+\alpha x)\sum_{l=1}^{+\infty}\mathds{1}_{\big\{l<\frac{1-x}{2x}\big\}}l^\alpha\\
  \nonumber
  &\le 2^\alpha\sum_{l=1}^{+\infty}\mathds{1}_{\big\{\frac{1-x}{2x}\le l<\frac{1+x}{2x}\big\}}l^\alpha
  -\alpha 2^{\alpha+1} x\sum_{l=1}^{+\infty}\mathds{1}_{\big\{l<\frac{1-x}{2x}\big\}}l^\alpha+
  \alpha(\alpha+1) 2^\alpha x^2\sum_{l=1}^{+\infty}\mathds{1}_{\big\{l<\frac{1+x}{2x}\big\}}l^\alpha\\
  &\le \bigg(\frac{1+x}{x}\bigg)^\alpha-\frac{\alpha x}{\alpha+1}\bigg(\frac{1-3x}{x}\bigg)^{\alpha+1}+
  \alpha(\alpha+1)x^2\bigg(\frac{1+x}{x}\bigg)^{\alpha+1}
  \label{fa_lemma_4}
\end{align}
and
\begin{align}
  \nonumber
  f_\alpha(x)&\ge2^\alpha\big(1-\alpha x\big)\sum_{l=1}^{+\infty}\mathds{1}_{\big\{l<\frac{1+x}{2x}\big\}}l^\alpha
  -2^\alpha\Big[1+\alpha x+\alpha(\alpha+1)2^{\alpha+1}x^2\Big]\sum_{l=1}^{+\infty}\mathds{1}_{\big\{l<\frac{1-x}{2x}\big\}}l^\alpha\\
  \nonumber
  &\ge 2^\alpha\sum_{l=1}^{+\infty}\mathds{1}_{\big\{\frac{1-x}{2x}\le l<\frac{1+x}{2x}\big\}}l^\alpha
  -\alpha 2^{\alpha+1} x\sum_{l=1}^{+\infty}\mathds{1}_{\big\{l<\frac{1+x}{2x}\big\}}l^\alpha
  -\alpha(\alpha+1) 2^{2\alpha+1}x^2\sum_{l=1}^{+\infty}\mathds{1}_{\big\{l<\frac{1-x}{2x}\big\}}l^\alpha\\
  &\ge \bigg(\frac{1-x}{x}\bigg)^\alpha
  -\frac{\alpha x}{\alpha+1}\bigg(\frac{1+3x}{x}\bigg)^{\alpha+1}
  -\alpha(\alpha+1) 2^\alpha x^2\bigg(\frac{1-x}{x}\bigg)^{\alpha+1}.
  \label{fa_lemma_5}
\end{align}
Bounds (\ref{fa_lemma_4}) and (\ref{fa_lemma_5}) give $\lim_{x\downarrow 0}(\alpha+1)\,x^\alpha f_\alpha(x)=1$.

\section{Proof of Proposition \ref{dist:Mn}}
\label{proof:distMn}

Part $(i)$ is an immediate consequence of the fact that
$(-V_1,\ldots,-V_n)$ is distributed as $(V_1,\ldots,V_n)$ for each
$n\ge 1$, which entails that $(1-R_1,\ldots,1-R_n)$ is distributed as
$(R_1,\ldots,R_n)$. Part $(ii)$ can be verified by induction. To this
aim, we first notice that by definition $P[M_1^+=0]=P[M_1^+=1]=1/2$.
Then, by appealing to the fact that $(R_2-V_1,\ldots,R_n-V_1)$ is
independent of $V_1$ and distributed as $(R_1,\ldots,R_{n-1})$, we
observe that for each $n\ge 2$ and $m\ge 0$
  \begin{align}
    \nonumber
    P\big[M_n^+=m\big]&=\sum_{v\in\mathbb{Z}_2}P\Big[M_n^+=m,V_1=v\Big]\\
    \nonumber
    &=\sum_{v\in\mathbb{Z}_2}P\bigg[\max\Big\{(1+v)/2,R_2-V_1+v\ldots,R_n-V_1+v\Big\}=m,V_1=v\bigg]\\
    \nonumber
    &=\sum_{v\in\mathbb{Z}_2}P\bigg[\max\Big\{(1+v)/2,R_1+v\ldots,R_{n-1}+v\Big\}=m\bigg]P\big[V_1=v\big]\\
    \nonumber
    &=\frac{1}{2}\sum_{v\in\mathbb{Z}_2}P\Big[\max\big\{(1-v)/2,M_{n-1}^+\big\}=m-v\Big]\\
    \nonumber
    &=\frac{1}{2}P\big[M_{n-1}^+=0\vee(m-1)\big]+\frac{1}{2}P\big[M_{n-1}^+=m+1\big].
    \qedhere
  \end{align}

\section{Proof of Lemma \ref{tail:Mn}}
\label{proof:tailMn}

Pick $n\ge 1$ and $L>0$. Assume that $L\le n$, otherwise there is
nothing to prove since $P[M_n^+>L]=0$ when $L>n$. It is a simple
exercise of calculus to verify that
$1\le(1-z)^{1-z}(1+z)^{1+z}\,e^{-z^2}$ for $z\in[0,1]$ with the
convention $0^0:=1$. It follows that
$1\le(1-z)^{1-z}(1+z)^{1+z}\,e^{-z^2}\le(1-z)^{2-\zeta}(1+z)^{\zeta}\,e^{-z^2}$
if $z\in[0,1]$ and $\zeta>1+z$. Thus, by setting $z=L/n$ and
$\zeta=2l/n$, $l$ being a given integer, and by taking the power $n/2$
we realize that
\begin{equation}
1\le\bigg(\frac{n-L}{n}\bigg)^{n-l} \bigg(\frac{n+L}{n}\bigg)^l e^{-\frac{L^2}{2n}}
\label{tail:Mn1}
\end{equation}
whenever $2l>n+L$. Combining part $(ii)$ of Proposition \ref{dist:Mn}
with (\ref{tail:Mn1}) we get
\begin{align}
\nonumber
P\big[M_n^+>L\big]&=\frac{1}{2^n}\sum_{m=0}^n\mathds{1}_{\{m>L\}}{n \choose \lfloor\frac{n+m+1}{2}\rfloor}
\le\frac{1}{2^{n-1}}\sum_{l=0}^n\mathds{1}_{\{2l>n+L\}}{n \choose l}\\
\nonumber
&\le\frac{1}{2^{n-1}}\sum_{l=0}^n{n \choose l}\bigg(\frac{n-L}{n}\bigg)^{n-l}\bigg(\frac{n+L}{n}\bigg)^l e^{-\frac{L^2}{2n}}=2e^{-\frac{L^2}{2n}}.
\end{align}

\section{Proof of Lemma \ref{moment:Mn}}
\label{proof:momentMn}

Pick $n\ge 1$ and $q>0$ and set $\lambda:=\sqrt{q/(en)}$. Part $(ii)$
of Proposition \ref{dist:Mn} gives
\begin{equation*}
  E\big[e^{\lambda M_n^+}\big]=\frac{1}{2^n}\sum_{m=0}^n e^{\lambda m}{n \choose \lfloor\frac{n+m+1}{2}\rfloor}
  \le\frac{1}{2^{n-1}}\sum_{l=0}^n e^{\lambda(2l-n)}{n \choose l}=2\cosh^n\lambda\le 2e^{\frac{q}{2}}.
\end{equation*}
The last inequality follows from the bound $\cosh\lambda\le
1+(\lambda^2/2)e^{\lambda}\le 1+e\lambda^2/2\le e^{e\lambda^2/2}$ if
$\lambda\le 1$ and $\cosh\lambda\le e^{\lambda}\le e^{e\lambda^2/2}$
if $\lambda>1$.  Then, by exploiting the fact that $z^q\le q^qe^{z-q}$
for all $z\ge 0$ we find
\begin{equation*}
  E\big[(M_n^+)^q\big]=\frac{E[(\lambda M_n^+)^q]}{\lambda^q}\le\frac{E\big[q^qe^{\lambda M_n^+-q}\big]}{\lambda^q}
  \le \frac{2q^qe^{-\frac{q}{2}}}{\lambda^q}=2(qn)^{\frac{q}{2}}.
\qedhere
\end{equation*}

\section{Proof of Proposition \ref{Udist}}
\label{proof:Udist}

Part $(i)$ follows from the fact that $(1-R_1,\ldots,1-R_n)$ is
distributed as $(R_1,\ldots,R_n)$ for every $n\ge 1$ and is immediate.
Let us discuss part $(ii)$.  Since $(R_2-V_1,\ldots,R_n-V_1)$ is
independent of $V_1$ and distributed as $(R_1,\ldots,R_{n-1})$ we have
for any $n\ge 2$, $r\ge 1$, and $u\ge 1$
\begin{align}
\nonumber
P\big[U_{n,r}=u\big]&=\sum_{v\in\mathbb{Z}_2}P\Big[U_{n,r}=u,V_1=v\Big]\\
\nonumber
&=\sum_{v\in\mathbb{Z}_2}P\Bigg[\mathds{1}_{\{(1+v)/2=r\}}+\sum_{k=2}^n\mathds{1}_{\{R_k-V_1=r-v\}}=u,V_1=v\Bigg]\\
\nonumber
&=\sum_{v\in\mathbb{Z}_2}P\Bigg[\mathds{1}_{\{(1+v)/2=r\}}+\sum_{k=1}^{n-1}\mathds{1}_{\{R_k=r-v\}}=u\Bigg]P\big[V_1=v\big]\\
\nonumber
&=\frac{1}{2}\sum_{v\in\mathbb{Z}_2}P\Big[\mathds{1}_{\{(1+v)/2=r\}}+U_{n-1,r-v}=u\Big]\\
\nonumber
&=\frac{1}{2}P\big[U_{n-1,r-1}=u-\mathds{1}_{\{r=1\}}\big]+\frac{1}{2}P\big[U_{n-1,r+1}=u\big].
\end{align}
As $U_{n,0}$ is distributed as $U_{n,1}$ by part $(i)$, this
equality shows that for all $n\ge 2$, $r\ge 1$, and $u\ge 1$
\begin{align}
  \nonumber
  P\big[U_{n,r}=u\big]&=\frac{1}{2}\begin{cases}
  P[U_{n-1,1}=u-1]+P[U_{n-1,2}=u] & \mbox{if }r=1,\\
  P[U_{n-1,r-1}=u]+P[U_{n-1,r+1}=u] & \mbox{if }r>1.
  \end{cases}\\
  \nonumber
  &=\frac{1}{2}\begin{cases}
  1-\sum_{k=1}^{+\infty} P[U_{n-1,1}=k]+P[U_{n-1,2}=1] & \mbox{if $r=1$ and $u=1$},\\
  P[U_{n-1,1}=u-1]+P[U_{n-1,2}=u] & \mbox{if $r=1$ and $u>1$},\\
  P[U_{n-1,r-1}=u]+P[U_{n-1,r+1}=u] & \mbox{if }r>1.
  \end{cases}
\end{align}
Part $(ii)$ can then be verified by induction through simple algebra
starting from the fact that $P[U_{1,1}=1]=1/2$ and $P[U_{1,r}=u]=0$ if
$r\ge 1$ and $u\ge 1$ are not simultaneously equal to $1$.

\section{Proof of Lemma \ref{boundU}}
\label{proof:boundU}

Fix $n\ge 1$ and $L>0$.  There is nothing to prove if $L>n$ as
$P[\max_{r\in\mathbb{Z}}\{U_{n,r}\}>L]=0$ in such a case. Therefore,
assume that $L\le n$. The condition
$\max_{r\in\mathbb{Z}}\{U_{n,r}\}>L$ implies that there exists at
least one $r\in\mathbb{Z}$ such that $U_{n,r}>L$. Then
\begin{equation*}
P\Big[\max_{r\in\mathbb{Z}}\{U_{n,r}\}>L\Big]\le\sum_{r=-\infty}^{+\infty} P\big[U_{n,r}>L\big].
\end{equation*}
At this point, by invoking part $(i)$ of Proposition \ref{Udist} first
and part $(ii)$ later, we can write down the bound
\begin{align}
\nonumber
P\Big[\max_{r\in\mathbb{Z}}\{U_{n,r}\}>L\Big]&\le2\sum_{r=1}^{+\infty} P\big[U_{n,r}>L\big]
=\frac{1}{2^{n-1}}\sum_{r=1}^{+\infty} \sum_{u=1}^{+\infty} \mathds{1}_{\{u>L\}}{n\choose\lfloor\frac{n+r+u}{2}\rfloor}\\
\nonumber
&\le\frac{1}{2^{n-2}}\sum_{r=1}^n \sum_{l=0}^n \mathds{1}_{\{2l\ge n+r+L\}}{n\choose l}
\le\frac{n}{2^{n-2}}\sum_{l=0}^n \mathds{1}_{\{2l>n+L\}}{n\choose l}\le 4n e^{-\frac{L^2}{2n}}.
\end{align}
The last inequality is due to (\ref{tail:Mn1}) as in the proof of Lemma
\ref{tail:Mn}.

\section{Proof of Lemma \ref{Ecal_estimate}}
\label{proof:Ecal_estimate}

For each $n\ge 1$, the condition $\min\{R_{n+1},\ldots,R_{n\vee
  p+1}\}>R_n$ implies $R_{n+1}>R_n$, which is feasible only if
$V_n=V_{n+1}=1$. The latter yields $R_n=S_n$. Then, for all integers
$0\le p\le q$ and $u>0$ we have
  \begin{align}
    \nonumber
    \mathcal{E}_u\{p,q\}&:=\sum_{n=1}^q\,P\Big[U_{n,R_n}=u,\,\min\big\{R_{n+1},\ldots,R_{n\vee p+1}\big\}>R_n\Big]\\
    \nonumber
    &\le \sum_{n=1}^q\,P\Big[U_{n,S_n}=u,\,\min\big\{R_{n+1},\ldots,R_{n\vee p+1}\big\}>S_n,\,V_{n+1}=1\Big]\\
    \nonumber
    &\le\sum_{n=1}^p\,P\Big[U_{n,S_n}=u,\,\min\big\{R_{n+1}-S_n,\ldots,R_{p+1}-S_n\big\}>0\Big]+\sum_{n=p+1}^q\,P\Big[U_{n,S_n}=u,\,V_{n+1}=1\Big].
\end{align}
On the other hand, $\min\{R_{n+1}-S_n,\ldots,R_{p+1}-S_n\}$ with $n<p$
is independent of $V_1,\ldots,V_n$ and distributed as
$\min\{R_1,\ldots,R_{p-n+1}\}=:M_{p-n+1}^-$. Thus, since $U_{n,S_n}$
is a measurable function of $V_1,\ldots,V_n$, we find
\begin{equation*}
\mathcal{E}_u\{p,q\}\le\sum_{n=1}^p P\big[U_{n,S_n}=u\big] P\big[M_{p-n+1}^->0\big]+\frac{1}{2}\sum_{n=p+1}^q P\big[U_{n,S_n}=u\big].
\end{equation*}
A further step is taken by observing that $U_{n,S_n}$ is distributed
as $U_{n,1}$, which is an immediate consequence of the fact that
$(V_n,\ldots,V_1)$ is distributed as $(V_1,\ldots,V_n)$.  We also
recall that $M_{p-n+1}^-$ is distributed as $1-M_{p-n+1}^+$ by
Proposition \ref{dist:Mn} and that $M_{p-n+1}^+\ge 0$. Then
\begin{equation}
\mathcal{E}_u\{p,q\}\le\sum_{n=1}^p P\big[U_{n,1}=u\big] P\big[M_{p-n+1}^+=0\big]+\frac{1}{2}\sum_{n=p+1}^q P\big[U_{n,1}=u\big].
\label{Eqpu_2}
\end{equation}
The last step makes use of the explicit distributions of $U_{n,1}$ and
$M_{p-n+1}^+$ provided by Propositions \ref{Udist} and \ref{dist:Mn},
respectively. We point out that the hypothesis $u>0$ is important
here. Bear in mind that the bounds $\sqrt{2\pi l}\,l^le^{-l}\le
l!\le\sqrt{2\pi l}\,l^le^{-l+\frac{1}{12 l}}$ valid for all positive
integers $l$ \cite{Stirling} give for every $n\ge 1$ and $u>0$
\begin{equation*}
  \frac{1}{2^n}{n\choose\lfloor\frac{n+u+1}{2}\rfloor}\le \frac{1}{2^n}{n\choose\lfloor\frac{n+1}{2}\rfloor}=
  \frac{1}{2^n}{n\choose\lfloor\frac{n}{2}\rfloor}\le
  \sqrt{\frac{2e^{\frac{1}{6n}}}{\pi n}}\le\frac{1}{\sqrt{n}}.
\end{equation*}
Thus, by combining (\ref{Eqpu_2}) with Propositions \ref{Udist} and
\ref{dist:Mn} we reach the result
\begin{align}
  \nonumber
  \mathcal{E}_u\{p,q\}&\le\sum_{n=1}^p \frac{1}{2^n}{n \choose \lfloor\frac{n+u+1}{2}\rfloor}
  \frac{1}{2^{p-n+1}}{p-n+1 \choose \lfloor\frac{p-n+2}{2}\rfloor}
  +\sum_{n=p+1}^q\frac{1}{2^{n+1}}{n \choose \lfloor\frac{n+u+1}{2}\rfloor}\\
  \nonumber
  &\le\sum_{n=1}^p\frac{1}{\sqrt{n(p-n+1)}}+\sum_{n=p+1}^q\frac{1}{2\sqrt{n}}=
  \frac{1}{\sqrt{p+1}}\sum_{n=1}^p\sqrt{\frac{1}{n}+\frac{1}{p-n+1}}+\sum_{n=p+1}^q\frac{1}{2\sqrt{n}}\\
  \nonumber
  &\le\frac{1}{\sqrt{p}}\sum_{n=1}^p\bigg(\frac{1}{\sqrt{n}}+\frac{1}{\sqrt{p-n+1}}\bigg)+\sum_{n=p+1}^q\frac{1}{2\sqrt{n}}
  \le 4+\sqrt{q}-\sqrt{p},
\end{align}
where we have used the fact that $1/\sqrt{n}\le 2\sqrt{n}-2\sqrt{n-1}$
for each integer $n\ge 1$ to get the last inequality. This bound
proves the lemma as $-\sqrt{p}\le 1-\sqrt{p+1}$ for each $p\ge 0$.

\section{Proof of Lemma \ref{sommeUv}}
\label{proof:sommeUv}

Fix $n\ge 1$.  To begin with, we observe that
$U_{n,R_n}-\mathds{1}_{\{V_n=-1\}}\ge 0$ is odd if $R_n\ge 1$ for
geometric reasons. In fact, for a random walk starting from the
origin, $U_{n,R_n}$ counts the passages up to jump $n$ over the
current edge $R_n$, which is traveled with velocity direction
$V_n$. On the other hand,
$U_{n,R_n}-\mathds{1}_{\{V_n=-1\}}=U_{n-1,R_n}+\mathds{1}_{\{V_n=1\}}$,
so that $\mathcal{P}_{u,v}\{n\}=0$ unless $u+\mathds{1}_{\{v=1\}}$ is
odd. This proves part of the lemma, and to complete the proof it is
sufficient to show that
\begin{equation}
\mathcal{P}_{u,-1}\{n\}+\mathcal{P}_{u,1}\{n\}=\frac{1}{2^n}{n-1\choose\lfloor\frac{n+u}{2}\rfloor}.
\label{Pcalnuv0}
\end{equation}

By writing $R_n$ as $S_{n-1}+\mathds{1}_{\{V_n=1\}}$ and by appealing
to the independence of the velocities, from the definition
(\ref{Pcalnuv}) we find
  \begin{equation}
   \mathcal{P}_{u,-1}\{n\}=\frac{1}{2}P\Big[U_{n-1,S_{n-1}}=u,S_{n-1}\ge 1\Big]
   \label{Pcalnuv1}
  \end{equation}
  and
  \begin{equation}
   \mathcal{P}_{u,1}\{n\}=\frac{1}{2}P\Big[U_{n-1,S_{n-1}+1}=u,S_{n-1}\ge 0\Big].
   \label{Pcalnuv2}
  \end{equation}  
  On the other hand, $(U_{n-1,S_{n-1}+r},S_{n-1})$ is distributed as
  $(U_{n-1,1-r},S_{n-1})$ for every $r\in\mathbb{Z}$ due to the fact
  that $(V_{n-1},\ldots,V_1)$ is distributed as
  $(V_1,\ldots,V_{n-1})$. Moreover, $(U_{n-1,0},S_{n-1})$ is
  distributed as $(U_{n-1,1},-S_{n-1})$ since $(V_1,\ldots,V_{n-1})$ is
  distributed as $(-V_1,\ldots,-V_{n-1})$.  It follows from
  (\ref{Pcalnuv1}) and (\ref{Pcalnuv2}) that
\begin{equation*}
\mathcal{P}_{u,-1}\{n\}=\frac{1}{2}P\Big[U_{n-1,1}=u,S_{n-1}\ge 1\Big]
\end{equation*}
and
\begin{equation*}
\mathcal{P}_{u,1}\{n\}=\frac{1}{2}P\Big[U_{n-1,0}=u,S_{n-1}\ge 0\Big]=\frac{1}{2}P\Big[U_{n-1,1}=u,S_{n-1}\le 0\Big],
\end{equation*}
so that
\begin{equation*}
\mathcal{P}_{u,-1}\{n\}+\mathcal{P}_{u,1}\{n\}=\frac{1}{2}P\big[U_{n-1,1}=u\big].
\end{equation*}
This identity proves (\ref{Pcalnuv0}) thanks to Proposition
\ref{Udist}, which gives
$P[U_{n-1,1}=u]=\frac{1}{2^{n-1}}{n-1\choose\lfloor\frac{n+u}{2}\rfloor}$
for all $u\ge 0$.

\end{appendix}

%%%%%%%%%%%%%%%%%%%%%%%%%%%%%%%%%%%%%%%%%%%%%%
%% Support information, if any,             %%
%% should be provided in the                %%
%% Acknowledgements section.                %%
%%%%%%%%%%%%%%%%%%%%%%%%%%%%%%%%%%%%%%%%%%%%%%
\begin{acks}[Acknowledgments]
The author is grateful to Frank den Hollander and Giambattista
Giacomin for critical reading of the manuscript and valuable
comments. The author is also grateful to the two anonymous referees
for careful reading of the manuscript, and for useful suggestions that
have allowed to improve the text and clarify some of the results.

\end{acks}
%%%%%%%%%%%%%%%%%%%%%%%%%%%%%%%%%%%%%%%%%%%%%%
%% Funding information, if any,             %%
%% should be provided in the                %%
%% funding section.                         %%
%%%%%%%%%%%%%%%%%%%%%%%%%%%%%%%%%%%%%%%%%%%%%%
%\begin{funding}
% The first author was supported by ...
%
% The second author was supported in part by ...
%\end{funding}

%%  imsart-number.bst  will be used to                     %%
%%  MR numbers will be added by VTeX.                      %%
%%                                                         %%
%%  Use \cite{...} to cite references in text.             %%
%%                                                         %%
%%%%%%%%%%%%%%%%%%%%%%%%%%%%%%%%%%%%%%%%%%%%%%%%%%%%%%%%%%%%%

%% if your bibliography is in bibtex format, uncomment commands:
%\bibliographystyle{imsart-number} % Style BST file
%\bibliography{bibliography}       % Bibliography file (usually '*.bib')

\end{document}